\documentclass[12pt]{amsart}
\usepackage{amssymb}
\usepackage{amscd}
\usepackage{amsmath,mathtools}
\usepackage[all]{xy}
\usepackage{enumitem}
\usepackage{bm}
\usepackage[T1]{fontenc} 
\usepackage{thmtools} 
\usepackage{url}
\usepackage[hidelinks]{hyperref}
\usepackage{tikz}
\usepackage{graphicx}

\newif\ifdebug                                                      %
\debugfalse

\DeclareFontFamily{U}{MnSymbolC}{}
\DeclareSymbolFont{MnSyC}{U}{MnSymbolC}{m}{n}
\DeclareFontShape{U}{MnSymbolC}{m}{n}{
    <-6>  MnSymbolC5
   <6-7>  MnSymbolC6
   <7-8>  MnSymbolC7
   <8-9>  MnSymbolC8
   <9-10> MnSymbolC9
  <10-12> MnSymbolC10
  <12->   MnSymbolC12}{}
\DeclareMathSymbol{\contract}{\mathbin}{MnSyC}{'270}

\newif\ifinfootnote
\infootnotefalse

\let\footnoteasusual\footnote
\renewcommand{\footnote}[1]
{\infootnotetrue\footnoteasusual{#1}\infootnotefalse}

\newcommand    {\ynote}[1]   {\ifdebug      {{ \scriptsize{#1} }} \else \fi}

\newcommand{\printname}[1]
{\ifmmode{ \smash{ \raisebox{5pt}{\text{\tiny{#1}}} } }
 \else   {\ifinfootnote \smash{\raisebox{0pt}{\tiny{#1}}}
             \else { \marginpar{
                     \smash{ \makebox[0pt]{\raisebox{-6pt}{\tiny{#1}} } }
} } \fi}
 \fi}

\newcommand{\labell}[1]
{\ifdebug {\label{#1}\printname{{#1}}} \else {\label{#1}} \fi}

\swapnumbers

\theoremstyle{plain}

\numberwithin{equation}{section}

\newtheorem {Lemma}[equation]{Lemma}

\newtheorem {Corollary}[equation]{Corollary}
\newtheorem* {Corollary*}{Corollary}
\newtheorem{Proposition}[equation]{Proposition}
\newtheorem*{Proposition*}{Proposition} 
 
\newtheorem*{Exercise*}{Exercise} 

\theoremstyle{definition}

\newtheorem{Construction}[equation]{Construction}
\newtheorem*{Construction*}{Construction}

\theoremstyle{remark}
\newtheorem{Remark}[equation]{Remark}
\newtheorem*{Remark*}{Remark}
\newtheorem{Example}[equation]{Example}

\def\eor{\unskip\ \hglue0mm\hfill$\diamond$\smallskip\goodbreak}
\def\eoe{\unskip\ \hglue0mm\hfill$\between$\smallskip\goodbreak}

\setlist{topsep=0pt,itemsep=6pt}

\makeatletter
\newcommand\dblmathcircled[1]{\mathpalette\@dblmathcircled{#1}} 
\newcommand\@dblmathcircled[2]{\tikz[baseline=(math.base)]
\node[draw,circle,double,inner sep=1pt] (math) {$\m@th#1#2$};}
\makeatother

\def \oodot {\dblmathcircled{\cdot}}

\def \C {{\mathbb C}}
\def \R {{\mathbb R}}
\def \Z {{\mathbb Z}}
\def \N {{\mathbb N}}

\def \Rplus {\R_{\geq 0}}
\def \Rpos {\R_{> 0}}

\def \cut {{\operatorname{cut}}}
\def \Mcut {M_\cut}
\def \Ncut {N_\cut}
\def \Wcut {W_\cut}
\def \Vcut {V_\cut}
\def \Ecut {E_\cut}

\def \red {{\text{red}}}
\def \Mred {M_\red}
\def \Nred {N_\red}
\def \Ered {E_\red}

\def \ol {\overline}

\def \del {\partial}
\def \ssminus {\smallsetminus}
\def \wh {\widehat}
\def \wt {\tilde}
\DeclareMathOperator \supp {supp}
\DeclareMathOperator \Id {Id}

\def \intM {\mathring{M}}

\def \intN {\mathring{N}}

\def \intY {\mathring{Y}}

\def \calF {{\mathcal F}}
\def \calO {{\mathcal O}}
\def \calS {{\mathcal S}}

\def \frakU {{\mathfrak U}}

\def \eps {\epsilon}

\def \bfc {{\mathbf{c}}}

\subjclass[2020]{Primary 53D20, Secondary 57R55}

\begin{document}

\title[Functoriality for cutting \ynote{ \today}]
{Functoriality for 
symplectic and contact cutting, and equivariant radial-squared blowups
\ynote{ \today}}
%

\author[Yael karshon \ynote{ \today}]{Yael Karshon}
\address{Dept.\ of Mathematics, University of Toronto, 40 St.~George Street,
6th floor, Toronto ON M5S 2E4, Canada}

\begin{abstract}
We exhibit Lerman's cutting procedure 
%
%
as a functor from the category of manifolds-with-boundary 
equipped with free circle actions near the boundary, 
with so-called equivariant transverse maps, 
to the category of manifolds and smooth maps.
We then apply the cutting procedure 
to differential forms that are not necessarily symplectic, 
to distributions that are not necessarily contact, and to submanifolds.
We obtain an inverse functor
from so-called equivariant radial-squared blowup.
%
%
\end{abstract}

\maketitle

\setcounter{tocdepth}{1}
\tableofcontents

\section{Introduction}
\labell{sec:intro}

\subsection*{The cutting construction}\ 

%
%

Lerman's symplectic cutting procedure~\cite{lerman:symplectic cuts},
and his related contact cutting procedure~\cite{lerman:contact cuts},
have had wide applications in symplectic and contact geometry.
The symplectic cutting procedure takes a symplectic manifold $M$
with a Hamiltonian circle action and a momentum map $\mu \colon M \to \R$,
passes to a regular sub-level set $\{ \mu \geq \lambda \}$,
which is a manifold-with-boundary, and collapses its boundary 
$\{ \mu = \lambda \}$ along the circle action.
Identifying the resulting space with the symplectic reduction
of the diagonal circle action on $M \times \C$
makes it into a symplectic manifold.
The name ``cutting'' is because 
taking the two sub-level sets $\{ \mu \geq \lambda\}$
and $\{ \mu \leq \lambda \}$ and collapsing their boundaries
can be considered as a ``cutting'' of $M$ into two pieces. 
If $M$ is a symplectic toric manifold 
and the circle action is a sub-circle of the toric action,
then---up to equivariant symplectomorphism---this construction amounts to 
decomposing the momentum polytope~$\Delta$ of $M$ into two convex polytopes
by slicing it along a hyperplane,
and taking the symplectic toric manifolds that correspond to the two
convex polytopes.

\begin{center}
\begin{figure}[h]
\includegraphics[scale=.75]{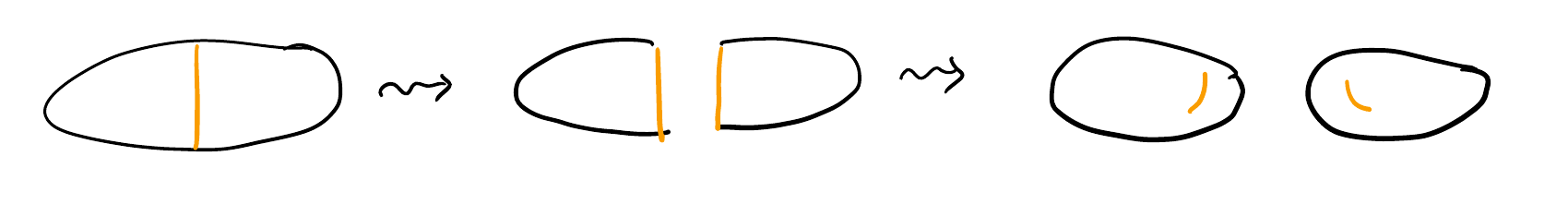}
\caption{Symplectic cutting}
\end{figure}
\end{center}

As noted by Lerman~\cite[Section 2]{lerman:contact cuts},
one can carry out this construction without a symplectic form,
and one can work with one-half of this construction,
beginning with a manifold-with-boundary
and a circle action near the boundary\footnote{``Near the boundary'' 
means ``on some neighbourhood of the boundary''.}.
In this situation, the name ``cutting'' might be less appropriate;
perhaps we can call it ``collapsing'' (because we collapse the boundary
along the circle action) or ``closing'' (because we create a manifold
without boundary), or perhaps ``sewing'' or ``stitching'';
though changing the name would hide the connection
with classical symplectic cutting.
%
Until this paper appears in print, I would like to hear your opinion 
on the terminology; drop me an email at \url{karshon@math.toronto.edu} 
and let me know what you think.

This cutting procedure takes a manifold-with-boundary $M$,
equipped with a free action of the circle group $S^1$ 
on a neighbourhood $U_M$ of the boundary,\footnote{Necessarily, 
the boundary $\del M$ is $S^1$-invariant.}
to the quotient of $M$ by the equivalence relation $\sim$
in which distinct points are equivalent if and only if 
they are both in the boundary $\del M$ and are in the same $S^1$ orbit.
We denote this quotient by 
$$\bm{\Mcut} := M/{\sim} $$ 
and the quotient map by 
$$ \bfc \colon M \to \Mcut.$$ 
We equip $\Mcut$ with the quotient topology.

Lerman \cite[Section 2]{lerman:contact cuts} 
begins with a circle action on $\del M$.
To obtain a smooth manifold structure on $\Mcut$, 
he uses the collar neighbourhood theorem to identify 
a neighbourhood of $\del M$ in $M$ with $\del M \times [0,1)$
and applies his symplectic cutting procedure
(without a two-form and with the momentum map replaced by the 
projection to the second factor) to $\del M \times (-1,1)$.
Our line of argument is slightly different.
First, in order to exhibit the cutting construction as a functor,
it is not enough to begin with a circle action on $\del M$;
we need to begin with a circle action near $\del M$.
(The manifold structure on $\Mcut$ does depend on the extension
of the circle action to a neighbourhood of $\del M$;
see Remark~\ref{dependence on action}.)
Second, we characterise the manifold structure on $\Mcut$ 
through its space of real valued smooth functions.
We do this in Construction~\ref{calF M},
followed by Lemma~\ref{FM indep of f} and Theorem~\ref{mfld str}.


(Warning: the map $c \colon M \to \Mcut$ is not smooth!)

As before, let $M$ be a manifold-with-boundary,
with a free circle action on a neighbourhood $U_M$ of the boundary.
An \textbf{invariant boundary defining function}\footnote{
The term ``boundary defining function'' is promoted
in John Lee's texbook \cite[Chapter~5]{JohnLee:smooth}.
} 
on $M$ 
is a smooth function $f \colon M \to \R_{\geq 0}$,
such that $f^{-1}(0) = \del M$ and $df|_{\del M}$ never vanishes,
and such that $f$ is $S^1$-invariant
on some $S^1$-invariant neighbourhood of $\del M$ in $U_M$.

\begin{Remark}
An invariant boundary-defining function always exists: 
take charts with values in $\R^{n-1} \times \R_{\geq 0}$ whose domains 
cover $\del M$;
project to the last coordinate, patch with a partition of unity,
restrict 
to an invariant neighbourhood\footnote{Every 
neighbourhood of $\del M$
         \label{invt nd}
contains an invariant neighbourhood;
see Lemma~\ref{quotient map}\eqref{intersection}.
}
 of $\del M$ in $U_M$, 
and average with respect to the circle action;
use a partition of unity to patch with the constant function on~$\intM$
with value $1$.
\end{Remark}

\begin{Remark}
If we start from a symplectic manifold with a Hamiltonian circle action,
and $M$ is the set $\mu \geq 0$
where $\mu$ is the momentum map and $0$ is a regular value for $\mu$,
then $f := \mu$ is an invariant boundary defining function.
\end{Remark}

\begin{center}
\begin{figure}[h]
\includegraphics[scale=.75]{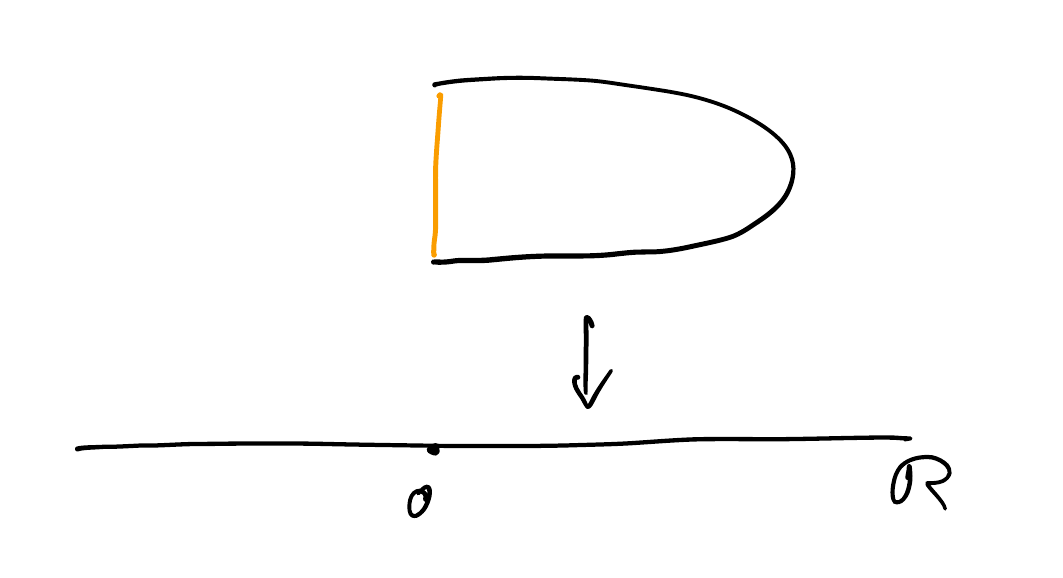}
\caption{Boundary defining function}
\end{figure}
\end{center}

We now describe the \textbf{smooth cutting construction},
which yields a smooth manifold structure on 
the topological space $\Mcut$.
We do this in Construction~\ref{calF M}, Lemma~\ref{FM indep of f},
and Lemma~\ref{mfld str}.
The technical condition in Construction~\ref{calF M}($\calF 2$)
is motivated by a characterization of smooth functions on $\R^2$
in terms of their expressions in polar coordinates;
see Lemma~\ref{cylinder}.

\begin{Construction} \labell{calF M}
%
Let $f \colon M \to \Rplus$ be an invariant boundary defining function.
We then consider the set $\bm{\calF_M}$ 
of those real valued functions $h \colon \Mcut \to \R$
whose composition $\wh{h} := h \circ c$
with the quotient map $c \colon M \to \Mcut$
satisfies the following two conditions.
\begin{enumerate}
\item[($\calF 1$)]
$\wh{h}|_{\intM} \colon \intM \to \R$ is smooth.
\item[($\calF 2$)]
There exists an $S^1$-invariant open neighbourhood $V$ of $\del M$ in~$U_M$
and a smooth function $H \colon V \times \C \to \R$
such that
\begin{enumerate}
\item
$H(a \cdot x , z) = H(x,az)$ for all $a \in S^1$ and $(x,z) \in V \times \C$;
and
\item
$\wh{h}(x) = H\big(x,\sqrt{f(x)} \, \big)$ for all $x \in V$.
\end{enumerate}
\end{enumerate}
\end{Construction}

\begin{Lemma} \labell{FM indep of f}
In the setup of Construction~\ref{calF M},
the set of functions $\calF_M$ is independent
of the choice of invariant boundary defining function $f$.
\end{Lemma}

\begin{proof}
Let $f_o \colon M \to \R_{\geq 0}$
be another\footnote{The subscript $o$ stands for ``other''.}
invariant boundary defining function.
Then $f$ is the product of $f_o$ with a smooth function on $M$
that is everywhere positive.  
Let $g \colon M \to \R_{>0}$ be the square root of this positive function.  
Then $g$ is smooth and is $S^1$-invariant near the boundary,
and $f(x) = g(x)^2 f_o(x)$ for all $x \in M$.

Fix any real valued function $h \colon \Mcut \to \R$.
Suppose that $\wh{h} := h \circ c \colon M \to \R$ 
satisfies Condition ($\calF 2$) of Construction~\ref{calF M},
with an open subset $V$ and a smooth function $H \colon V \times \C \to \R$.
After possibly shrinking $V$, we may assume that $g$ is $S^1$ invariant
on $V$.
Then $\wh{h}$ satisfies Condition ($\calF 2$) of Construction~\ref{calF M}
with $f$ replaced by $f_o$,
with the smooth function $H_o \colon V \times \C \to \R$ 
defined by $H_o(x,z) := H(x,g(x)z)$.

Because the boundary defining function does not appear
in Condition~($\calF 1$) of Construction~\ref{calF M},
and because the function $h$ was arbitrary,
we conclude that
the set of functions $\calF_M$ that is obtained from $f$
through Construction~\ref{calF M}
is contained in the set of functions that is obtained from $f_o$
through Construction~\ref{calF M}.
Flipping the roles of $f$ and $f_o$, we conclude that these two 
sets of functions coincide.
\end{proof}

\begin{restatable} {Theorem}{MFLD}
\labell{mfld str}
In the setup of Construction \ref{calF M},
there exists a unique manifold structure on $\Mcut$
such that $\calF_M$ is the set of real valued smooth functions on $\Mcut$.
\end{restatable}

We prove Theorem~\ref{mfld str} in Section~\ref{sec:proof};
see the first half of Proposition~\ref{Mcut:manifold}.

\begin{Remark} \labell{decomposition}
The decomposition $M = \intM \sqcup \del M$ of $M$
into the disjoint union of its interior and its boundary
descends to a decomposition of the cut space,
$$ \Mcut \, = \, \intM_\cut \, \sqcup \, \Mred \, ,$$
where $ \intM_\cut := c(\intM)$ and $\Mred := c(\del M) = (\del M)/S^1$.
There exist unique manifold structures
on the pieces $\intM_\cut$ and $\Mred$
such that the quotient map restricts to a diffeomorphism 
$c|_{\intM} \colon \intM \to \intM_\cut$
and to a principal $S^1$ bundle
$c|_{\del M} \colon \del M \to \Mred$.
(For $\intM_\cut$, 
this is because $c|_{\intM} \colon x \mapsto \{ x \}$
is a bijection onto $\intM_\cut$.
For $\Mred$, this is because, by Koszul's slice theorem,
the quotient map $\del M \to (\del M)/S^1$ is a principal circle bundle;
see Lemma~\ref{subordinate}\eqref{koszul}.)

The topologies for these manifold structures
are the quotient topologies induced from $\intM$ and $\del M$.
These topologies
coincide with the subset topologies induced from $\Mcut$;
this is a consequence of Lemmas~\ref{invt near bdry}
and~\ref{quotient map}(\ref{sweep-open}).
The inclusion map of $\intM_\cut$ 
is a diffeomorphism with an open dense subset of $\Mcut$
and the inclusion map of $\Mred$ 
is a diffeomorphism with an embedded submanifold of $\Mcut$;
see Lemma~\ref{Mcut:submanifolds}
and the second half of Proposition~\ref{Mcut:manifold}.

In Remark~\ref{dependence on action} we will see 
that the manifold structures on $\intM_\cut$ and on $\Mred$
(which depend only on the circle action on the boundary $\del M$
and not on its extension to a neighbourhood of the boundary)
do not determine the manifold structure on $\Mcut$.
\eor
\end{Remark}

{

As a prototype for the smooth cutting procedure,
and 
to motivate the technical condition ($\calF 2$) in Construction~\ref{calF M},
in the following lemma we characterize the smooth functions
on the plane 
in terms of the angle and the radius-squared in polar coordinates.

\begin{Lemma} \labell{cylinder}
Consider the half-cylinder $M := S^1 \times [0,\infty)$,
with the circle acting on the first component,
and with the function $f \colon M \to \Rplus$ given by $f(b,s) = s$. 
Consider the map $M \to \C$ given by $(b,s) \mapsto \sqrt{s} \, b$.
A real-valued function $h \colon \C \to \R$ is smooth
if and only if its pullback to the half-cylinder,
$\hat{h} \colon M \to \R$, satisfies the following condition.
There exists a smooth function $H \colon M \times \C \to \R$ such that
\begin{enumerate}
\item[(a)]
\ $H(a \cdot x, z) = H( x, az)$ for all $a \in S^1$ 
and $(x,z) \in M \times \C$; \ and 
\item[(b)]
\ $\hat{h}(x) = H \big( x, \sqrt{f(x)} \, \big)$ for all $x \in M$.
\end{enumerate}
\end{Lemma}

\begin{proof}
Identifying $M \times \C$ with $S^1 \times [0,\infty) \times \C$,
the conditions (a) and (b) become
\begin{enumerate}
\item[(a')]
\ $H(a b, s, z) = H( b, s, az)$ for all $a \in S^1$
and $(b,s,z) \in S^1 \times [0,\infty) \times \C$; \ and
\item[(b')]
$h (\sqrt{s}\, b \, ) = H \big( b, s, \sqrt{s} \, \big)$
for all $(b,s) \in S^1 \times [0,\infty)$.
\end{enumerate}
If $h$ is smooth,
then $H(b,s,z) := h(bz)$ is smooth and satisfies (a') and (b').
Conversely, if $H$ is smooth and satisfies (a') and (b'),
then writing $h(z) = H(1,|z|^2,\ol{z})$, we see that $h$ is smooth.
\end{proof}

\begin{center}
\begin{figure}[h]
\includegraphics[scale=.75]{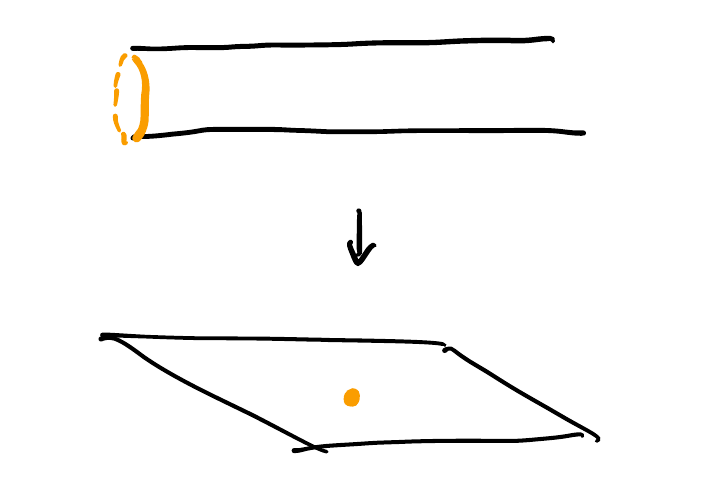}
\caption{Cutting a cylinder}
\end{figure}
\end{center}

}

\begin{Remark}[\textbf{Symplectic polar coordinates; cutting}]
\labell{rk:symplectic polar}
An equivariant symplectic geometer would recognize Lemma~\ref{cylinder}
as constructing $\C$ as the symplectic cut of the cylinder
$N := S^1 \times \Rplus$.
Equip the cylinder with the standard symplectic form 
$\omega_N := ds \wedge d\theta$,
where each point in the cylinder is written as $(e^{i\theta},s)$.
The circle group acts on the first component by left multiplication,
with the momentum map $f \colon N \to \R$ given by $f(b,s) = s$.
Also take $\C$, with the symplectic form\footnote{The factor $2$
simplifies later formulas}
$\omega_\C := 2 dx \wedge dy$,
where the complex coordinate is $z=x+iy$.
Then, take the product $N \times \C$, 
with the circle action $a \cdot (n,z) = (a \cdot n , a^{-1} z)$
and the momentum map $\mu(n,z) = f(n) - |z|^2$.
The reduced space $N//S^1 : = \mu^{-1}(0)/S^1$
is then a symplectic manifold,
which we can identify with the quotient of 
the super-level-set $\{ f \geq 0 \}$
by the equivalence relation that collapses the circle orbits in $f^{-1}(0)$.

More generally, 
in Lerman's symplectic cutting construction~\cite{lerman:symplectic cuts},
we start with a symplectic manifold $(\wt{M},\omega)$ 
with a circle action and a momentum map $f \colon \wt{M} \to \R$
such that the circle action is free on $f^{-1}(0)$
(hence $0$ is a regular level set),
and we take $M := \{ f \geq 0 \}$.
Writing an element of $\C$ as $z = x+iy$, we equip $\wt{M} \times \C$ 
with the split symplectic form $\omega \oplus (2 dx \wedge dy)$,
with the circle action $a \cdot (m,z) = (a \cdot m, a^{-1}z)$,
and with the momentum map $\mu(m,z) = f(m) - |z|^2$.
The symplectic cut of $\wt{M}$ is the reduced space
$(\wt{M} \times \C)/\!/S^1 := \mu^{-1}(0)/S^1$, which is a smooth manifold.
A real valued function on the reduced space is smooth
iff it lifts to an $S^1$-invariant smooth function on the level set 
$\mu^{-1}(0)$.
This level set is a closed submanifold of $\wt{M} \times \C$
that is contained in $M \times \C$;
its invariant smooth functions
are exactly the restrictions to the level set
of the invariant smooth functions on $M \times \C$.
The map $m \mapsto (m,\sqrt{f(m)})$
from $M$ to the zero level set $\mu^{-1}(0)$ 
descends to a bijection of $\Mcut$ with the reduced space
$\mu^{-1}(0)/S^1$;
we use this bijection to define the smooth structure on $\Mcut$.
\eor
\end{Remark}

The idea of the proof of Theorem~\ref{mfld str} is simple:
locally near the boundary we can identify $M$ with open subsets 
of $\R^{n-2} \times S^1 \times \R_{\geq 0}$;
the cutting construction
makes the $S^1$ and $\R_{\geq 0}$ components
into the angle and the radius-squared in polar coordinates on $\C$
as in Lemma~\ref{cylinder}.
This implies that a smooth manifold structure 
with the required properties exists locally on $\Mcut$.
These smooth local manifold structures are consistent
and fit into a global smooth manifold structure.

To provide accurate details,
we work with the formalism of differential spaces
as axiomatized by Sikorski \cite{sikorski1,sikorski2};
see Section~\ref{sec:differential spaces}.
The smooth cutting construction makes $\Mcut$ into a differential space;
see Section~\ref{sec:differential structure on Mcut}.
The consistency of the local manifold structures follows from 
functoriality of the smooth cutting construction 
with respect to inclusion maps of open subsets
that are invariant near the boundary.
{\it {A-priori}}, the smooth cutting construction defines a functor
that takes values in the category of differential spaces;
{\it {a-posteriori}}, it takes values in the category of smooth manifolds.

Functoriality first appears in Section~\ref{sec:functoriality},
where we apply the cutting construction to so-called
equivariant transverse maps.
Inclusions of open subsets that are invariant near the boundary
are examples of equivariant transverse maps.

In Section~\ref{sec:local} we give local models for neighbourhoods in $M$
as open subsets of $\R^{n-2} \times S^1 \times [0,\infty)$.

In Section~\ref{sec:topology} we spell out some point-set-topological
properties of the cut space $\Mcut$ and of the 
quotient map $c \colon M \to \Mcut$.

Section~\ref{sec:differential spaces}
contains an introduction to differential structures,
and in Section~\ref{sec:differential structure on Mcut}
we show that Construction~\ref{calF M}
makes $\Mcut$ into a differential space.

In Section~\ref{sec:proof} we combine the local models 
with the functoriality of the cutting construction 
with respect to inclusions of open subsets
that are invariant near the boundary
to conclude that the cut space is a manifold.

In Section~\ref{sec:submanifolds} we show that 
the cutting construction 
takes immersions (resp., submersions or embeddings) $M \to N$
to immersions (resp., submersions or embeddings) $\Mcut \to \Ncut$.
It follows that ``well behaved'' submanifolds-with-boundary of $N$
descend to submanifolds of $\Ncut$;
see Corollary~\ref{submanifold}.

Sections~\ref{sec:diff forms}--\ref{sec:contact cutting}
are about differential forms and distributions
(sub-bundles of tangent bundles).
For symplectic two-forms, contact one-forms, and contact distributions,
we recover Lerman's earlier results; in particular see 
\cite[Propositions~2.7 and~2.15 and Remark~2.14]{lerman:contact cuts}.
Our current treatment puts these results in a broader context
and provides more details.

In Section~\ref{sec:diff forms}, we show that a differential form $\beta$
on $M$ that is basic on $\del M$ and invariant near $\del M$
descends to a differential form $\beta_\cut$ on $\Mcut$.  
We show that
the map $\beta \mapsto \beta_\cut$ is linear, intertwines 
exterior derivatives, intertwines wedge products, is one-to-one,
and takes non-vanishing forms to non-vanishing forms.
We conclude that 
if $\beta$ is a symplectic two-form (resp., contact one-form) on $M$
then $\beta_\cut$ is a symplectic two-form (resp., contact one-form) 
on $\Mcut$.

In Section~\ref{sec:red} we show that 
the pullback of $\beta_\cut$ under the inclusion map $\Mred \to \Mcut$
coincides with the differential form $\beta_\red$ on $\Mred$
whose pullback to $\del M$ coincides
with the pullback of $\beta$ to $\del M$.
We also show that if $\beta$ 
is a symplectic two-form (resp., contact one-form) on $M$
then $\beta_\red$
is a symplectic two-form (resp., contact one-form) on $\Mred$,
so $\Mred$ is then a symplectic (resp., contact) submanifold of $\Mcut$.

In Sections~\ref{sec:symplectic cutting} and~\ref{sec:contact cutting},
we relate the smooth cutting procedure with differential forms
to the classical version of Lerman's 
symplectic cutting procedure~\cite{lerman:symplectic cuts}
and contact cutting procedure~\cite{lerman:contact cuts}.

In Section~\ref{sec:distributions}, we show that a distribution $E$ on $M$ 
that is $S^1$-invariant near $\del M$,
is transverse to $\del M$, and contains the tangents
to the $S^1$-orbits along $\del M$, 
descends to a distribution $E_\cut$ on $\Mcut$.
A foliation (resp., a contact distribution) on $M$ with these properties
gives a foliation (resp., a contact distribution) on $\Mcut$.
In fact, for a contact distribution $E$,
it is enough to assume that $E$ is $S^1$-invariant near $\del M$
and contains the tangents to the $S^1$-orbits along $\del M$.

In Appendix~\ref{sec:actions},
we recall some facts about actions of compact Lie groups
and their quotients.

In Appendix~\ref{sec:simultaneous}, we sketch 
the construction for simultaneous cutting
along the facets of a manifold-with-corners
that is equipped with commuting circle actions near its facets.
A-posteriori, such a manifold-with-corners 
is what is sometimes called a manifold-with-faces;
in particular, its facets are embedded submanifolds-with-corners.
We expect the results of this paper
to generalize to this setup of simultaneous cutting.
We expect that this generalization would streamline several 
procedures in the literature,
including the unfolding of folded symplectic structures
\cite{unfolding,origami},
the passage from a non-compact cobordism between compact 
Hamiltonian $T$-manifolds 
to a compact cobordism \cite{GGK:compact-noncompact},
and the equivalence of categories
between symplectic toric $T$ manifolds
and symplectic toric $T$-bundles \cite{karshon-lerman}.

\ynote{
\cite[Remark~2.13]{lerman:contact cuts} is simultaneous contact cutting,
yielding contact orbifolds.  }

\bigskip

\subsection*{Equivariant radial-squared blowups}\ 

Sections~\ref{sec:symplectic polar:diffeos}, 
\ref{sec:radial blowup}, and~\ref{sec:radial squared blowup}
constitute a second part to the paper,
in which we describe an inverse to the cutting construction.
This inverse is a modification of the radial blowup construction
that is inspired by Bredon
and which we call
the \emph{equivariant radial-squared-blowup} construction.
A baby-case is the passage 
from Cartesian coordinates $(x,y)$ with $x+iy=re^{i\theta}$
to symplectic polar coordinates $r^2$ (instead of $r$) and~$\theta$.

We start in Section~\ref{sec:symplectic polar:diffeos},
with a baby-example, describing 
equivariant diffeomorphisms in symplectic polar coordinates.
This section can be read independently of the others.
In Section~\ref{sec:radial blowup}
we describe the radial blowup construction
of a manifold along a closed manifold.
We give more details than we found in the literature.
This provides preparation 
for the equivariant-radial-squared-blowup construction,
which we introduce in Section~\ref{sec:radial squared blowup}.

The radial blowup construction 
is functorial with respect to diffeomorphisms,
so a Lie group action on a manifold $X$ 
that preserves a closed submanifold $F$ naturally lifts 
to the radial blowup $X \odot F$ of $X$ along $F$.
In contrast,
the equivariant radial-squared-blowup,
which we denote $X \oodot F$, requires a special group action:
each isotropy representation must be isomorphic 
to a product of groups, acting on a product of their representations,
where in each factor the action is either trivial 
or is transitive on the unit sphere.
This assumption goes back to so-called \emph{special $G$-manifolds},
introduced by J\"anich \cite{janich:earlier}
and Hsiang-Hsiang \cite{hsiang-hsiang},
and treated by Bredon \cite[Chap.~VI, Sec.~5 and~6]{bredon}
(though these authors focused on special cases).
In the special case of circle actions,
the equivariant-radial-squared blowup gives an inverse
to the cutting construction.

To put this in context,
we digress for a moment to discuss iterated (not simultaneous) 
radial (not equivariant-radial-squared) blowups.
For a proper action of a Lie group $G$ on $X$ such that $X/G$ is connected,
iterates of the radial blowup construction along minimal orbit type strata
leads to a manifold-with-corners $M$ with a $G$ action
with constant orbit type, which is a bundle 
whose base can be further collapsed to $M/G$.
This resolution of a group action using the radial blowup construction 
goes back to J\"anich \cite[Section~1.3]{janich}
and was further elaborated by Michael Davis \cite{davis}, 
Duistermaat and Kolk \cite[Section~2.9]{duistermaat-kolk}, 
and Albin and Melrose \cite{albin-melrose}.
Getting back from the manifold-with-corners $M$
to the $G$-manifold $X$ is not straight-forward.
Davis~\cite{davis} does this by keeping track 
of certain ``attaching data'' on $M$.
This iterated construction is different
from the simultaneous construction that we mentioned above.

Finally, we note that, 
as explained by Lerman \cite{lerman:symplectic cuts},
his original cutting construction,
which starts from a manifold cut along a hypersurface
(rather than from a manifold-with-boundary as we do), 
provides an inverse to Gompf's symplectic gluing construction~\cite{gompf}.
The equivariant case is addressed in~\cite{wardenksi}.

\bigskip

This work is inspired by collaborations with 
Eugene Lerman, River Chiang, Shintaro Kuroki, Ana Cannas da Silva,
and Liat Kessler.
With Eugene Lerman, in our classification of 
not-necessarily-compact symplectic toric manifolds \cite{karshon-lerman},
we applied a simultaneous cutting procedure,
and we used the functoriality of this procedure 
with respect to inclusions of invariant open subsets.
With River Chiang~\cite{chiang-karshon},
we use symplectic and contact cutting while keeping track of submanifolds.
With Liat Kessler~\cite{karshon-kessler}, 
we apply the smooth cutting construction to obtain
families of symplectic blowups.
With Shintaro Kuroki~\cite{karshon-kuroki}, 
we use simultaneous cutting to classify smooth manifolds 
with so-called locally standard torus actions.
With Ana Cannas da Silva, we can obtain 
so-called toric Lagrangians submanifolds of symplectic toric manifolds
through a smooth cutting construction,
continuing Cannas da Silva's earlier work
with her student Giovanni Ambrosioni.  

\subsection*{Acknowledgement}
I am grateful to Eugene Lerman, River Chiang, Shintaro Kuroki,
Ana Cannas da Silva, and Liat Kessler,
for inspiring collaborations and discussions
that have lead to this note and contributed to it in many ways.
I am grateful to River Chiang, Liat Kessler, Shintaro Kuroki,
and Eugene Lerman
for comments on drafts of this note.
I am grateful for Eugene Lerman for teaching me about $C^\infty$ rings
and for Jordan Watts for enlightening discussions 
about differential spaces.
I am grateful to Alejandro Uribe for pointing out
the possible relation between cutting and radial blowups \`a-la-Melrose.
I am grateful to Shintaro Kuroki for explaining to me 
subtle aspects of the work of J\"anich, Bredon, and Davis.
If I mistakenly omitted your name 
and you're one of the people who enriched my mathematical world
in ways that are relevant to this paper,
please do not hesitate to mention this to me.
This research is partly funded by the Natural Sciences and Engineering
Research Council of Canada.


{

\section{Functoriality 
with respect to equivariant transverse maps}
\labell{sec:functoriality}

Let $M$ and $N$ be manifolds-with-boundary
equipped with free circle actions on neighbourhoods $U_M$ and $U_N$
of the boundaries $\del M$ and $\del N$.
A map $ \psi \colon M \to N $ from $M$ to $N$ 
is an \textbf{equivariant transverse map} 
if it has the following properties.
\begin{itemize}
\item
$\psi$ is smooth.
\item For some---hence every---invariant boundary defining function 
$f_N \colon N \to \R_{\geq 0}$ on $N$,
the composition $f_N \circ \psi \colon M \to \R_{\geq 0}$
is an invariant boundary defining function on $M$.

Hence, $\psi$ takes $\intM$ to $\intN$ and $\del M$ to $\del N$.

\item
There exists a neighbourhood $V$ of $\del M$ in $M$,
contained in $U_M \cap \psi^{-1}U_N$ and $S^1$-invariant, 
such that $\psi|_V \colon V \to U_N$ is $S^1$-equivariant.
\end{itemize}

\begin{Example}
Let $M$ be a manifold-with-boundary,
equipped with a free circle action on a neighbourhood $U_M$ of the boundary.
Then the inclusion maps of $U_M$ and of $\intM$ into $M$
are equivariant transverse maps.
\eoe
\end{Example}

Manifolds-with-boundary equipped with free circle actions
on neighbourhoods of the boundary,
and their equivariant transverse maps, form a category.
The cutting construction gives a functor on this category,
which we describe in the following lemmas.

\begin{Lemma} \labell{descends}
Let $M$ and $N$ be manifolds-with-boundary, 
equipped with free circle actions on neighbourhoods of the boundary.
Let $\Mcut$ and $N_\cut$ be the corresponding cut spaces, 
and let $c_M \colon M \to \Mcut$ and $c_N \colon N \to N_\cut$
be the quotient maps.
Let $\psi \colon M \to N$ be an equivariant transverse map.
Then there exists a unique map $\bm{\psi_{\cut}} \colon \Mcut \to N_\cut$
such that the following diagram commutes.
$$ \xymatrix{
    M \ar[d]_{c_M} \ar[r]^{\psi} & N \ar[d]^{c_N} \\
    \Mcut \ar[r]_{\psi_{\cut}} & N_\cut
} $$
Moreover, 
\begin{itemize}
\item
$\psi_{\cut}$ takes $\intM_\cut$ to $\intN_\cut$ and $\Mred$ to $\Nred$.
\item
If $\psi$ is one-to-one, so is $\psi_\cut$.
If $\psi$ is onto, so is $\psi_\cut$.
\item
$\psi_{\cut}$ is continuous.
\item
If $\psi$ is open, so is $\psi_\cut$.
\item
If $\psi$ is open as a map to its image, so is $\psi_\cut$.
\end{itemize}
\end{Lemma}

\begin{Remark}
The quotient maps $c_M$ and $c_N$ are not open. 
\eor
\end{Remark}

A subset $W$ of $M$ is \textbf{saturated} with respect
to the equivalence relation $\sim$
if for every two points $x$ and $x'$ in $M$, 
if $x \in W$ and $x \sim x'$, then $x' \in W$.

\begin{proof}[Proof of Lemma~\ref{descends}]
Because $\psi$ restricts to an $S^1$ equivariant map
$\psi|_{\del M} \colon \del M \to \del N$,
it descends to a unique map $\psi_\cut$ such that the diagram commutes.
Because $\psi$ takes $\intM$ to $\intM$
and $\del M$ to $\del N$,
the map $\psi_\cut$ takes $\intM_\cut$ to $\intN_\cut$ and $\Mred$ to $\Nred$.
Assuming that $\psi$ is onto,
onto-ness of $\psi_\cut$ follows from that of $c_N$.
Assuming that $\psi$ is one-to-one,
one-to-one-ness of $\psi_\cut$ follows 
from those of $\psi|_{\intM}$ and of $\psi|_{\del M}$
and the equivariance of $\psi|_{\del M}$.
The continuity of $\psi_\cut$ follows by chasing the commuting diagram,
noting that $\psi$ and $c_N$ are continuous
and that the topology of $\Mcut$ is induced from the quotient map
$c_M \colon M \to \Mcut$.
Assuming that $\psi$ is open,
the openness of $\psi_\cut$ also follows by chasing the commuting square,
noting that $c_M$ is continuous, 
that $\psi$ is open and takes saturated sets to saturated sets,
and that the topology of $\Ncut$ is induced from the quotient map
$c_N \colon N \to \Ncut$.
A similar argument holds if $\psi$ is open as a map to its image,
noting that the image of $\psi$ is the preimage of the image of $\psi_\cut$.
\end{proof}

\begin{Lemma} \labell{composition}
Let $M_1$, $M_2$, and $M_3$ be manifolds with boundary,
with free circle actions near the boundary.
Let 
$$ \xymatrix{ M_1 \ar[r]^{\psi_1} & M_2 \ar[r]^{\psi_2} & M_3 } $$
be equivariant transverse functions.
Then $ (\psi_2 \circ \psi_1)_\cut = (\psi_2)_\cut \circ (\psi_1)_\cut$.
\end{Lemma}

\begin{proof}
Since $\psi_1$ takes $\intM_1$ to $\intM_2$
and $\psi_2$ takes $\intM_2$ to $\intM_3$,
and by the definition of $(\psi_1)_\cut$ and $(\psi_2)_\cut$,
the equality is true on the open dense subset $(\intM_1)_\cut$.
By continuity, the equality is true everywhere.
\end{proof}

\begin{Lemma} \labell{identity}
Let $M$ be a manifold with boundary, 
with a free circle action near the boundary.
The identity map $\Id \colon M \to M$
is an equivariant transverse map, 
and $\Id_\cut$ is the identity map on $\Mcut$.
\end{Lemma} 

\begin{proof}
By the definition of the map $\Id_\cut$,
this map restricts to the identity map on $\intM_\cut$.
By continuity, $\Id_\cut$ is the identity map everywhere.
\end{proof}

\begin{Corollary} \labell{inverse}
Let $M$ and $N$ be manifolds with boundary,
with free circle actions near their boundaries.
Let $\psi \colon M \to N$ be an equivariant diffeomorphism;
let $\psi^{-1} \colon N \to M$ be its inverse. 
Then $\psi_\cut \colon \Mcut \to \Ncut$ is invertible,
and its inverse is $(\psi^{-1})_\cut \colon \Ncut \to \Mcut$.
\end{Corollary}

\begin{proof}
This follows from Lemma~\ref{composition}
and Lemma~\ref{identity}.
\end{proof}

Thus, we get a functor, from the category 
whose objects are manifolds-with-boundary
equipped with free circle actions near the boundary
and whose morphisms are equivariant transverse maps,
to the category of topological spaces and their continuous maps.

Until now, we only considered the cut space as a topological space.
By Lemmas~\ref{descends}, \ref{composition}, and \ref{identity},
the cutting construction defines a functor
from the category of manifolds-with-boundary
equipped with free circle actions on neighbourhoods of their boundaries,
with their equivariant transverse maps,
to the category of topological spaces, with their continuous maps. 
The following lemma shows that the cutting functor
also takes smooth maps to smooth maps,
where the meaning of ``smooth'' 
is in terms of the collections of real valued functions
of Construction~\ref{calF M}.

\begin{Lemma} \labell{descends smooth}
Let $M$ and $N$ be manifolds with boundary,
with free circle actions near the boundary,
and let $\calF_M$ and $\calF_N$ be the sets of real valued functions 
on $\Mcut$ and on $\Ncut$
that are obtained from Construction~\ref{calF M}.
Let $\psi \colon M \to N$ be an equivariant transverse map.
Then for each function $h_N$ in $\calF_N$, 
the composition $h_M := h_N \circ \psi_\cut$ is in $\calF_M$.
\end{Lemma}

{
\begin{proof}
Let $h_N \colon \Ncut \to \R$ be in $\calF_N$.
We would like to show that $h_M := h_N \circ \psi_\cut \colon \Mcut \to \R$
is in $\calF_M$.
Write
$$ (h_M \circ c_M)|_{\intM}
 = (h_N \circ \psi_\cut \circ c_M)|_{\intM}
 = (h_N \circ c_N) \circ \psi|_{\intM} .$$
The right hand side is smooth, 
because $\psi$ restricts to a smooth function from $\intM$ to $\intN$,
and because---since $h_N$ is in $\calF_N$---the
composition $h_N \circ c_N|_{\intN}$ is smooth.
So the left hand side is smooth too,
and so $h_M \circ c_M$ 
satisfies Condition~($\calF 1$) of Construction~\ref{calF M}.
It remains to show that $h_M \circ c_M$ satisfies Condition~($\calF 2$) 
of Construction~\ref{calF M} with respect to 
some invariant boundary defining function on $M$.

Because $\psi \colon M \to N$ is an equivariant transverse map,
there exists a neighbourhood $V_\psi$ of $\del M$,
contained in $U_M \cap \psi^{-1}U_N$ and $S^1$-invariant,
such that $\psi|_{V_\psi} \colon V_\psi \to U_N$
is $S^1$ equivariant.  Fix such a $V_\psi$.

Let $f_N \colon N \to \R_{\geq 0}$ be an invariant boundary defining function 
on $N$.  Let $U_{f_N}$ be a neighbourhood of $\del N$,
contained in $U_N$ and $S^1$-invariant, on which $f_N$ is $S^1$-invariant.
Then $f_M := f_N \circ \psi$ is an invariant boundary defining function 
on $M$,
and $U_{f_M} := V_\psi \cap \psi^{-1}U_{f_N}$ is a neighbourhood of $\del M$,
contained in $U_M$ and $S^1$-invariant, on which $f_M$ is $S^1$-invariant.

Because $h_N$ is in $\calF_N$,
there exist an $S^1$-invariant open neighbourhood $V_N$ of $\del N$,
contained in $U_N$ and $S^1$-invariant,
and a smooth function $H_N \colon V_N \times \C \to \R$,
such that $H_N(a \cdot y, z) = H_N(y,az)$ 
for all $a \in S^1$ and $(y,z) \in V_N \times \C$,
and such that $h_N(c_N(y)) = H_N(y,\sqrt{f_N(y)})$ 
for all $y \in V_N$.
Fix such $V_N$ and $H_N$.

$V_M := U_{f_M} \cap \psi^{-1}V_N$ is a neighbourhood of $\del M$,
contained in $U_{f_M}$ and $S^1$-invariant.
Define $H_M \colon V_M \times \C \to \R$ by
$H_M(x,z) := H_N(\psi(x),z)$.
Then $H_M$ is smooth, $H_M(a \cdot x,z) = H_M(x,az)$
for all $(x,z) \in V_M \times \C$ and $a \in S^1$,
and $h_M(c_M(x)) = H_M(x,\sqrt{f_M(x)})$ for all $x \in V_M$.
Thus, $h_M \circ c_M$ satisfies 
Condition ($\calF 2$)(a,b) of Construction~\ref{calF M}
with respect to $f_M$.
\end{proof}
}

}

\section{A local model}
\labell{sec:local}

Let $\eps > 0$ be a positive number,
let $D^2$ be the open disc of radius $\sqrt{\eps}$ about the origin in $\R^2$,
and let $D^{n-2}$ be any open disc in $\R^{n-2}$.
{

\begin{Lemma} \labell{model}
Equip
$$ N := D^{n-2} \times S^1 \times [0,\eps) $$
with the circle action $a \cdot (\xi,b,s) = (\xi,ab,s)$
and with the invariant boundary defining function $(\xi,b,s) \mapsto s$.
Then the map
$$
 \wh{\psi} \colon N \to D^{n-2} \times D^2
\quad \text{ given by } \quad (\xi,b,s) \mapsto (\xi , b\sqrt{s}) 
$$
induces a homeomorphism
$$ \psi \colon \Ncut \to D^{n-2} \times D^2 $$
and induces a bijection 
$$ f \mapsto f \circ \psi $$
from the set of real valued smooth functions 
on $D^{n-2} \times D^2$
to the set $\calF_N$ of real valued functions on~$\Ncut$
that is described in Construction~\ref{calF M}.
\end{Lemma}

\begin{proof}
By its definition, the map $\wh{\psi}$ descends to a bijection $\psi$,
such that we have a commuting diagram
$$ \xymatrix{
 N \ar[rr]^{c} \ar@/_1pc/_{ \wh{\psi} }[rrrr] 
 && \Ncut \ar[rr]^{\psi} && D^{n-2} \times \C \, .
} $$
Because $\wh{\psi}$ is continuous
and the topology of $\Ncut$ 
is induced from the quotient map $c \colon N \to \Ncut$,
the map $\psi$ is continuous.
Because $\wh{\psi}$ is proper and $c$ is continuous, 
$\psi$ is proper;
because the target space $D^{n-2} \times \C$
is Hausdorff and locally compact, the proper map $\psi$ is a closed map.
Being a continuous bijection and a closed map, $\psi$ is a homeomorphism.

We need to show that a function $f \colon D^{n-2} \times D^2 \to \R$
is smooth iff the function 
$\wh{h}(\xi,b,s) := f(\xi,b\sqrt{s})$ from $N$ to $\R$ 
satisfies Conditions ($\calF 1$) and ($\calF 2$) of Construction~\ref{calF M}.
Because $\wh{\psi}$ restricts to a diffeomorphism
from $\intN$ to $D^{n-2} \times (D^2 \ssminus \{0\}) $,
Condition~($\calF 1$) of Construction~\ref{calF M} is equivalent to 
the restriction of $f$ to $D^{n-2} \times (D^2 \ssminus \{0\})$ being smooth.
It remains to show that
$\wh{h}$ satisfies Condition~($\calF 2$) of Construction~\ref{calF M} 
if and only if $f$ is smooth near $D^{n-2} \times \{ 0 \}$.

We proceed as in Lemma~\ref{cylinder}.

Suppose that $\wh{h}$ satisfies Condition~($\calF 2$) 
of Construction~\ref{calF M}. 
Let $V$ be an $S^1$-invariant open neighbourhood of $\del N$ 
and $H \colon V \times D^2 \to \R$ a smooth function
that satisfies Conditions ($\calF 2$)(a,b) of Construction~\ref{calF M}.
Because $V$ is an $S^1$-invariant open subset of $N$,
its image in $D^{n-2} \times D^2$ is open.
Writing $f(\xi,z) = H((\xi,1,|z|^2),\ol{z})$,
we see that $f$ is smooth on this image.

Conversely, suppose that $f$ is smooth near $D^{n-2} \times \{ 0 \}$.
Let $V$ be the preimage in $N$ 
of an $S^1$-invariant neighbourhood of $D^{n-2} \times \{ 0 \}$
on which $f$ is smooth.  
Then the function $H \colon V \times D^2 \to \R$ 
defined by $H((\xi,b,s),z) := f(\xi,bz)$ 
satisfies Conditions~($\calF 2$)(a,b) of Construction~\ref{calF M}.
\end{proof}

}

\begin{Proposition} \labell{prop:local}
Let $M$ be an $n$ dimensional manifold-with-boundary,
equipped with a circle action on a neighbourhood $U_M$ of the boundary.
Let $f \colon M \to \R$ be an invariant boundary defining function.
Then for each point $x$ in $\del M$ 
there exists an $S^1$-invariant open neighbourhood~$W$ of $x$ in $U_M$,
a positive number $\eps > 0$, and a diffeomorphism 
$$W  \to  D^{n-2} \times S^1 \times [0,\eps) $$
that intertwines the $S^1$ action on $W$
with the rotations of the middle factor
and that intertwines the function $f|_W$
with the projection to the last factor.

Moreover, if $v$ is a vector field near $\del M$
that is transverse to $\del M$
and that is $S^1$-invariant near $\del M$,
then this diffeomorphism can be chosen 
such that $v$ is tangent to the fibres of the projection 
$W \to D^{n-2} \times S^1$.
\end{Proposition}

Proposition~\ref{prop:local} is a minor variation
of the collar neighbourhood theorem. 
For completeness, we include a proof.

\begin{proof}[Proof of Proposition~\ref{prop:local}]
Let $v$ be a vector field near $\del M$ that is transverse to $\del M$
and that is $S^1$-invariant near $\del M$.
(We can obtain such a $v$ by using local coordinates charts near $\del M$
to obtain inward-pointing vector fields,  
patching these vector fields with a partition of unity,
and averaging with respect to the $S^1$ action.)
After multiplying $v$ by a non-vanishing smooth function,
we may assume that the derivative $vf$ of $f$ along~$v$ is equal to one.
The forward-flow of the vector field $v$
determines a diffeomorphism from an open neighbourhood $U$
of $\del M \times \{ 0 \}$ in $\del M \times [0,\infty)$ 
to an open neighbourhood of $\del M$ in $M$
that intertwines the $S^1$ action on the $\del M$ factor 
with the $S^1$ action on $M$
and whose composition with $f$ is the projection to the $[0,\infty)$ factor.

Let $x \in \del M$.
Let $\eps > 0$ be such that $\{ x \} \times [0,\eps]$ is contained in~$U$.
Using a chart for $\del M$ near $x$,
we obtain an open disc $D^{n-2}$ about the origin in $\R^{n-2}$
and an embedding of $D^{n-2}$ into $\del M$
that takes the origin to $x$ and is transverse to the $S^1$ orbit through $x$.
``Sweeping'' by the circle action, and possibly shrinking $D^{n-2}$,
we obtain an open neighbourhood $W'$ of $S^1 \cdot x$ in $\del M$
such that $W' \times [0,\eps]$ is still contained in $U$
and a diffeomorphism from $W'$ to $D^{n-2} \times S^1$
that intertwines the $S^1$ action on $W'$ with the rotations 
of the $S^1$ factor.
To finish, take $W$ to be the image of $W' \times [0,\eps)$ in $M$,
and compose the inverse of the diffeomorphism $W' \times [0,\eps) \to W$
with the diffeomorphism $W' \to D^{n-2} \times S^1$.
\end{proof}

\section{The cut space as a topological space}
\labell{sec:topology}

Let $M$ be an $n$ dimensional manifold-with-boundary,
equipped with a circle action on a neighbourhood $U_M$ of the boundary.
Let $\Mcut = M/{\sim}$ be obtained from $M$ by the cutting construction,
and let $c \colon M \to \Mcut$ be the quotient map.
Equip $\Mcut$ with the quotient topology.

Let $W$ be an open subset of $M$; 
then $W$ is a manifold-with-boundary,
and its boundary is given by $\del W = \del M \cap W$.
We say that such an $W$ is 
$\mathbf{S^1}$\textbf{-invariant near its boundary} 
if there exists an $S^1$-invariant open neighbourhood of $\del M$ in $U_M$ 
whose intersection with $W$ is $S^1$-invariant.
For example, open subsets of $\intM$
and $S^1$-invariant open subsets of $U_M$
are $S^1$-invariant near their boundary
(for subsets of $\intM$, this condition is vacuous).

\begin{Lemma} \labell{invt near bdry}
For any open subset $W$ of $M$
that is $S^1$-invariant near its boundary,
$\Wcut$ is an open subset of $\Mcut$,
and the quotient topology of $\Wcut$ that is induced from $W$ 
agrees with the subset topology on $\Wcut$ that is induced from $\Mcut$.
\end{Lemma}

\begin{proof}
Because the open subset $W$ of $M$
is $S^1$-invariant near its boundary,
the inclusion map of $W$ into~$M$ is an equivariant transverse map.
This map is one-to-one, continuous, and open.
By Lemma~\ref{descends}, the inclusion map of $\Wcut$ into $\Mcut$
is one-to-one, continuous, and open,
so it is a homeomorphism with an open subset of $\Mcut$.
\end{proof}


\begin{Lemma} \labell{Mcut Hausdorff}
$\Mcut$ is Hausdorff.
\end{Lemma}

\begin{proof}
Let $y_1$ and $y_2$ be distinct points of $\Mcut$.

\begin{itemize}
\item
Suppose that $y_1 = \{ x_1 \}$ and $y_2 = \{ x_2 \}$
for distinct points $x_1$ and $x_2$ of $\intM$.
Because $M$ is Hausdorff, there exist disjoint open subsets 
$V_1$ and $V_2$ of $M$ such that $x_1 \in V_1$ and $x_2 \in V_2$.
Let $V_1' := \intM \cap V_1$ and $V_2' := \intM \cap V_2$.

\item
Suppose that $y_1 = S^1 \cdot x_1$ and $y_2 = S^1 \cdot x_2$
are distinct orbits in $\del M$.
Then $y_1$ and $y_2$ are disjoint compact subsets of $M$
(see Lemma~\ref{quotient map}\eqref{orbits-compact-closed}).
Because $M$ is Hausdorff, there exist disjoint open subsets $V_1$
and $V_2$ of $M$ such that $y_1 \subset V_1$ and $y_2 \subset V_2$.
Let $V_1' := \bigcap\limits_{a \in S^1} a \cdot (U_M \cap V_1)$
and $V_2' := \bigcap\limits_{a \in S^1} a \cdot (U_M \cap V_2)$;
these subsets of $M$ are open 
(by Lemma~\ref{quotient map}\eqref{intersection}).

\item
Otherwise, after possibly switching $y_1$ and $y_2$,
we may assume that $y_1 = S^1 \cdot x_1$ for $x_1 \in \del M$
and $y_2 = \{ x_2 \}$ for $x_2 \in \intM$.
As before, $y_1$ and $y_2$ are disjoint compact subsets of $M$,
so there exist disjoint open subsets $V_1$
and $V_2$ of $M$ such that $y_1 \subset V_1$ and $y_2 \subset V_2$.
Let $V_1' := \bigcap\limits_{a \in S^1} a \cdot (U_M \cap V_1)$
and $V_2' := \intM \cap V_2$;
as before, these subsets of $M$ are open.
\end{itemize}

In each of these cases, $V_1'$ and $V_2'$
are open subsets of $M$ that are $S^1$-invariant
near their boundaries, and 
$y_1 \subset V_1' \subset V_1$ and $y_2 \subset V_2' \subset V_2$,
where $V_1$ and $V_2$ are disjoint.
By Lemma~\ref{invt near bdry}, $(V_1')_\cut$ and $(V_2')_\cut$
are then disjoint open neighbourhoods of $y_1$ and $y_2$ in $\Mcut$.
\end{proof}

\begin{Lemma} \labell{Mcut 2nd countable}
$\Mcut$ is second countable.
\end{Lemma}

\begin{proof}
Let $\frakU$ be a countable basis for the topology of $M$.
We will show that 
$\frakU' := \{ \big( \intM \cap U \big)_\cut \}_{U \in \frakU}
 \cup \{ \big( S^1 \cdot (U_M \cap U)  \big)_\cut \}_{U \in \frakU}$ 
is a countable basis for the topology of $M_\cut$.

The countability of $\frakU'$ follows from that of $\frakU$.

For each $U \in \frakU$, the sets $\intM \cap U$ 
and $S^1 \cdot (U_M \cap U)$ are open in $M$ 
(see Lemma~\ref{quotient map}\eqref{sweep-open})
and are $S^1$-invariant near their boundaries.
By Lemma~\ref{invt near bdry}, these sets are open in $\Mcut$.

Let $y$ be a point in $\Mcut$ and $W$ an open neighbourhood of $y$
in $\Mcut$.
\begin{itemize}
\item
Suppose that $y = \{ x \}$ for $x \in \intM$.
Let $U \in \frakU$ be such that $x \in U \subset c^{-1}(W)$. 
Then $ U' := (\intM \cap U)_\cut$ is in $\frakU'$, 
and $y \in U' \subset W$.
\item
Suppose that $y = S^1 \cdot x$ for $x \in \del M$.
By Lemma~\ref{quotient map}\eqref{intersection},
$V := \bigcap\limits_{a \in S^1} a \cdot ( U_M \cap c^{-1}(W) )$
is open (in $U_M$, hence) in $M$.
Let $U \in \frakU$ be such that $x \in U \subset V$.
Then $U' := (S^1 \cdot ( U_M \cap U) )_\cut$ is in $\frakU'$,
and $y \in U' \subset W$.
\end{itemize}
In either case, we found an element $U'$ of $\frakU'$
such that $y \in U' \subset W$.
Because $y$ and $W$ were arbitrary, $\frakU'$ is a basis for the topology.
\end{proof}

\begin{Lemma} \labell{loc Rn}
Each point in $\Mcut$ has a neighbourhood
that is homeomorphic to an open subset of~$\R^n$.
\end{Lemma}

\begin{proof}
Take a point $y$ in $\Mcut$.

\begin{itemize}
\item
Suppose that $y = \{ x \}$ for $x \in \intM$. 
Let $W$ be the intersection of $\intM$
with the domain of a coordinate chart on $M$ that contains~$x$.
By Lemma~\ref{invt near bdry}, $W_\cut$ is an open neighbourhood
of $y$ in $\Mcut$ that is homeomorphic to an open subset of $\R^n$.

\item
Suppose that $y = S^1 \cdot x$ for $x \in \del M$.
Let 
$$ W \to D^{n-2} \times S^1 \times [0,\eps) $$
be a diffeomorphism as in Proposition~\ref{prop:local}.
By Lemma~\ref{descends}, 
this diffeomorphism descends to a homeomorphism 
$$ \Wcut \to (D^{n-2} \times S^1 \times [0,\eps))_\cut.$$
Lemma~\ref{model} gives a homeomorphism
$$ (D^{n-2} \times S^1 \times [0,\eps))_\cut
 \to D^{n-2} \times D^2(\sqrt{\eps}).$$
By Lemma~\ref{invt near bdry},
$\Wcut$ is an open neighbourhood of $y$ in $\Mcut$,
and its topology as a subset of $\Mcut$ 
agrees with its quotient topology that is induced from $W$.
Composing the above two homeomorphisms,
we conclude that this neighbourhood of $y$ in $\Wcut$
is homeomorphic to an open subset of $\R^n$.
\end{itemize}
\end{proof}


\section{Differential spaces}
\labell{sec:differential spaces}

A differential space is a set 
that is equipped with a collection of real valued functions,
of which we think as the smooth functions on the set,
that satisfies a couple of axioms.
Variants of this structure were used by many authors.
In particular, we would like to mention Bredon
\cite[Chap.~VI, \S 1]{bredon}.
Our definition follows Sikorski \cite{sikorski1,sikorski2}.
We find it convenient to phrase it in terms of $C^\infty$ rings,
of which we learned from Eugene Lerman.
For further developments, see \'Sniatycki's book \cite{sniatycki}.

A non-empty collection $\calF$ of real-valued functions
on a set $X$ is a $\mathbf{C^\infty}$\textbf{-ring} 
(with respect to the usual composition operations)
if for any positive integer $n$ and functions $h_1,\ldots,h_n$ in $\calF$,
and for any smooth real valued function $g$ on $\R^n$,
the composition $x \mapsto g( h_1(x) , \ldots , h_n(x) )$ is in $\calF$.

The \textbf{initial topology} determined by a collection $\calF$
of functions on a set $X$ is the smallest topology on $X$
for which the functions in $\calF$ are continuous.

If $\calF$ is a $C^\infty$-ring, 
then the preimages of open intervals by elements of $\calF$
are a basis (not only a sub-basis) for the initial topology.

\begin{Remark} \labell{given=initial}
Let $X$ be a topological space,
and let $\calF$ be a $C^\infty$-ring $\calF$ of real-valued functions on $X$.
Then the initial topology determined by $\calF$
coincides with the given topology on $X$
iff 
\begin{itemize}
\item
All the functions in $\calF$ are continuous; and 
\item
$X$ is $\mathbf{\calF}$\textbf{-regular},
in the following sense:
for each closed subset $C$ of $X$ and point $x \in X \ssminus C$,
there exists a function $f \in \calF$,
such that $f(x) \neq 0$ and such that $f$ vanishes on a neighbourhood of $C$. 
\end{itemize}

Indeed,
the initial topology is contained in the given topology
iff the functions in $\calF$ are continuous,
and the given topology is contained in the initial topology
iff $X$ is $\calF$-regular.
\eor
\end{Remark}

Let $\calF$ be a non-empty collection of real-valued functions on set $X$.
Equip $X$ with the initial topology determined by $\calF$.
The collection $\calF$ is a \textbf{differential structure}
if it is a $C^\infty$-ring
and it satisfies the following \textbf{locality condition}.
Given any function $h \colon X \to \R$,
if for each point in $X$ there exists a neighbourhood $V$
and a function $g$ in $\calF$ such that $h|_V = g|_V$,
then $h$ is in~$\calF$.

A \textbf{differential space} is a set equipped with a differential structure.
Given differential spaces $(X,\calF_X)$ and $(Y,\calF_Y)$,
a map $\psi \colon X \to Y$ is \textbf{smooth}
if for any function $h$ in $\calF_Y$
the composition $h \circ \psi$ is in $\calF_X$;
the map $\psi$ is a \textbf{diffeomorphism}
if it is a bijection and it and its inverse are both smooth.

On a subset $A$ of a differential space $(X,\calF)$,
the \textbf{subset differential structure}
consists of those functions $h \colon A \to \R$ such that,
for every point in $A$, there exist a neighbourhood $U$ in $X$
and a function $g$ in $\calF$ such that $h|_{U \cap A} = g|_{U \cap A}$.

A manifold $M$, equipped with the set of real valued functions
that are infinitely differentiable, is a differential space.
A map between manifolds is infinitely differentiable
if and only if it is smooth in the sense of differential spaces.
In this way, we identify manifolds with those differential spaces $X$
that are Hausdorff and second countable
and 
that are locally diffeomorphic to Cartesian spaces in the following sense:
for each point in $X$ there exist a neighbourhood $U$
and a diffeomorphism of $U$ with an open subset of a Cartesian space $\R^n$.
Manifolds with boundary, and manifolds with corners 
\cite[Chap.~16]{JohnLee:smooth}, 
are similarly identified with those differential spaces 
that are Hausdorff and second countable
and
that are locally diffeomorphic to half-spaces $\R^{n-1} \times \R_{\geq 0}$
or, respectively, to orthants $\R_{\geq 0}^n$.

\section{The cut space as a differential space}
\labell{sec:differential structure on Mcut}

Let $M$ be a manifold-with-boundary, equipped with a free circle action 
on a neighbourhood $U_M$ of the boundary.
Let $\Mcut$ be obtained from $M$ by the cutting construction,
and let $c \colon M \to \Mcut$ be the quotient map.
Let $\calF_M$ be the set of real valued functions on $M$
that is described in Construction~\ref{calF M}
and Lemma~\ref{FM indep of f}.

\begin{Lemma} \labell{Mcut:Cinfty ring}
$\calF_M$ is a $C^\infty$ ring of real-valued functions on $\Mcut$.
\end{Lemma}

\begin{proof}
Let $n$ be a positive integer, and let $h_1,\ldots,h_n$ be functions 
in $\calF_M$.
Let $g$ be a smooth real valued function on $\R^n$.
Define $h \colon \Mcut \to \R$ by $h(x) := g\big(h_1(x),\ldots,h_n(x)\big)$.
We need to prove that the function 
$\wh{h} := h \circ c \colon M \to \R$ satisfies 
Conditions ($\calF 1$) and ($\calF 2$) of Construction~\ref{calF M}.

Let $\wh{h}_i := h_i \circ c$ for $i=1,\ldots,n$.


\begin{itemize}
\item 
The functions $\wh{h}_1,\ldots,\wh{h}_n$ 
satisfy Condition~($\calF 1$) of Construction~\ref{calF M},
so they are smooth on $\intM$.
Because the function $\wh{h}$ can be written as the composition
$g( \wh{h}_1(\cdot) , \ldots , \wh{h}_n(\cdot) )$
of these functions with the smooth function $g \colon \R^n \to \R$,
it too is smooth on $\intM$.
So $\wh{h}$ satisfies Condition~($\calF 1$) of Construction~\ref{calF M}.

\item 
The functions $\wh{h}_1,\ldots,\wh{h}_n$
satisfy Condition~($\calF 2$) of Construction~\ref{calF M},
so for each $i \in \{1,\ldots,n\}$
there exists an $S^1$-invariant neighbourhood $V_i$ of $\del M$ in~$U_M$
and a smooth function $H_i \colon V_i \times \C \to \R$
such that $H_i(a \cdot x , z) = H_i (x, az)$ for all $a \in S^1$
and $(x,z) \in V \times \C$
and such that $\wh{h}_i(x) = H_i(x,\sqrt{f(x)})$ for all $x \in V_i$.
The intersection $V := V_1 \cap \ldots \cap V_n$
is an $S^1$-invariant open neighbourhood of $\del M$ in~$U_M$.
Define $H' \colon V \times \C \to \R$ by 
$H'(x',z) = g( H_1(x',z) , \ldots, H_n(x',z) )$.
Then $H'$ is smooth,
$H' (a \cdot x' , z) = H' (x', az)$ 
for all $a \in S^1$ and $(x',z) \in V \times \C$,
and $\wh{h}(x') = H' ( x' , \sqrt{f(x')} )$ for all $x' \in V$.
So $\wh{h}$ satisfies Condition~($\calF 2$) of Construction~\ref{calF M}.
\end{itemize}

\end{proof}

\begin{Lemma} \labell{Mcut:initial topology}
The initial topology on $\Mcut$ determined by $\calF_M$
coincides with the given topology of $\Mcut$
(which is the quotient topology induced from $M$).
\end{Lemma}

\begin{proof}
By Remark~\ref{given=initial}, we need to show two things:
\begin{enumerate}
\item[(i)]
\emph{Every function in $\calF_M$ is continuous
(with respect to the quotient topology)}.
\item[(ii)]
\emph{$\Mcut$ (with the quotient topology) is $\calF_M$-regular.
}
\end{enumerate}

\noindent\emph{Proof of (i):} \ 
Let $h$ be a function in $\calF_M$,
and let $\wh{h} := h \circ c \colon M \to \R$.
Because $\wh{h}$ satisfies
Condition~($\calF 1$) of Construction~\ref{calF M},
it is continuous on $\intM$.
Because $\wh{h}$ satisfies 
Condition~($\calF 2$) of Construction~\ref{calF M},
it is continuous on an open neighbourhood $V$ of $\del M$ in $U_M$.
Because $\intM$ and $V$ are open in $M$ and their union is $M$,
the function $\wh{h} \colon M \to \R$ is continuous.
Chasing the commuting diagram
$$ \xymatrix{
 M \ar[rr]^{c} \ar@/_1pc/_{ \wh{h} }[rrrr] 
 && \Mcut \ar[rr]^{h} && \R \, ,
} $$
we conclude that $h \colon \Mcut \to \R$ is continuous
with respect to the quotient topology.

\noindent\emph{Proof of (ii):} \  
We need to show that for every point $y$ in $\Mcut$
and neighbourhood $\calO$ of $y$
there exists a function in $\calF_M$ 
that is non-zero on $y$
and whose support in $\Mcut$ is contained in~$\calO$.

Let $y$ be a point in $\Mcut$ 
and $\calO$ an open neighbourhood of $y$ in $\Mcut$
(with respect to the quotient topology).
Also consider $y$ as a subset of $M$ that is contained in $c^{-1}(\calO)$.

\begin{itemize}
\item
Suppose that $y = \{ x \}$ for $x \in \intM$.
Let $\wh{h} \colon M \to \R$ be a non-negative smooth function 
that is positive on $x$
and that vanishes outside a closed subset $A$ of $M$
that is contained in $c^{-1}(\calO) \cap \intM$.

\item
Suppose that $y = S^1 \cdot x$ for $x \in \del M$.
Then $U_1 := \bigcap\limits_{a \in S^1} 
                    a \cdot \left( c^{-1}(\calO) \cap U_M \right)$
is an $S^1$-invariant open neighbourhood of $y$ in~$U_M$
(by Lemma~\ref{quotient map}\eqref{intersection})
that is contained in $c^{-1}(\calO)$.
Let $\wh{h}_1 \colon M \to \R$ be a non-negative smooth function
that is positive along $S^1 \cdot x$
and vanishes outside a closed subset $A$ of $M$
that is contained in~$U_1$. 
Let $\wh{h} \colon M \to \R$ be the function that vanishes outside~$U_M$
and whose restriction to~$U_M$ 
is the $S^1$-average of the function~$\wh{h}_1$.
\end{itemize}

In either case, the function $\wh{h} \colon M \to \R$ 
has the following properties.
\begin{itemize}

\item[(a)]
$\wh{h}$ is smooth on $M$,
and it is $S^1$-invariant 
on some $S^1$-invariant open subset $V$ of $U_M$
that contains $\del M$.
(In the first case we can take $V = \bigcap\limits_{a \in S^1}
 a \cdot (U_M \ssminus A)$,
which is open by Lemma~\ref{quotient map}\eqref{intersection}.
In the second case we can take $V = U_M$.)

\item[(b)]
There exists a closed subset $A_1$ of $M$ 
that contains the carrier $\{ \wh{h} \neq 0 \}$ 
and is contained in $c^{-1}(\calO)$
and whose complement is $S^1$-invariant along its boundary.
(In the first case, we can take $A_1=A$.
In the second case, we can take $A_1 = S^1 \cdot A$,
which is closed by Lemma~\ref{quotient map}\eqref{action is closed}.)

\item[(c)]
$\wh{h}$ is positive along $y$.
\end{itemize}

By (a), the function $\wh{h}$ is constant
on the level sets of $c \colon M \to \Mcut$, 
so it determines a function $h \colon \Mcut \to \R$ 
such that ${\wh{h} = h \circ c}$.
Also by (a), the function $\wh{h}$ 
satisfies Conditions ($\calF 1$) and ($\calF 2$) of Construction~\ref{calF M}
(with $H(x,z):=\wh{h}(x))$, so the function $h$ is in $\calF_M$.
By (b) and (c), the function $h$ is positive on $y$
and its support in $\Mcut$ is contained in $\calO$.

\end{proof}

\begin{Lemma} \labell{Mcut:locality}
$\calF_M$ satisfies the locality condition of a differential structure.
\end{Lemma}

\begin{proof}
By Lemma~\ref{Mcut:initial topology},
in the formulation of the locality condition,
we may take ``neighbourhood''
to be with respect to the quotient topology on $\Mcut$
that is induced from the topology on $M$. 

Take any function $h \colon \Mcut \to \R$.
Assume that for each point $x \in M$
there exists a neighbourhood $\calO_x$ of $c(x)$ in $\Mcut$ 
and a function $g_x \colon \Mcut \to \R$ that is in $\calF_M$ 
and such that $h$ and $g_x$ coincide on $\calO_x$.
We need to prove that the function $\wh{h} := h \circ c \colon M \to \R$
satisfies Conditions ($\calF 1$) and ($\calF 2$) of Construction~\ref{calF M}.

For each point $x \in M$, fix $\calO_x$ and $g_x$ as above, 
and let $\wh{g}_x := g_x \circ c$.

\begin{itemize}

\item
Let $x \in \intM$. 
Because the function $\wh{g}_x$ satisfies Condition~($\calF 1$) 
of Construction~\ref{calF M}, it is smooth on $\intM$.
Because the function $\wh{h}$ coincides with $\wh{g}_x$
on the set $\intM \cap c^{-1}(\calO_x)$ 
and this set is an open neighbourhood of $x$ in $\intM$,
the function $\hat{h}$ is smooth at $x$.
Because $x$ was an arbitrary point of $\intM$,
we conclude that $\hat{h}$ 
satisfies Condition~($\calF 1$) of Construction~\ref{calF M}.

\item
For each $x \in \del M$, because $\wh{g}_{x}$ satisfies
Condition ($\calF 2$) of Construction~\ref{calF M},
there exist an $S^1$-invariant open neighbourhood $V_x$ of $\del M$ in $U_M$
and a smooth function $G_x \colon V_x \times \C \to \R$,
such that $G_x( a \cdot x' , z) = G( x',az)$ for all $a \in S^1$
and $(x',z) \in V_x \times \C$,
and such that $\wh{g}_x (x') = G_x ( x' , \sqrt{f(x')} \,) $
for all $x' \in V_x$.  Fix such $V_x$ and $G_x$.
The set $\calO_x' := \bigcap\limits_{a \in S^1} 
a \cdot (V_x \cap c^{-1}(\calO_x))$
is an open neighbourhood of $S^1 \cdot x$ in $U_M$
(see Lemma~\ref{quotient map}\eqref{intersection})
on which $\wh{h}$ coincides with $\wh{g}_x$.

The sets $\calO_x'$, together with the set $U_M \cap \intM$,
are a covering of $U_M$ by $S^1$-invariant open sets.
Let $( \rho_i \colon U_M \to \R_{\geq 0} )_{i \geq 0}$ be 
an $S^1$-invariant partition of unity that is subordinate to this covering
(see Lemma~\ref{subordinate}\eqref{partition of unity}),
with $\supp \rho_0 \subset U_M \cap \intM$
and $\supp \rho_i \subset \calO_{x_i}'$ for $i\geq 1$.

The set $V:= U_M \ssminus \supp \rho_0$
is an $S^1$-invariant open neighbourhood of $\del M$ in $U_M$,
and $(\left.\rho_i\right|_V)_{i \geq 1}$ (without $i=0$)
is a partition of unity on $V$.
Because each $\hat{g}_{x_i}$ coincides with $\hat{h}$
on the carrier $\{ \rho_i \neq 0 \}$,
we have $\wh{h}|_{V} = \sum_{i \geq 1} \rho_i \left.\wh{g}_{x_i} \right|_{V}$.

Define $H \colon V \times \C \to \R$
by $H := \sum_{i \geq 1} \rho_i G_{x_i}|_{V \times \C}$.
Then $H$ is smooth, $H(a \cdot x' , z) = H(x', a z)$
for all $a \in S^1$ and $(x',z) \in V \times \C$,
and $\wh{h}(x') = H(x',\sqrt{f(x')})$ for all $x' \in V$.
So $\hat{h}$ satisfies Condition~($\calF 2$) of Construction~\ref{calF M}.
\end{itemize}
\end{proof}

\begin{Corollary} \labell{FM is differential structure}
In the setup of Construction~\ref{calF M},
the set $\calF_M$ is a differential structure on~$\Mcut$.
\end{Corollary}

\begin{proof}
By Lemma~\ref{Mcut:Cinfty ring}, $\calF_M$ is a $C^\infty$ ring.
By Lemma~\ref{Mcut:locality},
$\calF_M$ also satisfies the locality condition of a differential structure.
\end{proof}


\begin{Lemma} \labell{open subset}
Let $M$ be a manifold-with-boundary,
equipped with a free circle action on a neighbourhood $U_M$ of the boundary.
Let $W$ be an open subset of $M$
that is $S^1$-invariant near its boundary.
Then the inclusion of $W$ into $M$
descends to a diffeomorphism (as differential spaces)
of $\Wcut$ with an open subset of $\Mcut$.
\end{Lemma}

\begin{proof}
Let $\calF_M$ and $\calF_W$, respectively, 
be the differential structures on $\Mcut$ and on $\Wcut$
that are obtained from Construction~\ref{calF M}.
We need to show that $\calF_W$ coincides with the subset differential structure
that is induced from $\calF_M$.

By Lemma~\ref{descends smooth}, the inclusion map of $\Wcut$ into $\Mcut$
is smooth as a map of differential spaces.
So the subset differential structure on $\Wcut$
that is induced from $\calF_M$ is contained in $\calF_W$.
For the converse, fix any function $h \colon \Wcut \to \R$ 
in the differential structure $\calF_W$,
let $y$ be a point in $\Wcut$, and let $\hat{h} := h \circ c \colon W \to \R$.
\begin{itemize}
\item
Suppose that $y = \{ x \}$ for $x \in \intM$.
Let $\rho \colon M \to \Rplus$ be a smooth function
that is equal to $1$ near $x$
and whose support is contained in $\intM \cap W$.
Let $\hat{g} \colon M \to \R$ be equal to $\rho \hat{h}$ on $\intM \cap W$
and zero elsewhere.
Let $g \colon \Mcut \to \R$ be such that $\hat{g} = g \circ c$.
\item
Suppose that $y = S^1 \cdot x$ for $x \in \del M$.
Let $\rho \colon M \to \Rplus$ be a smooth function
whose support is contained in $U_M \cap W$, that is $S^1$-invariant,
and that is equal to $1$ near $x$
(see Lemma~\ref{subordinate}\eqref{bump}).
Let $\hat{g} \colon M \to \R$ be equal to $\rho \hat{h}$ on $U_M \cap W$
and zero elsewhere.
Let $g \colon \Mcut \to \R$ be such that $\hat{g} = g \circ c$.
\end{itemize}
In each of these cases, the function $g$ is in $\calF_M$, 
and it is equal to $h$ on some open neighbourhood of $y$ in $\Mcut$
(with respect to the given topology of $\Mcut$ as a quotient of $M$,
hence---by Lemma~\ref{Mcut:initial topology}---also with respect 
to the initial topology of $\Mcut$ that is induced from $\calF_M$).
Because $y$ was arbitrary, this implies that $h \colon W \to \R$ is in the 
subset differential structure that is induced from $\calF_M$.
Because $h$ was arbitrary, 
$\calF_W$ is contained in the subset differential structure on $\Wcut$
that is induced from $\calF_M$.
\end{proof}

\section{The cut space is a manifold}
\labell{sec:proof}


By Lemmas~\ref{descends}, \ref{composition}, and~\ref{identity},
the cutting construction gives a functor 
from the category of manifolds-with-boundary,
equipped with free circle actions near the boundary,
and their equivariant transverse maps,
to the category of topological spaces and their continuous maps.
By Corollary~\ref{FM is differential structure} 
and Lemma~\ref{descends smooth},
when each $\Mcut$ is equipped with the set of real valued functions $\calF_M$
of Construction~\ref{calF M},
the cutting construction defines a functor
to the category of differential spaces and their smooth maps.
In Proposition~\ref{Mcut:manifold}, 
which is a reformulation of Theorem~\ref{mfld str},
we use these functoriality properties
to give a precise proof that cutting yields smooth manifolds.


\begin{Proposition} \labell{Mcut:manifold}
Let $M$ be a manifold-with-boundary,
equipped with a free circle action 
on a neighbourhood $U_M$ of its boundary.
Then the cut space $\Mcut$,
equipped with the differential structure $\calF_M$
that is obtained from Construction~\ref{calF M}
(see Corollary~\ref{FM is differential structure}), is a manifold.

Moreover, 
$\intM_\cut := c(\intM)$ is an open dense subset of $\Mcut$,
and $\Mred := c(\del M) = (\del M) / S^1$
is a codimension two closed submanifold of $\Mcut$.
\end{Proposition}

\begin{proof}
By Lemma~\ref{Mcut:initial topology},
the initial topology on $\Mcut$ determined by $\calF_M$
coincides with the quotient topology on $\Mcut$ induced from $M$.
By Lemmas~\ref{Mcut Hausdorff} and~\ref{Mcut 2nd countable},
$\Mcut$---with this topology---is Hausdorff and second countable.

Let $y$ be a point in $\Mcut$.  Let $n = \dim M$.

\begin{itemize}
\item 
If $y = \{ x \}$ for $x \in \intM$, let $W$ be the domain
of a coordinate chart on $M$ that satisfies $x \in W \subset \intM$.
Then the differential space $\Wcut$ is diffeomorphic 
to an open subset of~$\R^n$.

\item 
If $y = S^1 \cdot x$ for $x \in \del M$,
let $W$ be an $S^1$-invariant open neighbourhood of $x$ in $U_M$
as in Proposition~\ref{prop:local}, and let
$$ W \to D^{n-2} \times S^1 \times [0,\eps) $$
be a diffeomorphism that intertwines the $S^1$ action on $W$
with the rotations of the middle factor
and that intertwines an invariant boundary-defining function
with the projection to the last factor.
Then the differential space $\Wcut$ is diffeomorphic
to $\big(D^{n-2} \times S^1 \times [0,\eps)\big)_\cut$,
which, as a consequence of Lemma~\ref{model},
is diffeomorphic to $D^{n-2} \times D^2$.
\end{itemize}

Thus, in each of these cases, $\Wcut$---with the differential 
structure $\calF_W$ that is obtained from Construction~\ref{calF M}---is 
diffeomorphic to an open subset of $\R^n$.
By Lemma~\ref{open subset}, 
the inclusion map of $\Wcut$ into $\Mcut$
is a diffeomorphism 
with an open neighbourhood of $y$ in $\Mcut$
(where $\Wcut$ is equipped with the differential structure $\calF_W$
and its image is equipped with the subset differential structure  
induced from $\calF_M$).
So $y$ has a neighbourhood that is diffeomorphic
to an open subset of $\R^n$.

Because $\intM$ is open and dense in $M$ and is saturated,
$\intM_\cut$ is open and dense in $\Mcut$.
Because each of the above diffeomorphisms $\Wcut \to D^{n-2} \times D^2$
takes the intersection $\Mred \cap \Wcut$ 
to the subset $D^{n-2} \times \{0\}$ of $D^{n-2} \times D^2$,
the subset $\Mred$ is a codimension two submanifold of $\Mcut$.
\end{proof}


\begin{Lemma} \labell{Mcut:submanifolds}
Consider $\Mcut$ with its manifold structure
that is obtained from $\calF_M$; see Proposition~\ref{Mcut:manifold}.
Then the quotient map $c \colon M \to \Mcut$
restricts to a diffeomorphism 
$$\intM \to \intM_\cut$$
and to a principal circle bundle map 
$$\del M \to \Mred , $$
where the domains of these maps are equipped with their manifold structures
as subsets of $M$,
and the targets of these maps are equipped
with their manifold structures as subsets of $\Mcut$.
\end{Lemma}

\begin{proof}
By Lemma~\ref{open subset}, $\intM_\cut$ is open in $\Mcut$,
and its manifold structure as a subset of $\Mcut$
coincides with its manifold structure that is obtained 
from the differential structure $\calF_{\intM}$
of Construction~\ref{calF M}.
Because $c|_{\intM} \colon x \mapsto \{ x \}$
is a bijection onto $\intM_\cut$, 
to show that the map $\intM \to \intM_\cut$ is a diffeomorphism,
we need to show that,
for any function $h \colon \intM_\cut \to \R$,
the function $g$ is in $\calF_{\intM}$
iff its pullback $\wh{h} \colon \intM \to \R$ is smooth.
This, in turn, follows from 
Condition~($\calF 1$) of Construction~\ref{calF M}.

As we already mentioned in Remark~\ref{decomposition},
there exists a unique manifold structure on $\Mred$ such that 
the quotient map $\del M \to \Mred$ is a principal $S^1$ bundle.
With this structure,
a real valued function on $\Mred$ is smooth
iff its pullback to $\del M$ is smooth.
We will show that, for any function $g \colon \Mred \to \R$,
the function $g$ is smooth on $\Mred$ as a submanifold of $\Mcut$ 
iff its pullback $\wh{g} := g \circ c \colon \del M \to \R$ is smooth.

Suppose that $g \colon \Mred \to \R$ is smooth on $\Mred$ 
as a submanifold of $\Mcut$.
Because this submanifold is closed,
the function $g$ extends to a smooth function $h \colon \Mcut \to \R$.
The pullback $\wh{h} \colon M \to \R$ then satisfies
the conditions of Construction~\ref{calF M},
and its restriction to $\del M$ coincides with $\wh{g}$.
Let $V$ be an invariant open neighbourhood of $\del M$ in $U_M$ and 
$$ H \colon V \times \C \to \R $$ 
be a smooth function that relates to $\wh{h}$ 
as in Condition~($\calF 2$) of Construction~\ref{calF M}.
Then $\wh{g}(x) = \wh{h}(x) = H(x,0)$ for all $x \in \del M$.
Because $H$ is smooth, $\wh{g}$ is smooth.

Conversely, suppose that $\wh{g} \colon \del M \to \R$ is smooth.
Let $V$ be an $S^1$-invariant neighbourhood of $\del M$ in $U_M$
and $\pi \colon V \to \del M$
an $S^1$-equivariant collar neighbourhood projection map.
Define $\wh{h} \colon V \to \R$ by $\wh{h}(x) := \wh{g}(\pi(x))$.
Then $\wh{h}$ descends to a function $h \colon \Vcut \to \R$
such that $h \circ c = \wh{h}$,
and the function $h$ is in the differential structure $\calF_V$.
(In Condition~($\calF 2$) of Construction~\ref{calF M} applied to $V$, 
we can take the function $H \colon V \times \C \to \R$ 
to be $H(x,z) := \wh{g}(\pi(x))$.)
By Lemma~\ref{open subset}, $\Vcut$ is open in $\Mcut$,
and its manifold structure as an open subset of $\Mcut$
coincides with its manifold structure that is obtained 
from the differential structure $\calF_V$.
So $h \colon \Vcut \to \R$ is smooth on $\Vcut$ 
as an open subset of $\Mcut$,
and so $g$, being the restriction of $h$ to the submanifold $\Mred$,
is smooth on $\Mred$
as a submanifold (of $\Vcut$, hence) of~$\Mcut$. 
\end{proof}

\begin{Remark} \labell{dependence on action}
The topology on $\Mcut$,
and the manifold structures on the pieces $\intM_\cut$ and $\Mred$, 
depend only on the circle action on the boundary $\del M$.
But different extensions of this circle action 
to a neighbourhood the boundary 
can yield manifold structures on $\Mcut$ that are diffeomorphic
but not equal.
For example, take the half-cylinder $M = S^1 \times [0,\infty)$
with the standard $S^1$-action.
Then, define a new $S^1$-action 
by conjugating the standard $S^1$-action by a diffeomorphism
of the form $\psi( b ,s) := ( b , s\, g(b)^2)$
where $g \colon S^1 \to \Rpos$ is some smooth function.
The function $\wh{h}(e^{i\theta},s) := g(e^{i\theta}) \sqrt{s} \cos\theta$
descends to a real valued function on $\Mcut$
that is smooth with respect to the new differential structure 
but is not smooth with respect to the standard differential structure
unless $g$ is constant.
(Indeed, the corresponding function on $\R^2$ satisfies 
$$ h(x,y) = x g(e^{i\theta}) $$
whenever $x=r\cos\theta$ and $y=r\sin\theta$ with $r \geq 0$.
So
$$ \lim\limits_{\substack{t \to 0 \\ t>0}} \frac{h(tx,ty)-h(0,0)}{t} 
   = x g(e^{i\theta}) $$
whenever $x=\cos\theta$ and $y=\sin\theta$.
If $h$ is smooth, then, by the chain rule, 
the limit on the left is equal to $ax+by$
where $a= \left.\frac{\del h}{\del x}\right|_{(0,0)}$ 
 and $b= \left.\frac{\del h}{\del y}\right|_{(0,0)}$, so
$$ ax+by = xg(e^{i\theta}) $$
whenever $x=\cos\theta$ and $y=\sin\theta$.
Substituting $(x,y) = (1,0)$ and $(x,y) = (0,1)$,
we obtain that $a=g(1)$ and $b=0$, and so
$$ ax = xg(e^{i\theta}) $$
whenever $x=\cos\theta$.
So $g(e^{i\theta})=a$ 
(when $\cos \theta \neq 0$, and hence, by continuity,) 
for all $e^{i\theta} \in S^1$.
\eor
\end{Remark}

\section{Cuttings with immersions, submersions, embeddings}
\labell{sec:submanifolds}

In this section, we cut submanifolds, and a bit more.
Here is a precise statement.

\begin{Lemma} \labell{immersion etc}
Let $M$ and $N$ be manifolds-with-boundary,
equipped with free circle actions on neighbourhoods of the boundary. 
Let $\Mcut$ and $\Ncut$ be obtained from $M$ and $N$
by the smooth cutting construction.
Let $\psi \colon M \to N$ be an equivariant transverse map
and $\psi_\cut \colon \Mcut \to \Ncut$
the resulting smooth map of the cut spaces.
\begin{itemize}
\item
If $\psi$ is an immersion, so is $\psi_\cut$.
\item
If $\psi$ is a submersion, so is $\psi_\cut$.
\item
If $\psi$ is an embedding, so is $\psi_\cut$.
\end{itemize}
\end{Lemma}

\begin{proof}
Let $x_M$ be a point in $M$. 
Let $x_N$ be its image in $N$,
and let $y_M$ and $y_N$ be the images of $x_M$ and $x_N$
in $\Mcut$ and $\Ncut$, so that $\psi_\cut (y_M) = y_N$.

\begin{itemize}
\item
Suppose that $x_M \in \intM$. 
Then $x_N \in \intN$ (see Lemma~\ref{descends}). 
Since the quotient maps $c_M|_{\intM} \colon \intM \to \intM_\cut$ 
and $c_N|_{\intN} \colon \intN \to \intN_\cut$ are diffeomorphisms
(see Lemma~\ref{Mcut:submanifolds}),
$d\psi|_{x_M}$ is injective (resp., surjective)
iff $\left.d\psi_\cut\right|_{y_M}$ is injective (resp., surjective).

\item
Suppose that $x_M \in \del M$.
Equip $N$ with any invariant boundary defining function $f$,
and equip $M$ with the boundary defining function $f \circ \psi$.
Proposition~\ref{prop:local} implies that
we can identify an open neighbourhood of $x_M$ in~$M$
with $U \times S^1 \times [0,\eps)$ where $U$ is open in $\R^{n-2}$,
and an open neighbourhood of $x_N$ in~$N$
with $V \times S^1 \times [0,\eps)$ where $V$ is open in $\R^{k-2}$,
such that the map $\psi$ takes the form
$$ \psi(x,a,s) = \left( \, \ol{\psi}(x,s) , \, a \, b(x,s) , \, s \, \right)$$
for some smooth functions 
$${\ol{\psi} \colon U \times [0,\eps) \to V}
\quad \text{ and } \quad
{b \colon U \times [0,\eps) \to S^1}.$$
Expressing $\psi$ as the composition
of the diffeomorphism 
$$ (x,a,s) \mapsto (x,ab(x,s),s) $$
with the map 
$$ (x,a,s) \mapsto (\ol{\psi}(x,s),a,s),$$
we see that 
$\left.d\psi\right|_{x_M}$ is injective (resp., surjective)
iff $\left.d\ol{\psi}(\cdot,0)\right|_{\ol{x}_M}$ is injective 
(resp., surjective),
where $\ol{x}_M$ is the corresponding point of $U$.

By Lemmas~\ref{model} and~\ref{open subset}, we can further identify 
an open neighbourhood of $y_M$ in $\Mcut$ with $U \times D^2$
and an open neighbourhood of $y_N$ in $\Ncut$ with $V \times D^2$,
where $D^2$ is the open disc of radius $\sqrt{\eps}$ 
about the origin in $\R^2$,
and the map $\psi_\cut$ becomes
$$ \psi_\cut (x,z) 
   = \left( \, \ol{\psi}(x,|z|^2) , \, z \, b(x,|z|^2) \right ) .$$
Expressing $\psi_\cut$ as the composition
of the diffeomorphism 
$$ (x,z) \mapsto ( \, x, \, z\, b(x,|z|^2) \, ) $$
with the map 
$$ (x,z) \mapsto (\,\ol{\psi}(x,|z|^2),\,z\,),$$
we see that $\left.d\psi_\cut\right|_{y_M}$ too is injective 
(resp., surjective)
iff $\left.d\ol{\psi}(\cdot,0)\right|_{\ol{x}_M}$ is injective 
(resp., surjective).

We conclude that $\left.d\psi\right|_{x_M}$ is injective (resp., surjective)
iff $\left.d\psi_\cut\right|_{y_M}$ is injective (resp., surjective).
\end{itemize}

We conclude that if $\psi$ is an immersion, then so is $\psi_\cut$, 
and if $\psi$ is a submersion, then so is $\psi_\cut$.
By the local immersion theorem\footnote{
I've adopted this name from John Lee~\cite{JohnLee:smooth}
},
being an embedding is equivalent to being an immersion
and a topological embedding (namely, a homeomorphism with its image).
By Lemma~\ref{descends}, if $\psi$ is a topological embedding,
then so is $\psi_\cut$.
We conclude that if $\psi$ is an embedding, then so is $\psi_\cut$.
\end{proof}

Below, \emph{submanifold-with-boundary} refers to a subset
that, with the subset differential structure,
is a manifold-with-boundary.

\begin{Corollary} \labell{submanifold}
Let $N$ be a manifold-with-boundary,
equipped with a free circle action near the boundary,
and let $M$ be a submanifold-with-boundary of $N$.
Suppose that the boundary of $M$ is contained in the boundary of $N$,
that the intersection of $M$ with some invariant neighbourhood of $\del N$
is invariant, and that the inclusion map of $M$ in $N$ 
is an equivariant transverse map.
Then $\Mcut$ is a submanifold of $\Ncut$.
\end{Corollary}

\section{Cutting with differential forms}
\labell{sec:diff forms}

In this section,
we give a criterion for a differential form to descends to a cut space.

Throughout this section,
let $M$ be a manifold-with-boundary equipped with a free circle action 
on a neighbourhood $U_M$ of its boundary,
let $\Mcut$ be obtained from $M$ by the smooth cutting construction,
and let $c \colon M \to \Mcut$ be the quotient map.

\begin{Lemma} \labell{cut diff form}
Let $\beta$ be a differential form on $M$
that is basic on $\del M$ and invariant near $\del M$.
Then there exists a unique differential form $\beta_\cut$ on $\Mcut$
whose pullback through $c|_{\intM} \colon \intM \to \intM_\cut$
coincides with~$\beta$ on $\intM$.
\end{Lemma}

\begin{Remark*}
In Lemma~\ref{cut diff form}, the assumption on $\beta$ is that 
its pullback to $\del M$ is $S^1$ basic
and its restriction to some invariant neighbourhood 
of~$\del M$ in $U_M$ is $S^1$ invariant.
\end{Remark*}

We summarize Lemma~\ref{cut diff form} in the following diagram,
which encodes pullbacks of differential forms. 
(Note that we do not obtain $\beta$ is a pullback of $\beta_\cut$
because the map $M \to \Mcut$ is not smooth.)
$$ \xymatrix{
(M,\beta) && 
 \ar[ll]_{\text{inclusion}} \ar[d]_{c|_{\intM}}
 (\intM,\beta|_{\intM}) \\
(\Mcut,\beta_\cut) && 
 \ar[ll]_{\text{inclusion}} 
 (\intM_\cut, {\beta_\cut}|_{\intM_\cut})
}$$

\begin{proof}[Proof of Lemma~\ref{cut diff form}]
Because $\intM_\cut$ is open and dense in $\Mcut$
and the map $c|_{\intM} \colon \intM \to \intM_\cut$ 
is a diffeomorphism, 
it is enough to show that such a form $\beta_\cut$ exists 
locally near each point of $\intM_\cut$.
Indeed, by continuity, such local forms are unique 
and patch together into a form on $\Mcut$ as required.
Because $\Mcut = \Mred \sqcup \intM_\cut$,
it is enough to consider neighbourhoods of points in $\Mred$.
Thus, we may assume that $\beta$ is invariant 
and (by Proposition~\ref{prop:local} and Lemma~\ref{open subset}) that 
$$ M = D^{n-2} \times S^1 \times [0,\eps) .$$
We write the components of a point in $M$ as $x \in D^{n-2}$, 
$a = e^{i\theta} \in S^1$, and $s \in \R_{\geq 0}$.  
Let $k$ be the degree of $\beta$.
The assumption on $\beta$ implies (by Hadamard's lemma) that we can write 
$$ \beta \, = \, \beta_k 
         \, + \, \beta_{k-1} \wedge ds
         \, + \, s \hat{\beta}_{k-1} \wedge d\theta
         \, + \, \beta_{k-2} \wedge ds \wedge d\theta \, ,$$
where for each $\ell$
$$ \beta_\ell = \sum_{\substack{I = (i_1,\ldots,i_{\ell}) \\ 
                                i_1 < \ldots < i_\ell }} 
   b_I(x,s) dx_{i_1} \wedge \ldots \wedge dx_{i_\ell} $$
for some smooth functions $b_I(x,s)$ on $D^{n-2} \times [0,\eps)$,
and similarly for~$\hat{\beta}_\ell$.
(If $\ell < 0$, the sum is empty and $\beta_\ell = 0$.)

Identify $\Mcut$ with $D^{n-2} \times D^2$ as in Lemma~\ref{model},
with coordinates $x$ and $z = u+iv$.
The inverse of the diffeomorphism 
$c|_{\intM} \colon \intM \to \intM_\cut$
is then given by 
$(x,z) \mapsto (x,a,s)$ with $a = z/|z|$ and $s=|z|^2$.
Writing 
$$ ds = (2udu + 2vdv), \quad  sd\theta = udv - vdu, \quad
   ds \wedge d\theta = 2du \wedge dv \, , $$ 
and 
$$ (\beta_\ell)_\cut = \sum_{\substack{I = (i_1,\ldots,i_{\ell})
 \\ i_1 < \ldots < i_\ell}} 
   b_I(x,|z|^2) dx_{i_1} \wedge \ldots \wedge dx_{i_\ell} ,$$
and similarly for $\hat{\beta}_\ell$, 
we can then take
\begin{multline*}
\beta_\cut \, := \, (\beta_k)_\cut 
              \, + \, (\beta_{k-1})_\cut \wedge (2udu + 2vdv) \\
              \, + \, (\hat{\beta}_{k-1})_\cut \wedge (udv-vdu)
              \, + \, (\beta_{k-2})_\cut \wedge (2du \wedge dv) \, .
\end{multline*}
\end{proof}

We now give some important properties of the cutting procedure
for differential forms.

\begin{Lemma} \labell{d and wedge}
Let $\beta$ and $\beta'$ be differential forms on $M$
that are basic on~$\del M$ and invariant near~$\del M$.  Then 
$$ (d\beta)_\cut = d(\beta_\cut) \quad \text{ and } \quad
 (\beta \wedge \beta')_\cut = \beta_\cut \wedge \beta'_\cut.$$
\end{Lemma}

\begin{proof}
Because 
$c|_{\intM} \colon \intM \to \intM_\cut$
is a diffeomorphism that takes $\beta$ to $\beta_\cut$
and $\beta'$ to $\beta'_\cut$,
these equalities hold on $\intM_\cut$.
By continuity, they hold on all of $\Mcut$.
\end{proof}

\begin{Lemma} \labell{closed}
Let $\beta$ be a differential form on $M$
that is basic on $\del M$ and invariant near $\del M$.
Then $\beta$ is closed on $M$ 
if and only if $\beta_\cut$ is closed on $\Mcut$.
\end{Lemma}

\begin{proof}
Because 
$c|_{\intM} \colon \intM \to \intM_\cut$
is a diffeomorphism that takes $\beta$ to $\beta_\cut$,
\ $\beta$ is closed on $\intM$
if and only if $\beta_\cut$ is closed on $\intM_\cut$.
By continuity,
$\beta$ is closed on~$M$ if and only if $\beta_\cut$
is closed on~$\Mcut$.
\end{proof}

\begin{Lemma} \labell{nonzero}
$\beta \neq 0$ at a point $p$ of $M$
if and only if $\beta_\cut \neq 0$ at the point $c(p)$ of $\Mcut$.
\end{Lemma}

\begin{proof}
Because 
$c|_{\intM} \colon \intM \to \intM_\cut$
is a diffeomorphism that takes $\beta$ to $\beta_\cut$,
this is true at points of $\intM$.
Near a point in $\del M$,
we may identify $M$ with $D^{n-2} \times S^1 \times [0,\eps)$.
In the expression in coordinates for $\beta$ and $\beta_\cut$
that appear in the proof of Lemma~\ref{cut diff form},
we see that the non-vanishing of $\beta$
at any point of $M$
is equivalent to the non-vanishing of at least one of 
$\beta_k$, $\beta_{k-1}$, or $\beta_{k-2}$ at that point,
and that the non-vanishing of $\beta_\cut$ at any point of $\Mcut$
is equivalent to the non-vanishing of at least one of 
$(\beta_k)_\cut$, $(\beta_{k-1})_\cut$, or $(\beta_{k-2})_\cut$
at that point.
For a point $p$ in $\del M$ and its image $c(p)$ in $\Mcut$,
the coordinates of $p$ are $x,\theta,s$ with $s=0$, 
and the coordinates of $c(p)$ are $x,z$ with the same $x$ 
and with $z=0$.
We finish by noting that $\beta_\ell|_{s=0}$
and $\left.(\beta_\ell)_\cut\right|_{z=0}$ 
are both given by the same expression,
$$ \sum\limits_{\substack{I=(i_1,\ldots,i_\ell)\\i_1<\ldots<i_\ell}}
 b_I(x,0) dx_{i_1} \wedge \ldots \wedge dx_{i_\ell} \, . $$ 
\end{proof}

\begin{Corollary} \labell{symplectic and contact} \
In the setup of Lemma~\ref{cut diff form}, the following holds.
\begin{itemize}
\item
$\beta$ is a symplectic two-form on $M$
iff $\beta_\cut$ is a symplectic two-form on $\Mcut$.
\item
$\beta$ is a contact one-form on $M$ 
iff $\beta_\cut$ is a contact one-form on~$\Mcut$.
\end{itemize}
\end{Corollary}

\begin{proof}
First,
assume that $M$ has dimension~$2n$ and that $\beta$ is a two-form.
By Lemmas~\ref{d and wedge} and~\ref{nonzero},
$\beta$ is closed on $M$ iff $\beta_\cut$ is closed on $\Mcut$,
and $\beta^n$ is non-vanishing on $M$ 
iff $\beta_\cut^n$ is non-vanishing on $\Mcut$.
This gives the first result.

Next, 
assume that $M$ has dimension $2n+1$ and that $\beta$ is a one-form.
By Lemmas~\ref{d and wedge} and~\ref{nonzero},
$\beta \wedge (d\beta)^n \neq 0$ on $M$
iff $\beta_\cut \wedge (d\beta_\cut)^n \neq 0$ on $\Mcut$.
This gives the second result.
\end{proof}

\section{Relation with reduced forms}
\labell{sec:red}

Let $M$ be a manifold-with-boundary, 
equipped with a free circle action near its boundary, 
let $\Mcut$ be obtained from $M$ by the smooth cutting construction,
and let $c \colon M \to \Mcut$ be the quotient map.
Let $\Mred := c(\del M) = (\del M) / S^1$ in $\Mcut$.
So we have a commuting square
$$ \xymatrix{ \del M \ar[rr]^{\text{inclusion}}\ar[d]_{\text{quotient}}
   && M \ar[d]^{\text{quotient}} \\
   \Mred \ar[rr]^{\text{inclusion}}  && \Mcut \, ,}$$
(in which the map $M \to \Mcut$ is not smooth).


Here is a quick reminder on basic forms.
Let $N$ be a manifold with a free circle action,
generated by a vector field $\xi_N$.
A differential form $\beta_N$ on $N$ is \textbf{basic}
if it is horizontal, which means that $\xi_N \contract N = 0$,
and invariant.  This holds iff $\beta_N$ is the pullback to $N$
of a differential form on $N/S^1$.
In this situation,
the differential form on $N/S^1$ is determined uniquely by $\beta$.

So if $\beta$ is a differential form on $M$ that is basic on $\del M$
and invariant near $\del M$, then there exists a unique differential form
$\beta_\red$ on $\Mred$ whose pullback $\beta_{\del M}$ to $\del M$
coincides with the pullback of $\beta$ to $\del M$.
We summarize this in the following diagram.
$$ \xymatrix{
 (\del M,\beta_{\del M}) 
 \ar[rr]^{\text{inclusion}} \ar[d]_{\text{quotient}} 
 && (M,\beta) \\
 (\Mred,\beta_\red) && 
} $$

In Lemma~\ref{red} we show that $\beta_\red$ coincides
with the pullback of $\beta_\cut$
under the inclusion map $\Mred \to \Mcut$.
This result does not follow immediately from the above 
commuting square,
because the quotient map $M \to \Mcut$ is not smooth.
We summarize this result in the following diagram.
$$ \xymatrix{
 (\del M,\beta_{\del M}) 
 \ar[rr]^{\text{inclusion}} \ar[d]_{\text{quotient}} 
 && (M,\beta) 
 && \ar[ll]_{\text{inclusion}} (\intM,\beta|_{\intM}) \ar[d]^{c|_{\intM}}\\
 (\Mred,\beta_\red) \ar[rr]^{\text{inclusion}} 
 && (\Mcut,\beta_\cut) 
 && \ar[ll]_{\text{inclusion}} (\intM_\cut , \beta_\cut|_{\intM_\cut})
} $$

In Corollary~\ref{symplectic and contact}
we showed that if $\beta$ is a symplectic two-form 
(resp., a contact one-form) on $\Mcut$
then $\beta_\cut$
is a symplectic two-form (resp., a contact one-form) on $\Mcut$.
In Lemmas~\ref{symplectic on red} and~\ref{contact on red}
we will now show that, in these situations,
$\beta_\red$ is a symplectic two-form
(resp., a contact one-form) on $\Mred$.
Thus, in these situations, $\Mred$ is a symplectic (resp., contact)
submanifold of $\Mcut$.

\begin{Lemma} \labell{red}
Let $\beta$ be a differential form on $M$
that is basic on $\del M$ and invariant near $\del M$.
Let $\beta_\red$ and $\beta_\red$
be the induced differential forms on $\Mred$ and $\Mcut$ 
as described above.
Then $\Mred$ coincides with the pullback of $\beta_\cut$
under the inclusion map $\Mred \to \Mcut$.
\end{Lemma}

\begin{proof}
As in the proof\footnote
{
we violate the principle that one should refer to statements,
not to proofs
}
of Lemma~\ref{cut diff form}, 
we may assume that $M$ is $D^{n-2} \times S^1 \times [0,\eps)$,
with coordinates $x_i$, $\theta$, and $s$.

By setting $s=0$ in the expression in coordinates for $\beta$ 
in the proof of Lemma~\ref{cut diff form},
we obtain that the restriction of $\beta$ to $\del M$ is
\begin{multline*}
 \left.\beta\right|_{\del M} = 
\sum\limits_{\substack{I = (i_1,\ldots,i_k) \\ i_1 < \ldots < i_k }} 
            b_I(x,0) dx_{i_1} \wedge \ldots \wedge dx_{i_k}  \\
 + \beta_{k-1}|_{s=0} \wedge ds
 + \beta_{k-2}|_{s=0} \wedge ds \wedge d\theta .
\end{multline*}
Further setting $ds=0$, we obtain that the pullback of $\beta$
to $\del M$ is
$$ \beta_{\del M} = 
\sum\limits_{\substack{I = (i_1,\ldots,i_k) \\ i_1 < \ldots < i_k }} 
            b_I(x,0) dx_{i_1} \wedge \ldots \wedge dx_{i_k}  $$
as a differential form on $D^{n-2} \times S^1 \times \{ 0 \}$.
Identifying $\Mred$ with $D^{n-2}$, we obtain that
$$ \beta_\red = 
\sum\limits_{\substack{I = (i_1,\ldots,i_k) \\ i_1 < \ldots < i_k }} 
            b_I(x,0) dx_{i_1} \wedge \ldots \wedge dx_{i_k}  $$
as a differential form on $D^{n-2}$.

By setting $z=u+iv=0$ in the expression in coordinates for $\beta_\cut$
in the proof of Lemma~\ref{cut diff form},
we obtain that the restriction of $\beta_\cut$ to $\Mred$ is
$$ \left.\beta_\cut\right|_{\Mred}
 = \sum\limits_{\substack{I = (i_1,\ldots,i_k)
 \\ i_1 < \ldots < i_k }} 
            b_I(x,0) dx_{i_1} \wedge \ldots \wedge dx_{i_k} 
 + (\beta_{k-2})_\cut \wedge (2du \wedge dv). $$

Further setting $du=dv=0$, we obtain that the pullback of $\beta_\cut$
to $\Mred$ coincides with $\beta_\red$, as required.
\end{proof}

We now specialize to the context 
of symplectic cutting and contact cutting.
In Corollary~\ref{symplectic and contact}
we showed that, in the setting of Lemma~\ref{cut diff form}, 
if $\beta$ is a symplectic two-form (resp., contact one-form) on $M$ 
then $\beta_\cut$ is a symplectic two-form (resp., contact one-form)
on $\Mcut$.
In Lemmas~\ref{symplectic on red} and~\ref{contact on red}
we will show that, in this setting,
if $\beta$ is a symplectic two-form (resp., contact one-form) on $M$ 
then $\beta_\red$ is a symplectic two-form 
(resp., contact one-form) on $\Mred$. 

\begin{Lemma} \labell{symplectic on red}
Let $\beta$ be a symplectic two-form on $M$
that is basic on~$\del M$ and invariant near~$\del M$.
Then $\beta_\red$ is a symplectic two-form on~$\Mred$.
\end{Lemma}

\begin{proof}
Because $\beta$ is closed, its pullback $\beta_{\del M}$ to $\del M$ 
is closed, and so $\beta_\red$ is closed.
It remains to prove that $\beta_\red$ is non-degenerate.

Let $2n = \dim M$.

As in the proofs of Lemmas~\ref{cut diff form}
and~\ref{red},
we may assume that $M = D^{2n-2} \times S^1 \times [0,\eps)$
and $\Mred = D^{2n-2}$, and we obtain the expressions 
\begin{multline*}
 \beta \, = \, \sum\limits_{i<j} b_{i,j} (x,s) dx_i \wedge dx_j \\
 \, + \, \sum_i b_i(x,s) dx_i \wedge ds 
 \, + \, s \sum_i \wh{b}_i(x,s) dx_i \wedge d\theta
 \, + \, b_0(x,s) ds \wedge d\theta  
\end{multline*}
and 
$$ \beta_\red = \sum_{i<j} b_{i,j}(x,0) dx_i \wedge dx_j \, ,$$
where $b_{i,j}$, $b_i$, $\wh{b}_i$, and $b_0$
are smooth functions on $D^{2n-2} \times [0,\eps)$.
By setting $s=0$ in the expression for $\beta$, we obtain
that the restriction of $\beta$ to $\del M$ is
\begin{multline*}
\beta|_{\del M} 
 \, = \, \sum\limits_{i<j} b_{i,j} (x,0) dx_i \wedge dx_j \\
    \, + \, \sum_i b_i(x,0) dx_i \wedge ds
    \, + \, b_0(x,0) ds \wedge d\theta \, .
\end{multline*}
Taking the $n$th wedge
and noting that the number of $x$-coordinates is $2n-2$, 
and then comparing with the expression for $\beta_\red$, we obtain
\begin{align*}
\beta^n|_{\del M} 
 & = n \Big(\sum\limits_{i<j} b_{i,j} (x,0) dx_i \wedge dx_j\Big)^{n-1}
   \wedge \Big( b_0(x,0) ds \wedge d\theta \Big) \\
 & = n \Big( \pi^* \beta_\red \Big)^{n-1}
   \wedge \Big( b_0(x,0) ds \wedge d\theta \Big) 
\end{align*}
where $\pi \colon D^{2n-2} \times S^1 \times [0,\eps) \to D^{2n-2}$
is the projection map.
The non-vanishing of this form along $\del M$ implies 
that $\beta_\red^{n-1}$ is non-vanishing
and hence that $\beta_\red$ is non-degenerate.
\end{proof}

\begin{Lemma} \labell{contact on red}
Let $\beta$ be a contact one-form on~$M$
that is basic on~$\del M$ and invariant near~$\del M$.
Then $\beta_\red$ is a contact one-form on~$\Mred$.
\end{Lemma}

\begin{proof}
Let $2n+1 = \dim M$.

As in the proofs of Lemmas~\ref{cut diff form} and~\ref{red},
we may assume that $M = D^{2n-1} \times S^1 \times [0,\eps)$
and $M_\red = D^{2n-1}$,
and we obtain the expressions
$$ \beta \, = \, \sum_i b_i(x,s) dx_i 
              \, + \, b_0(x,s) ds \, + \, s\wh{b}_0(x,s) d\theta $$
and
$$ \beta_\red \, = \, \sum_i b_i(x,0) dx_i .$$
From the expression for $\beta$ we get
\begin{multline*}
d\beta = 
 \sum_{i,j} \frac{\del b_i}{\del x_j}(x,s) dx_j \wedge dx_i 
   \, + \, \sum_i \frac{\del b_i}{\del s}(x,s) ds \wedge dx_i \\
   \, + \, \sum_i \frac{\del b_0}{\del x_i}(x,s) dx_i \wedge ds 
   \, + \, \wh{b}_0(x,0) ds \wedge d\theta 
   \, + \, s \eta
\end{multline*}
for some two-form $\eta$ on $D^{2n-1} \times S^1 \times [0,\eps)$.
By setting $s=0$ in the expressions for $\beta$ and for $d\beta$,
we obtain that their restrictions to $\del M$ are
$$ \beta|_{\del M} \, = \, \sum_i b_i(x,0) dx_i \, + \, b_0(x,0) ds $$
and
$$
 (d \beta)|_{\del M} \, = \, 
 \sum_{i,j} \frac{\del b_i}{\del x_j}(x,0) dx_j \wedge dx_i
 + \gamma \wedge ds \, 
$$
for some one-form $\gamma$ on $D^{2n-1} \times S^1 \times [0,\eps)$
along $\{s=0\}$.
Taking the $n$th wedge and noting that the number of $x$-coordinates
is~$2n-1$, 
we obtain
$$
(d\beta)^n|_{\del M} 
 \, = \, n \Big( \sum_{i,j} \frac{\del b_i}{\del x_j}(x,0) 
    dx_j \wedge dx_i \Big)^{n-1}
\wedge \gamma \wedge ds \, .
$$
Taking the wedge of the expression for $\beta|_{\del M}$
and the expression for $(d\beta)^n|_{\del M}$,
and then comparing with the expression for $\beta_\red$, we obtain
\begin{multline*}
\beta \wedge (d\beta)^n|_{\del M}  \\
   \, = \, n \Big( \sum_i b_i(x,0) dx_i \Big) 
   \wedge \Big( 
   \sum_{i,j} \frac{\del b_i}{\del x_j}(x,0) dx_j \wedge dx_i \Big)^{n-1}
 \wedge \gamma \wedge ds   \\
   \, = \, n \Big( \pi^* \beta_\red \Big) 
       \wedge \Big( \pi^* d \beta_\red \Big)^{n-1} 
 \wedge \gamma \wedge ds 
\end{multline*}
where $\pi \colon D^{2n-1} \times S^1 \times [0,\eps) \to D^{2n-1}$
is the projection map.
The non-vanishing of this form along $\del M$ implies that
$\beta_\red \wedge (d\beta_\red)^{n-1}$ is nonvanishing
and hence that $\beta_\red$ is a contact one-form.
\end{proof}

\section{Relation with symplectic cutting}
\labell{sec:symplectic cutting}

We now relate the smooth cutting procedure with differential forms
to the classical version of Lerman's symplectic cutting procedure.

Let $\wt{M}$ be a manifold equipped with a circle action
and with an invariant invariant two-form $\omega$.
Let $\mu \colon M \to \R$ be a corresponding momentum map;
this means that 
$$ \xi_M \contract \omega = - d\mu ,$$
where $\xi_M$ is the vector field that generates the circle action.
It implies that $\mu$ is invariant.
Suppose that the circle action on the zero level set $\mu^{-1}(\{0\})$
is free.
The level set $\mu^{-1}(\{0\})$ is regular
iff $\xi_M$ is not in the null-space of $\omega$ 
at any point of $\mu^{-1}(\{0\})$.
Assume that this holds;
note that it always holds if $\omega$ is non-degenerate. 
Then
$$ M := \mu^{-1}([0,\infty))$$
is a submanifold-with-boundary of $\wt{M}$,
its boundary is $\del M = \mu^{-1}(\{0\})$,
and $\mu|_M$ is an invariant boundary defining function.

Let $\Mcut$ be obtained from $M$ by the smooth cutting construction.
The equation $ \xi_M \contract \omega = - d\mu $
implies that the pullback of $\omega$ to $\del M$ is basic.
By Lemmas~\ref{cut diff form} and~\ref{closed},
$\omega$ induces a closed two-form $\omega_\cut$ on $\Mcut$.
By Corollary~\ref{symplectic and contact},
if $\omega$ is symplectic, so is $\omega_\cut$.

Suppose that a Lie group $G$ acts smoothly on $M$, 
commutes with the circle action, and preserves $\omega$.
By functoriality, the action descends to a smooth $G$ action on $\Mcut$.
This action preserves $\omega_\cut$ 
on $\intM_\cut$, hence (by continuity) everywhere.

Suppose now that the $G$ action on $M$ is Hamiltonian 
with momentum map $\mu_G$.
Because the $G$ action preserves $\omega$, commutes with the $S^1$ action,
and preserves the zero level set of $\mu$, it preserves $\mu$ near $\del M$;
this implies that the $S^1$ action preserves $\mu_G$ near $\del M$.
By Lemma~\ref{cut diff form},
$\mu_G$ descends to a smooth map $(\mu_G)_\cut$ on $\Mcut$.
The momentum map equation for $G$
holds on $\intM_\cut$ and hence (by continuity) everywhere.
So $(\mu_G)_\cut$ is a momentum map for the $G$ action
on $(\Mcut,\omega_\cut)$.

{
\section{Relation with contact cutting}
\labell{sec:contact cutting}

We now relate the smooth cutting procedure with differential forms
to the classical version of Lerman's contact cutting procedure.
We give here a version that should be useful for our paper-in-progress 
\cite{chiang-karshon}.

Let $\wt{M}$ be a manifold equipped with a circle action
and with an invariant contact one-form $\beta$.
Consider the corresponding momentum map,
$$ \mu := \xi_M \contract \beta \colon M \to \R .$$ 
Assume that the circle action is free 
on the zero level set $\mu^{-1}(\{0\})$.
In particular, $\xi_M$ is non-vanishing along $\mu^{-1}(\{0\})$.
Because $\xi_M$ is in $\ker \beta$ along $\mu^{-1}(\{0\})$,
and because restriction of $d\beta$ to $\ker \beta$ is non-degenerate
(because $\beta$ is a contact one-form),
it follows that 
$\xi_M \contract d\beta$ is non-vanishing along $\mu^{-1}(\{0\})$.
But $\xi_M \contract d\beta = - d \xi_M \contract \beta = - d\mu$,
so the level set $\mu^{-1}(\{0\})$ is regular,
$$ M := \mu^{-1}([0,\infty)) $$
is a submanifold-with-boundary of $\wt{M}$,
its boundary is $\del M = \mu^{-1}(\{0\})$,
and $\mu|_M$ is an invariant boundary defining function on $M$.

More generally, let $M$ be a manifold with boundary,
equipped with a free circle action near the boundary $\del M$,
and let $\beta$ be a contact one-form on $M$
that is invariant near $\del M$.
Assume that the corresponding momentum map
$$ \mu := \xi_M \contract \beta \colon M \to \R $$ 
vanishes along $\del M$.
Then $\mu$ is an invariant boundary defining function near $\del M$.

Let $\Mcut$ be obtained from $M$ by the smooth cutting construction.
Because $\xi_M \contract \beta = \mu$ vanishes along $\del M$
and is $S^1$ invariant near $\del M$,
the differential form $\beta$ is basic on $\del M$.
By Lemma~\ref{cut diff form} and Corollary~\ref{symplectic and contact},
$\beta$ induces a contact one-form $\beta_\cut$ on $\Mcut$.
By Lemmas~\ref{red} and~\ref{contact on red},
the pullback of $\beta_\cut$ 
under the inclusion map $\Mred \to \Mcut$
is the contact one-form on $\Mred$,
so $\Mred$ is a contact submanifold of $\Mcut$.

%

\ynote{Where do we differ from Eugene}


}

\section{Cutting with distributions}
\labell{sec:distributions}

We begin with a quick reminder on distributions.

A \textbf{distribution} on a manifold $N$
is a sub-bundle $E$ of the tangent bundle $TN$.
Lie brackets of vector fields determines a $TN/E$-valued two-form on $E$,
which we write as
$$ \Omega \colon E \times E \to TN/E :$$
for any two vector fields $u,v$ with values in $E$,
we have the equality $[u,v] + E = \Omega(u,v)$ of sections of $TN/E$.
The distribution is \textbf{involutive} if this two-form is zero;
it's \textbf{contact} if $E$ is of codimension one
and this two-form is non-degenerate.

A subset $E$ of $TN$ is a codimension-$k$ distribution 
iff can locally be written as the null-space 
of a non-vanishing decomposable $k$-form $\beta$.
(Namely, $E = \{ v \in TM \ | \ v \contract \beta = 0 \}$, \, 
$\beta = \beta_1 \wedge \ldots \wedge \beta_k$ for some one-forms $\beta_j$,
and $\beta \neq 0$.)
A subset $E$ of $TN$ is a contact distribution
iff $\dim N$ is odd, say, $\dim N = 2n+1$,
and $E$ can locally be written as the null-space 
of a one-form $\beta$ such that $\beta \wedge (d\beta)^n\neq 0$.
If this holds, then $\beta \wedge (d\beta)^n \neq 0$
for every 
local one-form $\beta$ whose null-space is $E$

\begin{Lemma}\labell{cut distribution}
Let $M$ be a manifold-with-boundary, with a free circle action
on a neighbourhood $U_M$ of the boundary $\del M$.
Let $\Mcut$ be obtained from $M$ by the smooth cutting construction,
and let $c \colon M \to \Mcut$ be the quotient map.
Let $E$ be a distribution on $M$.
Assume that $E$ contains the tangents to the $S^1$ orbits along $\del M$,
is $S^1$-invariant near $\del M$, and is transverse to $\del M$.  
Then there exists a unique distribution $\Ecut$ on $\Mcut$
whose preimage under $c|_{\intM} \colon \intM \to (\intM)_\cut$
is $E|_{\intM}$. 
Moreover,
%
%
\begin{itemize}
\item $E$ is involutive iff $\Ecut$ is involutive; and
\item $E$ is a contact distribution iff $\Ecut$ is a contact distribution.
\end{itemize}
\end{Lemma}

\begin{center}
\begin{figure}[h]
\includegraphics[scale=.75]{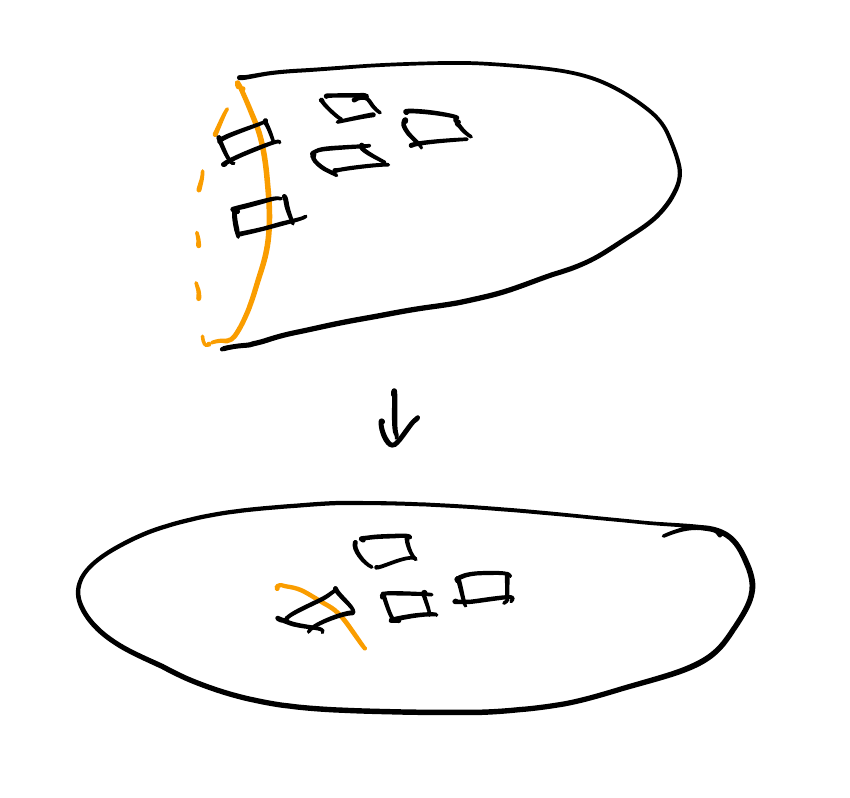}
\caption{Cutting with distributions}
\end{figure}
\end{center}

We resist our temptation to define $\Ecut$
to be the preimage of $E$ under the differential
of the quotient map $c \colon M \to \Mcut$:
this map $c$ is not smooth.

Here is a proof of the lemma.

\begin{proof}[Proof of Lemma~\ref{cut distribution}] \ 

Uniqueness follows from continuity.
So does the statement about involutivity.

It is enough to check that a distribution $\Ecut$ 
with the required properties exists locally near each point of $\Mcut$:
by continuity, such local distributions 
patch together into a distribution on $\Mcut$ as required.
Because $c$ restricts to a diffeomorphism from $\intM$ to $(\intM)_\cut$,
it is enough to consider neighbourhoods of points 
in $\Mred := (\del M)_\cut \, ( = (\del M)/S^1) $.

Assume, first, that $E$ has codimension one.
The assumptions on $E$ imply that near each orbit in $\del M$
we can write $E = \ker \beta$
where $\beta$ is an $S^1$-invariant one-form
whose pulback to $\del M$ is $S^1$-basic and non-vanishing.
By Lemmas~\ref{cut diff form} and~\ref{nonzero},
we obtain a differential form $\beta_\cut$ on $\Mcut$
such that $E_\cut := \ker \beta_\cut$ is a distribution 
whose preimage under $c|_{\intM}$ is $E|_{\intM}$.

By Corollary~\ref{symplectic and contact},
$\beta$ is a contact form on $M$
iff $\beta_\cut$ is a contact form on $\Mcut$.
So $E$ is a contact distribution on $M$
iff $E_\cut$ is a contact distribution on $\Mcut$.

\bigskip

Now, let $E$ be a distribution of any codimension.
Let $k$ be the codimension of the distribution $E$.

The assumptions on $E$ imply that $E \cap T \del M$
is a codimension $k$ distribution on $\del M$ 
that is $S^1$-invariant and that contains the tangents to the $S^1$ orbits
on $\del M$.  
So it determines a distribution $\Ered$ on $\Mred$
such that $E \cap T \del M$ is the preimage of $\Ered$
under the differential of the quotient map $\del M \to \Mred$.

Fix a point $p \in \del M$.
By Proposition~\ref{prop:local}, we may assume that 
$$ M = D^{n-2} \times S^1 \times [0,\eps) ,$$
with coordinates $x_i$ on $D^{n-2}$, $\theta \mod 2\pi$ on $S^1$,
and $s$ on $[0,\eps)$,
that $E$ contains the vector field $\del/\del s$,
and that the point $p$ is given by $x=0$, $\theta=0$, and $s=0$.
Identify $\Mred$ with the disc $D^{n-2}$.

There exists $S^1$-invariant 
differential one-forms $\beta_1,\ldots,\beta_k$ on $M$
whose pullbacks to $\del M$ are basic,
whose null-space contains $E$,
and such that the $k$-form $\beta:= \beta_1 \wedge \ldots \wedge \beta_k$
is non-vanishing (at $p$, hence) near $p$.
Shrinking $D^{n-2}$ and $\eps$,
we may assume that $\beta$ is non-vanishing everywhere.

(Here is how to obtain such one-forms.
Equip $M$ with the Riemannian metric that is standard
in the coordinates $x$, $\theta$, $s$.
Let $(\eta_1)_p , \ldots , (\eta_k)_p$
be a basis to the orthocomplement of $E|_p$ in $T M$. 
Because $E$ contains $\del/\del s$,
the $(\eta_j)_p$ are tangent to $\del M$.
Extend each $(\eta_j)_p$
to a vector field on $M$ with constant coefficients
with respect to our coordinates,
and project to the orthocomplement of $E$ in $TM$.
We obtain $S^1$-invariant vector fields $\eta_1,\ldots,\eta_k$
that are everywhere orthogonal to $E$
and that at $p$ are a basis to the orthocomplement of $E|_p$.
For each $j$, inner product with $\eta_j$ defines a one-form $\beta_j$.
These one-forms have the required properties.)

Lemmas~\ref{cut diff form}, \ref{d and wedge}, and \ref{nonzero}
then give one-forms $(\beta_j)_\cut$ on $\Mcut$
such that $\beta_\cut = (\beta_1)_\cut \wedge \ldots \wedge (\beta_k)_\cut$
is non-vanishing.
Let $\Ecut$ be the null-space of $\beta_\cut$.
Then $\Ecut$ is a codimension-$k$ distribution on $\Mcut$
whose preimage under $c|_{\intM}$ is $E|_{\intM}$.

\end{proof}

\section{Diffeomorphisms in symplectic polar coordinates}
\labell{sec:symplectic polar:diffeos}

In this section, which can be read independently of the others,
we describe a baby-example
that motivates the equivariant-radial-squared-blowup construction
of Section~\ref{sec:radial squared blowup}.

Symplectic geometers find it natural to parametrize
a complex number $z = x+iy = re^{i\theta}$ 
by $s := \frac{1}{2} r^2$ and $\theta$  
rather than the standard polar coordinates $r$ and $\theta$.  
With this parametrization,
the standard symplectic form $dx \wedge dy$ becomes $ds \wedge d\theta$.
That is, $(s,\theta)$ are symplectic coordinates.
This works outside the origin; polar coordinates --- symplectic or not ---
generally fail to be useful at the origin.

In equivariant differential topology,
symplectic polar coordinates turn out to be useful
even at the origin.   We now explain how.
To simplify formulas, 
we now take $s$ to be $r^2$ instead of $\frac{1}{2} r^2$.

Consider the half-cylinder $S^1 \times \R_{\geq 0}$
with the circle acting on the first component.  
The map
$$ E \colon S^1 \times \Rplus \to \C
 \quad , \quad E(u,s) := \sqrt{s} \, u $$
descends to a bijection
$\ol{E} \colon (S^1 \times \R_{\geq 0})/{\sim} \to \C$,
where $\sim$ is the equivalence relation in which 
distinct points $(u_1,s_1)$ and $(u_2,s_2)$ 
are equivalent iff $s_1=s_2=0$.
Because the map $E$ is continuous and proper,
the bijection $\ol{E}$ is a homeomorphism.

The map $E$ is not smooth,
but the relation $E \circ \psi = \varphi \circ E$,
expressed in the commuting diagram below,
defines a bijection between $S^1$-equivariant diffeomorphisms 
of the half-cylinder and $S^1$-equivariant diffeomorphisms of $\C$.
$$ \xymatrix{
 S^1 \times \R_{\geq 0} \ar[r]^{\psi} \ar[d]_{ E }
 & S^1 \times \R_{\geq 0} \ar[d]^{ E } \\
 \C \ar[r]^{\varphi} & \C \, . 
} $$

Indeed, 
every $S^1$-equivariant diffeomorphism $\psi$ of the half-cylinder
has the form $\psi(u,s) = (\, a(s)u, \, g(s) s \,)$
for some smooth maps $a \colon \R_{\geq 0} \to S^1$
and $g \colon \R_{\geq 0} \to \R_{>0}$.
(This is a consequence of Hadamard's lemma.)
The corresponding map of $\C$ is 
$\varphi(z) \, = \, \sqrt{g(|z|^2)} \, a(|z|^2) \, z$, which is smooth.
Because we can obtain a smooth inverse for $\varphi$ 
by applying the same argument to the inverse of $\psi$,
we conclude that $\varphi$ is a diffeomorphism.

Conversely, 
every $S^1$-equivariant diffeomorphism $\varphi$ of $\C$
has the form $\varphi(z) =  r(|z|^2) \, a(|z|^2) \, z$ 
where $r \colon \Rplus \to \R_{>0}$
and $a \colon \Rplus \to S^1$ are smooth.
(This is a consequence of Hadamard's lemma
and of the fact that any $S^1$-invariant smooth function 
of $z \in \C$ is smooth as a function of $|z|^2$.
This fact, in turn, is a consequence of Whitney's theorem 
about smooth even functions~\cite{whitney}.)
The corresponding map of the half-cylinder is 
$\psi(u,s) = (a(s)u,r(s)^2s)$, which is smooth.
Because we can obtain a smooth inverse for $\psi$
by applying the same argument to the inverse of~$\varphi$, 
we conclude that $\psi$ is a diffeomorphism.

\begin{center}
\begin{figure}[h]
\includegraphics[scale=.75]{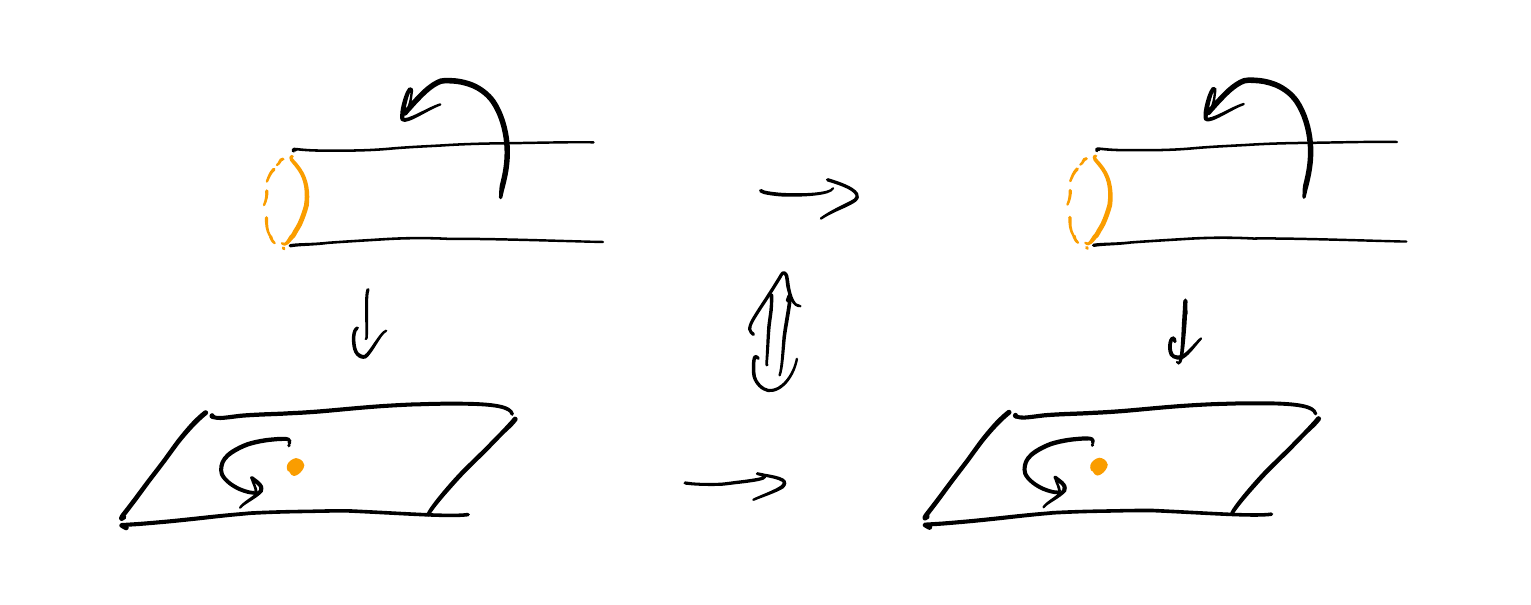}
\caption{Bijection between equivariant diffeomorphisms}
\end{figure}
\end{center}

Ordinary polar coordinates do not have this property:
with the corresponding map
$$ p \colon S^1 \times \Rplus \to \C \quad , \quad p(u,r) = ru ,$$
the diagram 
$$ \xymatrix{
 S^1 \times \R_{\geq 0} \ar[r]^{\psi} \ar[d]_{ p }
 & S^1 \times \R_{\geq 0} \ar[d]^{ p } \\
 \C \ar[r]^{\varphi} & \C \, 
} $$
does not give a bijection between $S^1$-equivariant diffeomorphisms
of the half-cylinder and $S^1$-equivariant diffeomorphisms of $\C$.
We do get a map in one direction:
for each $S^1$-equivariant diffeomorphism $\varphi \colon \C \to \C$,
this diagram determines an equivariant diffeomorphism 
$\psi \colon S^1 \times \Rplus \to S^1 \times \Rplus$; 
we explain this further below.
But the other direction does not work:
for example, the equivariant diffeomorphism
$\psi(u,r) := (ue^{ir}, r)$ of $S^1 \times \Rplus$
descends to the equivariant homeomorphism $\varphi(z) := z e^{i|z|}$ 
of $\C$, which is not smooth at the origin.

Once we pass to homeomorphisms, the distinction
between the radius and radius-squared becomes irrelevant,
and every equivariant homeomorphism of $S^1 \times \Rplus$ 
descends to an equivariant homeomorphism of $\C$.
But now we have a problem in the other direction:
an equivariant homeomorphism of $\C$
generally does not lift 
to an equivariant homeomorphism of $S^1 \times \Rplus$.
For example, the equivariant homeomorphism of $\C$
that is given by $\varphi(z) = ze^{i/|z|}$ when $z \neq 0$
does not lift to an equivariant homeomorphism of $S^1 \times \Rplus$.

We now explain why, under the ordinary polar coordinates map
$p \colon S^1 \times \Rplus \to \C$,
an equivariant diffeomorphism of $\C$ lifts
 to an equivariant diffeomorphism of the half-cylinder.
It is enough to consider neighbourhoods of the origin.
Identifying $\C$ with $\R^2$ and applying Hadamard's lemma, 
we obtain a smooth function 
$A \colon \R^2 \to \R^{2 \times 2}$
to the set of $2 \times 2$ matrices
such that $\varphi(z) = A(z) \cdot z$ for all $z \in \R^2$.
Because the differential of $\varphi$ at the origin is invertible,
the matrix $A(z)$ is invertible for $z$ near the origin.
The function 
$\psi(u,r) 
 := \big( \frac{A(ru) \cdot u}{|A(ru) \cdot u|}, r |A(ru) \cdot u| \big)$
is then a smooth lifting of $\varphi$ near the origin.
Because we obtain a smooth inverse of $\psi$ near the origin
by applying the same argument to the inverse of $\varphi$,
we conclude that $\varphi$ lifts to a diffeomorphism $\psi$
of the half-cylinder.
In fact, this argument 
applies to arbitrary---not necessarily equivariant---diffeomorphisms 
of $\C$ that fix the origin.
This is a special case of the functoriality of the 
radial blowup construction, which we describe in the next section.

\section{Radial blowup}
\labell{sec:radial blowup}

In this section we describe the radial blowup construction,
which was described by Klaus J\"anich in 1968 \cite{janich}.
This construction is a generalization of the passage to polar coordinates.
This section provides preparation for the next section,
where we describe the equivariant-radial-squared blowup construction,
which provides an inverse to the cutting construction.

Below, we use the symbol $0$ to denote the zero section
in any vector bundle.
Whenever we use it, it should be clear from the context
what is the ambient vector bundle.


Whereas any vector bundle can be naturally identified
with the normal bundle of its zero section,
we do not make this identification,
so as to not introduce ambiguity on a set-theoretic level.

\smallskip\noindent\textbf{Radial blowup 
of a vector bundle along its zero section}

Let $E \to F$ be a vector bundle.  Consider the sphere bundle 
$$ S_E(0) := ( \nu_E(0) \ssminus 0) / \Rpos \, ,$$
where $\nu_E(0)$ is the normal bundle of the zero section in $E$.
As a set, the radial blowup of $E$ along~$0$ is the disjoint union
$$ E \odot 0 := (E \ssminus 0) \sqcup S_E(0) \, ,$$
equipped with the blow-down map
$$ p \colon E \odot 0 \to E $$
whose restriction to $E \ssminus 0$ is the inclusion map into $E$
and whose restriction to $S_E(0)$
is induced from the bundle map $\nu_E(0) \to 0$.

Let
$$ q \colon E \ssminus 0 \to S_E(0) $$
be the quotient map to the sphere bundle
that is obtained from the natural identification $E \cong \nu_E(0)$
by restricting to the complement of the zero section
and composing with the quotient map to~$S_E(0)$.
Fix a fibrewise inner product on $E$ with fibrewise norm $|\cdot|$.
There exists a unique manifold-with-boundary structure on $E \odot 0$
such that the bijection
$$ E \odot 0 \to S_E(0) \times \Rplus $$
that is given by 
$$
 x \mapsto \begin{dcases}
  \ (q(x),|x|)  & \text{ if } x \in E \ssminus 0
\\
  \ \phantom{q(} (x \phantom{)} , \phantom{|} 0) & \text{ if } x \in S_E(0) 
\end{dcases} 
$$
is a diffeomorphism. 

For any two fibrewise inner products on $E$,
the ratio of their norms is a smooth function on $E \ssminus 0$
that is constant on $\Rpos$-orbits,
so it has the form $x \mapsto \rho(q(x))$, where
$$ \rho\colon S_E(0) \to \Rpos $$ 
is a smooth function with positive values.
The corresponding bijections $E \odot 0 \to S_E(0) \times \Rplus$
are then related by the diffeomorphism $(s,r) \mapsto (s,\rho(s)r)$
of $S_E(0) \times \Rplus$.
It follows that the manifold-with-boundary structure of $E \odot 0$
is independent of the choice of fibrewise inner product.

For any open neighbourhood $\calO$ of the zero section $0$ in $E$,
$$ \calO \odot 0 := (\calO \ssminus 0) \sqcup S_E(0) $$
is an open subset of the manifold-with-boundary $E \odot 0$,
so it too becomes a manifold-with-boundary.
Its boundary is
$$ \del (\calO \odot 0) = S_E(0) \, ,$$
and its interior is $\calO \ssminus 0$.
The manifold structures on $S_E(0)$ and on $\calO \ssminus 0$
that are induced from $\calO \odot 0$
coincide with their manifold structures that are induced from $E$.

\begin{center}
\begin{figure}[h]
\includegraphics[scale=.75]{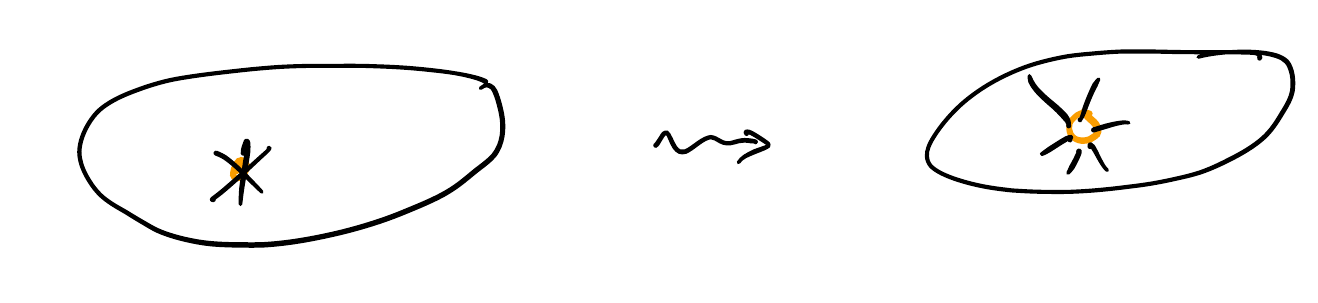}
\caption{Radial blowup}
\end{figure}
\end{center}

{

\smallskip\noindent\textbf{Lifting diffeomorphisms}

The following lemma is a crucial step 
in extending the radial blow-up construction
from vector bundles to manifolds.
%
 
\begin{Lemma}
\labell{diffeo lift to radial}
Let $E \to F$ be a vector bundle and $0$ its zero section.
Let $\calO_1$ and $\calO_2$ be open neighbourhoods of the zero section,
and let 
$$ \varphi \colon \calO_1 \to \calO_2 $$
be a diffeomorphism whose differential along the zero section 
is the identity map.
Then $\varphi$ lifts to a diffeomorphism 
$\psi \colon \calO_1 \odot 0 \to \calO_2 \odot 0$.
\end{Lemma}

Lemma~\ref{diffeo lift to radial} 
follows from the following more general result. 

\begin{Lemma}
\labell{smooth lift to radial}
Let $E_1 \to F_1$ and $E_2 \to F_2$ be vector bundles.
Let $\calO_1$ be an open neighbourhood 
of the zero section $0_1$ in $E_1$,
let $\calO_2$ be an open neighbourhood 
of the zero section $0_2$ in $E_2$,
and let
$$ \varphi \colon \calO_1 \to \calO_2 $$
be a smooth map.
Assume that for some---hence every---fibrewise inner product on $E_2$,
the function on $\calO_1$ that is given by 
$e \mapsto |\varphi(e)|^2$ vanishes exactly on the zero section,
and its Hessian is non-degenerate on the normal bundle 
to the zero section.
Then there exists a unique smooth map $\psi \colon \calO_1 \odot 0_1
 \to \calO_2 \odot 0_2$ such that the diagram
$$ \xymatrix{
   \calO_1 \odot 0_1 \ar[r]^{\psi}\ar[d]^{p_1} 
 & \calO_2 \odot 0_2 \ar[d]^{p_2} \\
 \calO_1 \ar[r]^{\varphi} & \calO_2
} $$
commutes.
Moreover, the composition of $\psi$ with any boundary-defining function
on $\calO_2 \odot 0_2$
is a boundary-defining function on $\calO_1 \odot 0_1$.
\end{Lemma}

\begin{proof}
The assumption on $\varphi$ is equivalent to the following property:
$\varphi$ takes the zero section $0_1$
to the zero section $0_2$
and the complement $\calO_1 \ssminus 0_1$
to the complement $\calO_2 \ssminus 0_2$,
and the fibrewise linear map $\nu_{E_1} (0_1) \to \nu_{E_2} (0_2)$
between the normal bundles to the zero sections 
that is induced from the differential of $\varphi$
is injective at each point of~$0_1$.

Because $\varphi$ takes $\calO_1 \ssminus 0_1$ to $\calO_2 \ssminus 0_2$,
and because
the blow-down maps $p_1$ and $p_2$ restrict to the identity maps
over the complements of the zero sections,
the commuting of the diagram
determines the map $\psi$ on the open dense subset $\calO_1 \ssminus 0_1$.
By continuity, if such a lift $\psi$ exists then it is unique.
So it is enough to show that a smooth lift $\psi$ of $\varphi$ 
exists locally near each point of the zero section $0_1$.

For each $i$,
choose a fibrewise inner product on $E_i$,
and let $S_i \subset E_i$ be the corresponding unit sphere bundle.
The inclusion map of $S_i$ into $E_i$
induces a diffeomorphism $S_i \cong S_{E_i}(0)$,
and, furthermore, a diffeomorphism of manifolds-with-boundary
$$ S_i \times [0,\infty) \xrightarrow{\cong}  E_i \odot 0_i  ,$$
whose composition with the blow-down maps is given by $(u,r) \mapsto ru$.

For each $i$, let $d_i$ denote the dimension 
of the base manifold $F_i$,
and let $k_i$ denote the rank of the vector bundle $E_i \to F_i$.
Because we may work locally, we may assume that
each $F_i$ is an open subset of $\R^{d_i}$
and each $E_i$ is the trivial bundle $F_i \times \R^{k_i}$,
and we may assume that $\calO_1 = D \times B$ 
where $D$ is an open subset of $\R^{d_1}$
and $B$ is a ball, say, of radius $\eps$, about the origin in $\R^{k_1}$.
So $\varphi$ becomes a smooth map 
from $D \times B$ to $\R^{d_2} \times \R^{k_2}$,
and we seek a smooth map $\psi$ such that the diagram
$$ \xymatrix{
 D \times S^{k_1-1} \times [0,\eps) \ar[r]^{\psi} 
 \ar[d]^{(t,u,r) \mapsto (t,ru)} 
    & \R^{d_2} \times S^{k_2-1} \times [0,\infty)\ar[d] 
 \ar[d]^{(t,u,r) \mapsto (t,ru)} 
 \\
 D \times B \ar[r]^{\varphi} & \R^{d_2} \times \R^{k_2}
} $$
commutes.

Because $\varphi$ maps the zero section to the zero section,
Hadamard's lemma with parameters gives a smooth map
$$ A \colon D \times B \to \R^{k_2 \times k_1} ,$$
where $\R^{k_2 \times k_1}$ denotes the set of $k_2 \times k_1$ matrices,
such that 
$$ \varphi(t,x) \, = \, ( \, \varphi_1(t,x) \, , \, A(t,x) \cdot x \, ) $$
for all $t \in D$ and $x \in B$,
where $\cdot$ denotes the product of a matrix with a column vector
and where $\varphi_1$ is the first component of $\varphi$.

Writing $u=x/|x|$ and $r=|x|$, we have
$$ \frac{ A(t,x) \cdot x } { | A(t,x) \cdot x | }
   = \frac{ A(t,ru) \cdot u } { | A(t,ru) \cdot u | } 
\quad \text{ and } \quad
|A(t,x) \cdot x| = r |A(t,ru) \cdot u| $$
whenever the denominator does not vanish.
But the matrix $A(t,0)$ represents
the linear map between the normal bundles to the zero sections
at the point $(t,0)$,
and the kernel of this linear map is trivial by our assumption on $\varphi$,
so $A(t,ru) \cdot u$ is non-vanishing when $r=0$.
Because $\varphi$ takes the complement to the zero section
to the complement to the zero section, 
$A(t,ru) \cdot u$ is non-vanishing also when $r > 0$.  So
$$ (t,u,r) \mapsto | A(t,ru) \cdot u | $$
defines a smooth map with positive values,
and 
we obtain a smooth lift $\psi$ of $\varphi$ with the required properties
by setting 
$$ \psi(t,u,r) \, := \, 
\Big( \varphi_1(t,ru), \  
      \frac{ A(t,ru) \cdot u } { |A(t,ru) \cdot u| } , \  
      r |A(t,ru) \cdot u | \Big) \, .$$
\end{proof}

}

\smallskip\noindent\textbf{Radial blowup of a manifold along a submanifold}

Let $X$ be a manifold, and let $F$ be a closed submanifold.
Consider the sphere bundle
$$ S_X(F) \, := \, (\nu_X(F) \ssminus 0)/\Rpos \, ,$$
where $\nu_X(F)$ is the normal bundle of $F$ in~$X$.
As a set, the radial blowup of $X$ along $F$ is the disjoint union
$$ X \odot F \, := \, (X \ssminus F) \, \sqcup \, S_X(F) \, ,$$
equipped with the blow-down map
$$ p \colon X \odot F \to X $$
whose restriction to $X \ssminus F$ is the inclusion map into $X$
and whose restriction to $S_X(F)$ is induced from the bundle map
$\nu_X(F) \to F$.
This definition is consistent with the earlier definition 
for the radial blowup of a vector bundle along its zero section.

Let $\calO$ be a starshaped open neighbourhood of the zero section $0$
in the normal bundle $\nu_X(F)$, let 
$$ \varphi \colon \calO \to X $$
be a tubular neighbourhood embedding,
and let
$$ \psi \colon \calO \odot 0 \to X \odot F $$
be the injection that is induced from $\varphi$.

(We have $\calO \odot 0 = (\calO \ssminus 0) \sqcup S_X(F)$.
On the subset $\calO \ssminus 0$,
the injection $\psi$ coincides with $\varphi$.
On the subset~$S_X(F)$, the injection $\psi$ 
is induced from the identification 
of $\nu_{\nu_X(F)}(0)$ with~$\nu_X(F)$.)

Then we have
$$ X \odot F \, = \, \psi(\calO \odot 0) \, \cup \, (X \ssminus F) \, .$$
There exists a unique manifold-with-boundary structure
on $X \odot F$ such that the maps
$$ \psi \colon \calO \odot 0 \, \to \, X \odot F 
\quad \text{ and } \quad
X \ssminus F \, \xrightarrow{\text{inclusion}} \, X \odot F $$
are diffeomorphisms with open subsets of $X \odot F$,
where, in the domains of these maps,
$\calO \odot 0$ is a manifold-with-boundary as an open subset 
of $E \odot 0$ for the vector bundle $E := \nu_X(F)$,
and $X \ssminus F$ is a manifold as an open subset of $X$.
This is because the preimage in $\calO \odot 0$ 
of the intersection $\psi (\calO \odot 0) \cap (X \ssminus F) $
is the open subset $\calO \ssminus 0$,
the preimage in $X \ssminus F$ of this intersection 
is the open subset $\varphi (\calO \ssminus 0)$,
and $\psi$ restricts to the diffeomorphism $\varphi$
between these preimages.
Because each of $\calO \odot 0$ and $X \ssminus F$ is second countable,
so is $X \odot F$.
To see that $X \odot F$ is Hausdorff, 
note that any point of $X \ssminus F$
has a neighbourhood $W$ whose closure $\ol{W}$ in $X$ is disjoint from $F$;
let $\calO_1$ be the preimage in $\calO$ 
of the complement $X \ssminus \ol{W}$;
then $\psi(\calO_1 \odot 0)$
is a neighbourhood of $S_X(F)$ in $X \odot F$
that is disjoint from the neighbourhood $W$ of the given point
of $X \ssminus F$.

The boundary of the manifold-with-boundary $X \odot F$ is 
$$ \del (X \odot F) = S_X(F) , $$
and the interior of $X \odot F$ is $X \ssminus F$.
The manifold structures on $S_X(F)$ and on $X \ssminus F$
that are induced from $X \odot F$
coincide with their manifold structures that are induced from $X$.

The manifold-with-boundary structure on $X \odot F$
is independent of the choice of tubular neighbourhood embedding.
J\"anich \cite{janich} states this fact without proof, and he writes this:
\emph{The proof has nothing to do with ``uniqueness of tubular maps'',
it is just an exercise in calculus}.
He probably had in mind Lemma~\ref{diffeo lift to radial}.
Indeed, for any two tubular neighbourhood embeddings 
$\varphi_1 \colon \calO_1 \to X$ and $\varphi_2 \colon \calO_2 \to X$,
their germs along the zero section of $\nu_X(F)$
are related by a diffeomorphism between neighbourhoods of the zero section
whose differential along the zero section is the identity map.
Lemma~\ref{diffeo lift to radial}
implies that---after possibly shrinking $\calO_1$ and $\calO_2$
so as to have the same image in $X$---the corresponding maps
$\psi_1 \colon \calO_1 \odot 0 \to X \odot F$
and $\psi_2 \colon \calO_2 \odot 0 \to X \odot F$
are related by a diffeomorphism 
from $\calO_1 \odot 0$ to $\calO_2 \odot 0$.
This implies that the identity map on $X \odot F$
is a diffeomorphism 
from the manifold-with-boundary structure 
that is induced from the tubular neighbourhood embedding $\varphi_1$
to the manifold-with-boundary structure 
that is induced from the tubular neighbourhood embedding $\varphi_2$.

\section{Equivariant radial-squared blowup}
\labell{sec:radial squared blowup}

In this section we describe the radial-squared blowup,
which---for circle actions---provides an inverse
to the cutting construction.

\smallskip\noindent\textbf{Radial-squared blowup of a vector bundle
along its zero section}

Let $E \to F$ be a vector bundle.
As a set, the radial-squared blowup of $E$ along $0$
is the same as the radial blowup:
$$ E \oodot 0 := (E \ssminus 0) \sqcup S_E(0) ,$$
where $0$ is the zero section 
and $S_E(0) = (\nu_E(0) \ssminus 0) / \Rpos$ is the sphere bundle
of its normal bundle. We take the same blow-down map
$$ p \colon E \oodot 0 \to E $$
and quotient map
$$ q \colon \, E \ssminus 0 \, \to S_E(0)$$ 
as before,
and fix a fibrewise inner product on $E$ with fibrewise norm $| \cdot |$.
There exists a unique manifold-with-boundary structure
on $E \oodot 0$ such that the bijection
$$ E \oodot 0 \to S_E(0) \times \Rplus $$
that is given by
$$
 x \mapsto \begin{dcases}
  \ (q(x),|x|^2)  & \text{ if } x \in E \ssminus 0
\\
  \ \phantom{q(} (x \phantom{)} , \phantom{|} 0) & \text{ if } x \in S_E(0) 
\end{dcases} 
$$
(this time with $|x|^2$, not $|x|$)
is a diffeomorphism. 
As before, this manifold-with-boundary structure on $E \oodot 0$
is independent of the choice of fibrewise inner product,
its boundary is $S_E(0)$, and its interior is $E \ssminus 0$.

The identity map 
$$ E \odot 0 \xrightarrow{\text{identity}} E \oodot 0 $$
is a homeomorphism and is smooth,
but it is not a diffeomorphism.
Thus, the radial blowup $E \odot 0$ 
and the radial-squared blowup $E \oodot 0$
yield the same topological manifold,
with the same smooth structures on its boundary and on its interior,
but with different manifold-with-boundary structures.
For any neighbourhood $U$ of the zero section, 
there exists a diffeomorphism between these manifolds-with-boundary
that is supported in $U$ 
and that restricts to the identity map on the zero section, 
but we cannot take this diffeomorphism to be the identity map.

\smallskip\noindent\textbf{Attempt at radial-squared blowup 
of a manifold along a submanifold}

We now discuss the possibility of defining
the radial-squared blowup of a manifold $X$ along a closed submanifold $F$.
As a set, this will be the same as the radial blowup:
$$ X \oodot F = (X \ssminus F) \sqcup S_X(F) \,,$$
with the same blow-down map $p \colon X \oodot F \to X$.
Let $\calO$ be a starshaped open neighbourhood
of the zero section $0$ in the normal bundle $\nu_X(F)$, 
let
$ \calO \oodot 0 := ( \calO \ssminus 0 ) \sqcup S_E(0) $
be the corresponding open subset of the manifold-with-boundary 
$E \oodot 0$,
fix a tubular neighbourhood embedding
$$ \calO \to X \, , $$
and let
$$ \psi \colon \calO \oodot 0 \to X \oodot F $$
be the injection that it induced.
As before,
there exists a unique manifold-with-boundary structure
on $X \oodot F$ such that the maps 
$$ \psi \colon \calO \oodot 0 \to X \oodot F \quad \text{ and } \quad
X \ssminus F \xrightarrow{\text{inclusion}} X \oodot F $$
are diffeomorphisms with open subsets of $X \oodot F$.
For this manifold-with-boundary structure
to be independent on the choice of the tubular neighbourhood embedding,
we need an analogue of Lemma~\ref{diffeo lift to radial}:
for every diffeomorphism $\varphi \colon \calO \to \calO'$ 
between neighbourhoods of the zero section
whose differential along the zero section is the identity map,
we would like to have a diffeomorphism $\hat{\psi}$ 
such that the square
$$ \xymatrix{
    \calO \oodot 0 \ar[r]^{\hat{\psi}} \ar[d] 
  & \calO' \oodot 0 \ar[d] \\
\calO \ar[r]^{\varphi} & \calO'
} $$
commutes.

Unfortunately, such a $\hat{\psi}$ might not exist.
To see this, we locally identify $\calO$ with $D \times B$
where $D$ is an open subset of $\R^d$ and $B$ is a ball,
say, of radius-squared $\eps$, about the origin in $\R^k$,
and we write the diffeomorphism $\varphi$ as
$$ \varphi(t,x) = (\varphi_1(t,x),A(t,x) \cdot x ) ,$$
as in the proof of Lemma~\ref{smooth lift to radial},
where $\varphi(t,0)=t$ and $A(t,0)$ is the identity matrix
for all $t$.
We then seek a smooth map $\hat{\psi}$ such that the diagram
$$
 \xymatrix{
 D \times S^{k-1} \times [0,\eps) \ar[r]^{\hat{\psi}}
 \ar[d]^{(t,u,s) \mapsto (t, \sqrt{s} u)} 
    & \R^{d} \times S^{k-1} \times [0,\infty)\ar[d] 
 \ar[d]^{(t,u,s) \mapsto (t, \sqrt{s} u)} 
 \\
 D \times B \ar[r]^{\varphi} & \R^{d} \times \R^{k}
} 
$$
commutes. Necessarily,
$$ \hat{\psi}(t,u,s) \, = \, 
\Big( \varphi_1(t, \sqrt{s} u), \  
      \frac{ A(t, \sqrt{s} u) \cdot u } { |A(t, \sqrt{s} u) \cdot u| } , \  
      s |A(t, \sqrt{s} u) \cdot u |^2 \Big) \, ;$$
this map is continuous, but it might not be smooth,
Take, for example, $\varphi(t,x) = (t + L(x), x)$
for some non-trivial linear map $L \colon \R^k \to \R^d$.
Then $\hat{\psi}(t,u,s) = (t+\sqrt{s}L(u), u, s)$, which is not smooth
along $\{s=0\}$.

Fortunately, in the special case that relates to cutting,
we can obtain a manifold structure on $X \oodot F$
through tubular neighbourhood embeddings that are equivariant near $F$.
We will do this in a moment.

In this special case that we need, $F$ has codimension $2$,
and the manifold $X$ is equipped with a circle action
near $F$ that fixes $F$ and is free outside $F$.
But keeping in mind future applications,
we will allow more than circle groups.
This is inspired by J\"anich's and Bredon's work 
(\cite{janich:earlier,janich}, \cite[Chapter VI, Section 6]{bredon}).
Future applications may include simultaneous cutting
of toric Lagrangians \cite{acannas}
that are invariant under a product of several copies of $S^1$ and $\Z_2$.

\smallskip\noindent\textbf{Equivariant radial-squared blowup 
of a manifold along a sub-manifold}

Let $G$ be a compact Lie group, and let $E \to F$ be a vector bundle,
equipped with a fibrewise linear $G$ action
that is transitive on the fibres of the sphere bundle 
$(E \ssminus 0) / \Rpos$.
Fix a $G$-invariant fibrewise inner product on $E$,
and let $S$ be the unit sphere bundle in $E$ with respect to 
the corresponding fibrewise norm $| \cdot |$.
The natural diffeomorphism $S \cong (E \ssminus 0)/\Rpos$
gives a diffeomorphism of manifolds-with-boundary
$$ S \times [0,\infty) \xrightarrow{\cong} E \oodot 0 $$
whose composition with the radial-squared-blowdown map is 
$$ (u,s) \mapsto \sqrt{s} u \, .$$

Let $E' \to F'$ be another vector bundle with a $G$ action
with these properties, with zero section $0'$ 
and unit sphere bundle $S'$.

We have the following analogues
of Lemmas~\ref{diffeo lift to radial} and~\ref{smooth lift to radial}.

\begin{Lemma}\labell{diffeo lift to radial squared}
Let $\calO$ and $\calO'$
be $G$-invariant open neighbourhoods of the zero section of $E$, and let
$$ \varphi \colon \calO \to \calO' $$
be a $G$-equivariant diffeomorphism 
whose differential along the zero section is the identity map.  
Then $\varphi$ lifts to a diffeomorphism
$ \psi \colon \calO \oodot 0 \to \calO' \oodot 0$.
\end{Lemma}

Lemma~\ref{diffeo lift to radial squared}
follows from the following more general result.

\begin{Lemma}\labell{smooth lift to radial squared}
Let $\calO$ and $\calO'$ be $G$-invariant neighbourhoods
of the zero sections in $E$ and $E'$, and let 
$$ \varphi \colon \calO \to \calO' $$
be a $G$-equivariant smooth map.
Assume that the function on $\calO$
that is given by $e \mapsto |\varphi(e)|^2$
vanishes exactly on the zero section
and that its Hessian 
is non-degenerate on the normal bundle to the zero section.
Then there exists a unique smooth map $\psi'$ such that the diagram
$$ \xymatrix{
    \calO \oodot 0 \ar[r]^{\psi'} \ar[d] 
  & \calO' \oodot 0' \ar[d] \\
\calO \ar[r]^{\varphi} & \calO'
} $$
commutes.
Moreover, $\psi'$ is equivariant,
and its composition with any boundary-defining function 
on $\calO' \odot 0'$ is a boundary-defining function on $\calO \odot 0$.
\end{Lemma}

\begin{proof}
As in the proof of Lemma~\ref{smooth lift to radial}, 
we locally identify $\varphi$ with a map 
$$ \varphi \colon D \times B \to \R^{d'} \times \R^{k'} $$
of the form
$$ \varphi(t,x) = (\varphi_1(t,x),A(t,x) \cdot x) $$
where $A(t,0)$ is invertible for all $t$,
and we seek a smooth map $\hat{\psi}$ such that the diagram
$$
 \xymatrix{
 D \times S^{k-1} \times [0,\eps) \ar[r]^{\hat{\psi}}
 \ar[d]^{(t,u,s) \mapsto (t, \sqrt{s} u)} 
    & \R^{d'} \times S^{k'-1} \times [0,\infty)\ar[d] 
 \ar[d]^{(t,u,s) \mapsto (t, \sqrt{s} u)} 
 \\
 D \times B \ar[r]^{\varphi} & \R^{d'} \times \R^{k'}
} 
$$
commutes.  Because $\varphi$ is $G$ equivariant, 
$\varphi_1$ and $A(t,x)$ are $G$ invariant.   
Because the $G$ action is transitive on the fibres of $S$,
there exist functions $\tilde{\varphi}_1$ and $\tilde{A}$ such that
$$ \varphi_1(t,x) = \tilde{\varphi}_1(t,|x|^2) 
\quad \text{ and } \quad
   A(t,x) = \tilde{A}(t,|x|^2) \, .$$
Because $\varphi_1$ and $A$ are smooth,
the functions $\tilde{\varphi}_1$ and $\tilde{A}$ are also smooth.
(This special case of Schwarz's theorem \cite{schwarz}
appears in Bredon's book \cite[Chapter~VI, Theorem~5.1]{bredon};
it can be deduced from Whitney's theorem about even functions
\cite{whitney}.)
Then 
$$ \hat{\psi}(t,u,s) \, := \, 
\Big( \tilde{\varphi}_1(t, s), \  
      \frac{ \tilde{A}(t, s) \cdot u } { |\tilde{A}(t,s) \cdot u| } , \  
      s |\tilde{A}(t,s) \cdot u |^2 \Big) \,  \\
$$
is a smooth lift of $\varphi$ with the required properties.
\end{proof}

We can now define the equivariant radial-squared blowup 
$X \oodot F$
of a manifold $X$ along a submanifold $F$, 
when a compact Lie group $G$ acts on a neighbourhood of $F$, 
and when this action fixes $F$ 
and is transitive on the fibres of the sphere bundle 
$S_X(F) = (\nu_X(F) \ssminus 0 ) / \Rpos$.
As a set, it is the same as the radial blowup:
$$ X \oodot F := (X \ssminus F) \, \sqcup \, S_X(F) .$$
Fix a $G$-invariant starshaped open neighbourhood 
$\calO$ of the zero section $0$ in the normal bundle $\nu_X(F)$
and a $G$-equivariant tubular neighbourhood embedding
$$ \calO \to X \,,$$
and let
$$ \psi \colon \calO \oodot 0 \to X \oodot F $$
be the injection that is induced from $\varphi$.
There exists a unique manifold-with-boundary structure on $X \oodot F$
such that the maps
$$ \psi \colon \calO \oodot 0 \to X \oodot F \quad \text{ and } \quad
X \ssminus F \xrightarrow{\text{inclusion}} X \oodot F $$
are diffeomorphisms with open subsets of $X \oodot F$.
Lemma~\ref{diffeo lift to radial squared}
implies that this manifold-with-boundary structure
is independent on the choice of the tubular neighbourhood embedding.

As in the case of a vector bundle,
the radial blowup $X \odot F$
and the radial-squared blowup $X \oodot F$
yield the same topological manifold, with the same smooth structures
on its boundary and on its interior, 
but with different manifold-with-boundary structures.
for any neighbourhood of $F$, there exists an equivariant
diffeomorphism between these manifolds-with-boundary
that is supported in this neighbourhood
and that restricts to the identity map on $F$,
but we cannot take this diffeomorphism to be the identity map.

\smallskip\noindent\textbf{Functoriality:}

By Lemma~\ref{smooth lift to radial squared},
the radial-squared blowup procedure defines a functor.
The domain of this functor is the following category.
An object is a pair $(X,F)$, 
where $X$ is a manifold and $F$ is a closed submanifold,
equipped with a $G$-action on a neighbourhood of $F$
that fixes $F$ and acts transitively on the fibres
of the sphere bundle of $\nu_X(F)$.
For such a pair $(X,F)$, 
an \textbf{invariant norm-squared-like function} 
is a function $X \to \Rplus$ that is $G$-invariant near $F$, 
that vanishes exactly on $F$,
and whose Hessian is non-degenerate on the normal bundle $\nu_X(F)$.
A morphism from $(X,F)$ to $(X',F')$ is a smooth map
from $X$ to $X'$ whose composition with (some, hence every) 
invariant norm-squared-like function on $X'$
is an invariant norm-squared-like function on $X$,
and that is $G$-equivariant near $F$.

The functor takes an object $(X,F)$ to $X \oodot F$
and a morphism $\varphi \colon X \to X'$
to the smooth map $\psi \colon X \oodot F \to X' \oodot F'$
such that the following diagram commutes
$$\xymatrix{
X \oodot F \ar[r]^{\psi} \ar[d] & X' \oodot F' \ar[d] \\
X \ar[r]^{\varphi} & X' ,
}$$
where the vertical arrows are the 
equivariant-radial-squared-blowdown maps.
Such a map $\psi$ exists by Lemma~\ref{smooth lift to radial squared}.

When $G$ is the circle group and its action is faithful,
the transitivity of the action on the fibres of the sphere bundle $S_E(F)$
implies that $F$ has codimension two in $X$
and that the action is free on a punctured neighbourhood of $F$ in $X$.
The equivariant-radial-squared-blowup functor then takes values
in the category whose objects are manifolds-with-boundary,
equipped with free circle actions near their boundaries,
and whose morphisms are equivariant transverse maps,
which we defined in Section~\ref{sec:functoriality}.
We now show that this functor provides an inverse
to the cutting functor.

\smallskip\noindent\textbf{Radial-squared blowup and cutting:}

We return to the circle group, $G=S^1$.

Let $X$ be a manifold and $F$ a closed codimension two submanifold,
and let $U_X$ be an open neighborhood of $F$ in $X$,
equipped with a circle action that fixes $F$
and is free on $U_X \ssminus F$.
The radial-squared blowup construction yields
a manifold-with-boundary $X \oodot F$.
The circle action on $U_X$ lifts to a free circle action on 
the open neighbourhood $U_X \oodot F$ of the boundary.
The cutting construction then yields a manifold $(X \oodot F)_\cut$ 
and a quotient map $c \colon X \oodot F \to (X \oodot F)_\cut$.
The interior of the manifold-with-boundary $X \oodot F$ is $X \ssminus F$.
We claim that there exists a unique diffeomorphism 
$$ X \xrightarrow{\cong} (X \oodot F)_\cut $$
whose restriction to $X \ssminus F$ coincides with the quotient map $c$
(which takes each point of $X \ssminus F$ to the singleton
containing that point).

Conversely, let $M$ be a manifold-with-boundary,
equipped with a free circle action on an open neighbourhood $U_M$
of the boundary. 
The cutting construction yields a manifold $\Mcut$ 
and codimension two submanifold $\Mred$,
and a quotient map $c \colon M \to \Mcut$.
The circle action on $U_M$
descends to a circle action on an open neighbourhood of $\Mred$
that fixes $\Mred$ and is free outside $\Mred$.
The radial-squared blowup then yields
a manifold-with-boundary $\Mcut \oodot \Mred$
with a free circle action on the open neighbourhood of the boundary.
The interior of the manifold-with-boundary 
$\Mcut \oodot F$ is $(\intM)_\cut$.
We claim that there exists a unique diffeomorphism 
$$ M \xrightarrow{\cong} \Mcut \oodot \Mred $$
whose restriction to the interior $\intM$ 
coincides with the quotient map $c$
(which takes each point of $\intM$ to the singleton containing that point).

By functoriality, it is enough to check these claims locally. 
Locally, these claims 
follow from the local models for the cutting construction
and from the radial--squared blowup construction of a vector bundle.

Namely, if 
$$ X = D^{n-2} \times D^2 \quad \text{ and } F = D^{n-2} \times \{ 0 \} $$
where $D^{n-2}$ is an open subset of $\R^{n-2}$
and $D^2$ is the open disc of radius-squared $\eps$
about the origin in $\R^2$,
then the definition of the radial-squared blowup
gives a diffeomorphism 
$$ X \oodot F \xrightarrow{\cong} D^{n-2} \times S^1 \times [0,\eps) $$
that carries the radial-squared-blowdown map 
$X \oodot F \to X$ to 
$$ (t,u,s) \mapsto (t,\sqrt{s}u) ,$$
which is exactly the map that induces the diffeomorphism 
$$ \Big( D^{n-2} \times S^1 \times [0,\eps) \Big)_\cut
 \xrightarrow{\cong} D^{n-2} \times D^2 .$$

Conversely, if $M = D^{n-2} \times S^1 \times [0,\eps)$, 
then we have a diffeomorphism 
$$ \Mcut \xrightarrow{\cong} D^{n-2} \times D^2 $$
that takes the quotient map $c \colon M \to \Mcut$ to
$$ (t,u,s) \mapsto (t,\sqrt{s}u) ,$$
which is exactly the map that induces the diffeomorphism
$$ \Mcut \oodot \Mred \xrightarrow{\cong} 
   D^{n-2} \times S^1 \times [0,\eps).$$

\appendix

\section{Actions and quotients}
\labell{sec:actions}

In this appendix we collect some well-known facts
about actions of compact groups and their quotients.

First, we recall some facts about proper maps.
A continuous map $g \colon A \to B$
between topological spaces is \textbf{proper} if the preimage
of every compact subset of $B$ is a compact subset of $A$.

\begin{Exercise*}
Every closed map with compact level sets is proper.
Every proper map has compact level sets.
If $B$ is Hausdorff and locally compact, 
then every proper map to $B$ is a closed map.
\end{Exercise*}

Recall that, for a topological space, 
being $T_1$ means that singletons are closed,
and being normal means that for any two disjoint closed sets $C_1$ and $C_2$
there exist open sets $U_1$ and $U_2$ that are disjoint 
and such that $U_1$ contains $C_1$ and $U_2$ contains $C_2$.
Every $T_1$ normal space is Hausdorff.

Manifolds are assumed to be Hausdorff and second countable.
Because they are locally compact, 
these properties imply that they are paracompact
(every open cover has a locally finite open refinement)
and normal;
see, e.g., \cite[Theorems~4.77 and~4.81]{JohnLee:topological}.

\begin{Lemma} \labell{quotient map}
Let a compact topological group $G$ 
act continuously on a $T_1$ normal topological space $N$.
Consider the action map $\rho \colon G \times N \to N$,
which we write as $(a,x) \mapsto a \cdot x$,
and the quotient map $\pi \colon N \to N/G$.
For any subset $A$ of $N$, write
$G \cdot A = \{ a \cdot x \ | \ a \in G \text{ and } x \in N \}$.
Write the orbit of a point $x \in N$ also as $G \cdot x$
(and not only as $G \times \{ x \}$).

\begin{enumerate}


\item \labell{orbits-compact-closed}
Orbits in $N$ are compact and closed.


\item \labell{action is closed}
The action map $G \times N \to N$ is a closed map.
\item \labell{sweep-open}
For every open subset $U$ of $N$, the subset $G \cdot U$ of $N$ is open.
\item \labell{intersection}
For every open subset $U$ of $N$, 
the intersection $\bigcap\limits_{a \in G} a \cdot U$ is open.
\item \labell{quotient:T1-normal}
The quotient topological space $N/G$ is $T_1$ and normal.
\end{enumerate}
\end{Lemma}

\begin{proof} 

For any compact subset $K$ of $N$,
the subset $G \cdot K$ of $N$,
being the image of the compact set $G \times K$
under the continuous action map $(a,x) \mapsto a \cdot x$,
is also compact.

Applying this to singletons in $N$,
we obtain that $G$-orbits in $N$ are compact.
Because $N$ is Hausdorff, this implies that $G$-orbits in $N$ are closed.
This proves~\eqref{orbits-compact-closed}.

We now show that the action map $G \times N \to N$ is a proper map.
Let $K$ be a compact subset of $N$.
Because $N$ is Hausdorff, the compact set $K$ is closed in $N$,
so its preimage under the continuous action map $(a,x) \mapsto a \cdot x$
is closed.
Being a closed subset of the compact set $G \times (G \cdot K)$,
this preimage is compact.
Because $K$ was an arbitrary compact subset of $N$,
the action map is proper.

Because the action map $G \times N \to N$ is proper 
and its target space $N$ is Hausdorff and locally compact,
the action map is a closed map.
This proves~\eqref{action is closed}.

Let $U$ be an open subset of $N$.
For each $a \in G$,
because the maps $x \mapsto a \cdot x$
and $x \mapsto a^{-1} \cdot x$ are continuous
and are inverses of each other, 
the set $a \cdot U$ is open.
The set $G \cdot U$,
being the union of the open sets $a \cdot U$ over all $a \in G$,
is then open.
This proves~\eqref{sweep-open}.

Let $U$ be an open subset of $N$.
Then the complement of $U$ in $N$ is closed in $N$.
By~\eqref{action is closed},
the image of this complement under the action map $G \times N \to N$
is closed in $N$.
So the complement in $N$ of this image is open.
But this complement is exactly $\bigcap\limits_{a \in G} a \cdot U$.
This proves~\eqref{intersection}.

Because (by~\eqref{orbits-compact-closed}) orbits in $N$ are closed,
the quotient space $N/G$ is $T_1$.
It remains to show that the quotient space $N/G$ is normal.
Let $C_1$ and $C_2$ be disjoint closed subsets of $N/G$.
Then $\pi^{-1}(C_1)$ and $\pi^{-1}(C_2)$ are 
disjoint closed subsets of $N$.
Because $N$ is normal, there exist open subsets $\wh{U}_1$ and $\wh{U}_2$
of $N$ that are disjoint and such that
$\wh{U}_1$ contains $\pi^{-1}(C_1)$ and $\wh{U}_2$ contains $\pi^{-1}(C_2)$.
Let
$U_1 := \big( \bigcap\limits_{a \in G} a \cdot \wh{U}_1 \big)/G$
and 
$U_2 := \big( \bigcap\limits_{a \in G} a \cdot \wh{U}_2 \big)/G$.
Then $U_1$ and $U_1$ are open subsets of $N/G$
(by~\eqref{intersection}),
they are disjoint, 
$U_1$ contains $C_1$, and $U_2$ contains $C_2$.
Because $C_1$ and $C_2$ were arbitrary disjoint closed subsets of $N/G$,
we conclude that $N/G$ is normal.
This proves~\eqref{quotient:T1-normal}.
\end{proof}

\begin{Lemma} \labell{quotient:T1-normal-secondcountable}
Let a compact topological group $G$ 
act continuously on a topological space $N$.
Suppose that $N$ is $T_1$ normal, and second countable.
Then $N/G$ is also $T_1$ normal, and second countable.
\end{Lemma}

\begin{proof}[Proof of Lemma~\ref{quotient:T1-normal-secondcountable}]
Let $y$ be a point of $N/G$. 
By~\eqref{orbits-compact-closed}, 
$y$ is closed as a subset of $N$.
Because the subset $y$ of $N$ 
is the preimage under the quotient map $N \to N/G$
of the singleton $\{ y \}$, 
this singleton is closed in $N/G$.
Because the point $y$ of $N/G$ was arbitrary, 
the topological space $N/G$ is $T_1$.

Let $\hat{C}_1$ and $\hat{C}_2$ be disjoint closed subsets of $N/G$.
Let $C_1$ and $C_2$ be their preimages in $N$;
then $C_1$ and $C_2$ are disjoint closed subsets of $N$.
Because $N$ is normal,
there exist disjoint open subsets $U_1$ and $U_2$ of $N$, 
that, respectively, contain $C_1$ and $C_2$.
By~\eqref{intersection}, the $G$-invariant sets 
$U'_1 := \bigcap\limits_{a \in G} a \cdot U_1$ 
and 
$U'_2 := \bigcap\limits_{a \in G} a \cdot U_2$ 
are open in $N$.
The quotients $U'_1/G$ and $U'_2/G$ are 
disjoint open subsets of $N/G$
that, respectively, contain $\hat{C}_1$ and $\hat{C}_2$.
Because the closed subsets $\hat{C}_1$ and $\hat{C}_2$ of $N/G$
were arbitrary, the topological space $N/G$ is normal.

\end{proof}

\begin{Lemma} \labell{subordinate}
Let a compact Lie group $G$ act freely on a manifold-with-boundary $N$.
Then the following holds.

\begin{enumerate}
\item \labell{koszul}
There exists a unique manifold-with-boundary structure on $N/G$
such that the quotient map $N \to N/G$ is a submersion.  
Moreover, this quotient map is a principal $G$-bundle.
Finally, the boundary of $N$ maps to the boundary of $N/G$,
and the interior of $N$ maps to the interior of $N/G$.
\item \labell{bump}
For every orbit $G \cdot x$ in $N$
and $G$-invariant open set $U$ that contains the orbit
there exists a $G$-invariant smooth function 
$\rho \colon N \to \R$
that is equal to $1$ on a neighbourhood of the orbit
and whose support is contained in the open set $U$.
\item \labell{partition of unity}
For every cover of $N$ by $G$-invariant open sets
there exists a partition of unity 
by $G$-invariant smooth functions
that is subordinate to the cover. 
\end{enumerate}
\end{Lemma}

\begin{proof} 
The following two local facts are special cases of Koszul's slice theorem
and its analogue for manifolds-with-boundary.
They can be proved by a straightforward adaptation
of the proof of Proposition~\ref{prop:local}.
\begin{itemize}
\item
Each orbit in the interior of $N$
has a neighbourhood that is equivariantly diffeomorphic to $D^k \times G$,
with $G$ acting by left translation on the middle factor,
and where $D^k$ is a disc in $\R^k$.
\item
Similarly,
each orbit in the boundary of $N$
has a neighbourhood that is equivariantly diffeomorphic 
to $D^{k-1} \times G \times [0,\eps)$,
with $G$ acting by left translation on the middle factor,
and where $D^{k-1}$ is a disc in $\R^{k-1}$ and $\eps > 0$.
\end{itemize}

By Lemma~\ref{quotient map}\eqref{quotient:T1-normal},
the quotient space $N/G$ is $T_1$ and normal.

Let $\frakU$ be a countable basis for the topology of $N$;
then $\frakU' := \{ (G \cdot U/G \}_{U \in \frakU}$
is a countable basis for the topology of $N/G$.
So $N/G$ is second countable.

The first part of the lemma follows from the above two local facts
and from $N/G$ being Hausdorff and second countable.
The second and third parts of the lemma
then follow from the existence of smooth bump functions
and smooth partitions of unity on the manifold-with-boundary $N/G$.
\end{proof}

\section{Simultaneous cutting}
\labell{sec:simultaneous}

We expect the results of this paper to generalize to 
the setups of simultaneous cutting.
We include here the relevant definitions,
deferring the details to another occasion.

Let $M$ be an $n$ dimensional manifold-with-corners
(as introduced by Jean Cerf and by Adrien Douady \cite{cerf,douady};
see \cite[Sect.~16]{JohnLee:smooth}).

The \emph{depth} of a point $x \in M$ is the (unique) integer $k$
such that there exists a chart $U \to \Omega$
from a neighbourhood $U$ of $x$ in $M$
to an open subset $\Omega$ of the sector $\R_{\geq 0}^{k} \times \R^{n-k}$
that takes $x$ to a point in $\{ 0 \}^k \times \R^{n-k}$.
The \emph{k-boundary} of $M$, denoted $M^{(k)}$, 
is the set of points $x \in M$ of depth~$k$.  
The \emph{interior} of $M$ is its $0$-boundary.
The \emph{strata} of $M$
are the connected components of $M^{(k)}$ for $k = 0, \ldots, n$.

Every open subset $U$ of $M$ is also a manifold with corners.
The manifold with corners $M$ is a \emph{manifold with faces}
if each point $x$ of depth $k$ 
is in the closure of $k$ distinct components of $M^{(1)}$.

Now assume that $M$ is an $n$ dimensional manifold with faces. 
The \emph{faces} of $M$ are the closures of the strata.
The \emph{facets} of $M$ are the $n-1$ dimensional strata.
Let $\calS$ denote the set of facets of $M$.
Every codimension $k$ face $Y$ of $M$, obtained as the closure
of a stratum~$\intY$, is itself a manifold with corners 
(in fact, with faces),
whose interior is the stratum $\intY$.
Also, the face $Y$ is the intersection of a unique 
$k$ element subset $S_Y$ of $\calS$.

For each facet $F \in S$, let $U_F$ be a neighbourhood of $F$ in $M$.
Suppose that, for each subset $S$ of $\calS$, 
the neighbourhoods $U_F$, for $F \in S$, have a nonempty intersection
only if the facets $F$, for $F \in S$, have a nonempty intersection.
This intersection is then the union of those faces $Y$ with $S_Y = S$.
Suppose that we are given a free circle action on each neighbourhood $U_F$
and that these actions commute on the intersections 
of these neighbourhoods.

To each face $Y$ of $M$ we then associate neighbourhood
$U_Y$ that is contained in the intersection $\cap \{ U_F \ | \ F \in S_Y \}$
and is invariant with respect to the torus
$T_Y := (S^1)^{S_Y}$, which acts on $U_Y$.
(If $Y$ has codimension $k$, then $S_Y$ is a set of $k$ elements,
so $T_Y$ has dimension $k$.)

Consider the equivalence relation $\sim$ on $M$
such that, for $m \neq m'$, we have that $m \sim m'$ if and only if
there exists a codimension $k$ face $Y$ of~$M$ 
and a torus element $a \in T_Y$
such that $m,m' \in Y$ and $m = a \cdot m'$.

Let $\Mcut := M/{\sim}$; equip $\Mcut$ with the quotient topology;
let $c \colon M \to \Mcut$ denote the quotient map.
For each face $Y$ of $M$, denote $X_Y = \intY_\cut$.
There exists a unique manifold structure on $X_Y$ 
such that a real valued function $h \colon X_Y \to \R$ is smooth
if and only if the function $h \circ c|_{\intY} \colon \intY \to \R$
is smooth.  Moreover, with this manifold structure,
the map $c|_{\intY} \colon \intY \to X_Y$ is a principal $T_Y$-bundle.
This follows from the slice theorem for the $T_Y$ action on $\intY$.

So we have a decomposition of the topological space $\Mcut$
into disjoint subsets $X_Y$, 
and we have a smooth manifold structure on each of the subsets $X_Y$.

For each facet $F \in \calS$, fix a smooth function 
$f^{(F)} \colon U_F \to \R_{\geq 0}$
whose zero level set is $F$ and such that $df^{(F)}|_F$ never vanishes.  
Moreover, assume that each $f^{(F)}|_{U_F}$ is invariant,
with respect to the circle action on $U_F$,
on some invariant neighbourhood of $F$ in $U_F$.
Furthermore, assume that for each face $Y$ 
the functions $f^{(F)}$, for $F \in S_Y$, 
are $T_Y$-invariant on some neighbourhood of $Y$.

Generalizing Construction~\ref{calF M},
we define $\calF_M$ to be 
the set of those real valued functions $h \colon \Mcut \to \R$
whose composition $\wh{h} := h \circ c$
with the quotient map $c \colon M \to \Mcut$
satisfies the following two conditions.
\begin{itemize}
\item[($\calF 1$)]
$\left.\wh{h}\right|_{\intM} \colon \intM \to \R$ is smooth.
\item[($\calF 2$)]
For each face $Y$ of $M$,
there exists a $T_Y$-invariant neighbourhood $V$ of $Y$ in $U_Y$
and a smooth function $H \colon V \times \C^{S_Y} \to \R$ such that
\begin{itemize}
\item[(a)]
$H(a \cdot x, z) = H(x,az)$ 
for all $a \in T_Y$ and $(x,z) \in V \times \C^{S_Y}$; and
\item[(b)]
$\wh{h}(x) = H(x,z)$,
where the coordinates of $z$
are given by $z_F = \sqrt{f^{(F)}}(x)$ for all ${F \in S_Y}$.
\end{itemize}
\end{itemize}

As in Lemma~\ref{FM indep of f} and Theorem~\ref{mfld str},
we expect the set of functions $\calF_M$ 
to be independent of the choice of invariant boundary defining functions 
$f^{(F)}$,
and we expect there to exist a unique manifold structure on $\Mcut$
such that $\calF_M$ is the set of real valued smooth functions on $\Mcut$.

We expect the results of this paper
(functoriality of $\psi \mapsto \psi_\cut$,
cutting of submanifolds,
cutting of differential forms, 
symplectic and contact cutting)
to extend to this setup.

We expect that a simultaneous (not iterated)
equivariant-radial-squared-blowup construction
would provide an inverse to the simultaneous cutting construction.



{\ifdebug \newpage \else \end{document} \fi}

\section{Leftovers}

\ynote{Eugene: In his cutting papers he didn't assume nondegeneracy. \\
(Classification of contact toric mflds: he did the case of proper cones;
the other cases are "obvious"..)}

\textbf{To do:}
\begin{itemize}
\item
Give Bredon credit for describing the issue 
in Janich's construction 
(although Bredon did not work with radial blowups).
\end{itemize}

\textbf{Journals:}
\begin{itemize}
\item Inventiones: send email to editor.
Kai Cieliebak? Danny Calegari? 
\item Duke: submit online.
Some editors: Richard Hain, Tobias Colding $\ldots$
\item Geometry and Topology: 
submit online and suggest editors.
Frances Kirwan? Etienne Ghys?

\item[$\circ$:] Memoirs want at least 80 pages.
(This paper is 65; with ``fullpage'' it's 52.)
\end{itemize}

\textbf{Not for action:}

Some nice latex trick:
\begin{center}
\begin{minipage}{.25\textwidth}
\begin{multline*}
\calO_X := \, \varphi(\calO) \\ \subset \, X \, 
\end{multline*}
\end{minipage}
\quad , \quad
\begin{minipage}{.65\textwidth}
\begin{multline*}
  \calO_{X \odot F} := \, p^{-1}(\calO_X) \\
  \ = (\calO_X \ssminus F) \sqcup S_X(F)  \ 
  \, \subset \, X \odot F \,  
\end{multline*}
\end{minipage} 
\end{center}

\textbf{Not for action:}
\begin{itemize}

\item
This iterated construction is different
from the simultaneous construction that we mentioned above;
it applies more generally, but without additional data
it yields less information about the original $G$-manifold.
 
\item
We expect that a simultaneous equivariant-radial-squared blowup construction
would provide an inverse to the simultaneous cutting construction.
and would lead to a classification of these $G$-manifolds,
generalizing generalizing the earlier work 
of J\"anich, Hsiang-Hsiang, and Bredon.
The special case of locally standard torus action 
is the topic of our a paper-in-progress 
with Shintaro kuroki~\cite{karshon-kuroki}.

\item
Guillemin-Lerman \cite[Proposition~1.1]{guillemin-lerman}
define on $\Mcut$ the structure of an orbifold 
in which the structure group of all the points in $\Mred$ is $\Z_2$.

\item
Removed from the intro (only works in the locally
standard case):

When $G$ is the torus $T:=(S^1)^k$,
Haefliger and Salem~\cite{haefliger-salem}
recover a $T$-manifold $M$ from data on $M/T$
together with a cohomology class in $H^2(M/T,\Z^k)$.
Haefliger and Salem do not work with radial blowups,
but their cohomology class
encodes a principal $T$ bundle over $M/T$
that can be obtained by 
simultaneous (not iterated!) radial blowups of $M/T$.

\item
2020 Math Subject Classification: \\
53D20 Momentum maps; symplectic reduction \\
57R55 Differentiable structures in differential topology

\item
Relation w/ Eckhard's claim (from his s.geom lecture notes)
that we can obtain Delzant classification through cutting.

\item
Ana CdS's paper with Victor and Chris; see page 18 culminating with Theorem 7
\url{https://people.math.ethz.ch/~acannas/Papers/unfolding.pdf} \\
Ana CdS's paper with Victor and Ana Rita; see Proposition 2.10
\url{https://people.math.ethz.ch/~acannas/Papers/origami.pdf}

\item
Albin+Melrose incorrectly claim that the earlier people 
only did ``real blowup'' or produced the double covering of the radial blowup.
Melrose has this construction in his book-draft.
Bredon has relevant stuff in his book, in Section 5 of Chapter VI.
Bredon:
Consider the $\Z_2$ action on $\R$.
Take the orbit map $x \mapsto |x|$.
Then the diffeomorphism $x \mapsto x + x^2$ of the quotient
does not lift to a diffeomorphism of $\R$.
We can view this same example as follows.
The radial blowup of $\R$ is $\Rplus \sqcup \Rplus$.
The formula $x \mapsto x+x^2$ defines an equivariant diffeomorphism
of this radial blowup that does not correspond to an equivariant
diffeomorphism of $\R$.

\item
Are the smooth structures on $\Mcut$ 
that come from different extensions of the circle action on $\del M$
to a circle action near $\del M$ always isomorphic?
\item
Cutting of Lagrangian submanifolds

\item
The maps $\psi_\cut$ that we get downstairs
are probably transverse (near the cut locus) 
in the sense of Michael Davis.

\item
I think that it was MacPherson that described
a symplectic toric manifold, as a topological space with a torus action,
as obtained by taking a polytope times a torus
and collapsing along the facets;
this was probably in the early 1990s and maybe earlier.

\item
A nice diagram, leftover from a too-complicated old proof 
that $c|_{U_M} \colon U_M \to (U_M)_\cut$ is proper
and that $\bfc \colon M \to \Mcut$ is proper:
$$ \xymatrix{
 U_M \ar[r]^{c} \ar@/_1pc/_{\pi}[rr] & (U_M)_\cut \ar[r]^{q} & U_M/S^1 ,
}$$

\item
Wording that might be useful in \S\ref{sec:intro}: \\
The smooth manifold structure on $\Mcut$ is unique
because the required properties 
determine which real valued functions are smooth 
which, in turn, determines the smooth manifold structure.

\item
For the notion of a manifolds with corners, 
J\"anich \cite{janich} refers to A.\ Douady 
(``Varietes a bords anguleux et voisinages tubulaires'', 1961/62).
Janich defines ``manifold with faces''.

\end{itemize}

\section{Leftovers on simultaneous cutting}

We now have a category.   Its objects are manifolds with faces,
equipped with free circle actions on neighbourhoods of the facets.  
Its morphisms are \ynote{$\ldots$}.
As in Section \ref{sec:functoriality}, the cutting construction
gives a functor from this category to the category whose objects
are smooth manifolds and whose morphisms are \ynote{$\ldots$}.
In particular, if $G$ is a Lie group,
a smooth $G$ action on $M$ that commutes with the circle actions
descends to a smooth $G$ action on~$\Mcut$.

As in Section \ref{sec:diff forms}, 
let $\beta$ be an invariant differential form on $M$ 
whose pullback to each stratum $\intY$ is basic with respect to $T_Y$,
thus this pullback comes from a differential form $\beta_Y$ on $X_Y$.
Then $\beta$ descends to a differential form $\beta_\cut$ on $X$,
whose pullback to each stratum $X_Y$ is exactly $\beta_Y$.
\ $\beta$ is closed on $M$ if and only if $\beta_\cut$
is closed on $\Mcut$.
For two such differential forms $\beta$ and $\beta'$,
we have $(d\beta)_\cut = d(\beta_\cut)$
and $(\beta \wedge \beta')_\cut = \beta_\cut \wedge \beta'_\cut$.
For each point $x$ of $M$,
the form $\beta$ is nonzero at $x$ if and only if the form $c(\beta)$
is nonzero at $c(x)$.

As in Section \ref{sec:symplectic cutting}, 
let $\omega$ be an invariant closed two-form $\omega$ on $M$.
Suppose that each facet $F$ is a level set of a momentum map
for the corresponding circle action.
Then the pullback of $\omega$ to each $\intY$ is basic with respect
to $T_Y$, and we get a closed two-form $\omega_\cut$ on $\Mcut$.
Moreover, $\omega_\cut$ is symplectic on $\Mcut$
if and only if $\omega$ is symplectic on $M$.
Also, for a Lie group $G$, a $G$ action on $M$ that commutes
with the circle actions near the facets of $M$ 
and preserves their momentum maps
descends to a smooth $G$ action on $\Mcut$,
and a momentum map for this $G$ action on $M$
descends to a momentum map for the $G$ action on $\Mcut$.

As in Section \ref{sec:contact cutting}, 
let $\xi$ be a codimension one distribution on $M$
that contains the tangents to the $S^1$-orbits along each face $F$
and that is $(S^1)_F$-invariant near each face $F$.
Then $\xi$ descends to a codimension one
distribution $\xi_\cut$ on $\Mcut$.
Moreover, $\xi$ is contact if and only if $\xi_\cut$ is contact.
Finally, as before, a smooth action of a Lie group $G$ on $M$
that commutes with the circle actions and preserves $\xi$
descends to a smooth action of $G$ on $\Mcut$ that preserves $\xi_\cut$.

\section{The smooth cutting construction yields a manifold -- earlier version}

Let $X$ be a topological space.
A collection of subsets of $X$ is an \textbf{open cover} of $X$
if the subsets are open and their union is $X$.
A collection of subsets of $X$ is \textbf{locally finite}
if each point in $X$ has a neighbourhood 
that meets only finitely many subsets in the collection.
The \textbf{carrier} of a function $f \colon X \to \R$
is the set of points where $f$ does not vanish;
the \textbf{support} $\supp f$ of $f$ is the closure of its carrier.
A \textbf{partition of unity} on $X$ is a collection of functions 
$\rho_i \colon X \to \Rplus$
whose set supports $\supp \rho_i$ is locally finite 
and such that the sum $\sum_i \rho_i$ (which is hence well defined)
is everywhere equal to one.
The partition of unity is \textbf{subordinate} to a cover $\frakU$ of $X$
if for each $i$ there exists $U \in \frakU$
such that $\supp \rho_i \subset U$.

Let $X$ be a topological space,
and let $\calF$ be a $C^\infty$ ring of real-valued functions on $X$
that are continuous.
The space $X$ is $\mathbf{\calF}$\textbf{--regular}
if for each closed subset $C$ and point $x \in X \ssminus C$
there exists a function $f \in \calF$
such that $f(x) \neq 0$ and such that $f$ vanishes on a neighbourhood of $C$. 
The space $X$ is $\mathbf{\calF}$\textbf{--paracompact}
if for each open covering $\frakU$ of $X$
there exists a partition of unity $\rho_i$ subordinate to $\frakU$
whose elements $\rho_i$ are all in $\calF$.
The \textbf{initial topology} determined by $\calF$
is the smallest topology on $X$
for which the functions in $\calF$ are continuous.

\begin{Proposition} \labell{prop:F paracompact}
Let $X$ be a topological space 
that is Hausdorff, second countable, and locally compact.
Let $\calF$ be a $C^\infty$ ring of real valued functions on $X$ 
that are continuous.
Suppose that $X$ is $\calF$-regular \ynote{define}.
Then 
\begin{enumerate}
\item
The initial topology determined by $\calF$
coincides with the given topology on $X$. 
\item
$X$ is $\calF$-paracompact.
\end{enumerate}
\end{Proposition}

\begin{proof}
Because the functions in $\calF$ are continuous,
the initial topology is contained in the given topology.

Being $\calF$-regular implies that the given topology
is contained in the initial topology.

Because $X$ is Hausdorff, second countable, and locally compact,
$X$ admits an exhaustion by compact sets:
there exist compact subsets $K_i$ of $\Mcut$, for $i \in \N$,
such that $\Mcut = \cup_i K_i$
and such that $K_i \subset \text{int}(K_{i+1})$ for all $i$.
See, e.g., \ynote{John Lee intro to topological manifolds Proposition 4.75}.

The remaining argument is a minor adaptation of 
\ynote{John Lee intro to top.mflds Theorem 4.77}.
Let $\frakU$ be an open covering of $\Mcut$.
Let $(K_j)_{j=1}^\infty$ be an exhaustion of $X$ by compact sets.
For each $j$, let $A_j = K_{j+1} \ssminus \int K_j$
and $W_j = \int K_{j+2} \ssminus K_{j-1}$,  
where we set $K_j=\emptyset$ if $j<1$.
Then $A_j$ is a compact subset contained in the open subset $W_j$.
For each $x \in A_j$, choose $U_x \in \frakU$ containing $x$.
Because $X$ is $\calF$--regular,
there exists a non-negative function $f_x \colon X \to \R$ in $\calF$
such that $f_x(x) > 0$ and such that $\supp f_x \subset U_x \cap W_j$.
Because $A_j$ is compact, 
there exists a finite collection of points 
$x_1^{(j)} , \ldots , x_{n_j}^{(j)}$ in $A_j$
such that the set of carriers $\{ f_{x_i^{(j)}} \neq 0 \}$ covers $A_j$.
consider the collection of functions $f_{x_i^{(j)}}$, 
for $j \in \N$ and $i \in \{ 1,\ldots,n_j \}$;
relabel it as $f_i$ for $i \in \N$.
Then the set of supports $\supp f_i$ is locally finite 
and the set of carriers $\{ f_i > 0 \}$ is an open cover of $X$.
The sum $\sum f_i$ is then well defined and everywhere positive.
\ynote{NOPE!  Need locality.}
\end{proof}

\end{document}